\documentclass{amsart}
\usepackage[english]{babel}
\usepackage[utf8]{inputenc}
\usepackage[T1]{fontenc}
\usepackage{tikz-cd}
\usepackage{amsmath}
\usepackage{mathtools}
\usepackage{amsfonts}
\usepackage{stmaryrd}
\usepackage{amssymb}
\usepackage{graphicx}
\usepackage[all]{xy}
\usepackage{pdfsync}
\usepackage{enumerate}
\usepackage[colorlinks]{hyperref}
\hypersetup{urlcolor = {black}, linkcolor = {black}, citecolor = {black}}
\usepackage{multirow} % Required for multirows
\usepackage{tensor} %required for double fiber products
\usepackage{xcolor}
\usepackage{enumitem}    

\newtheorem{theorem}{Theorem}[section]

\newtheorem{proposition}[theorem]{Proposition}
\newtheorem{corollary}[theorem]{Corollary}
\newtheorem{lemma}[theorem]{Lemma}
\theoremstyle{definition}
\newtheorem{definition}[theorem]{Definition}
\newtheorem{example}[theorem]{Example}
\newtheorem{examples}[theorem]{Examples}

\newtheorem{remark}[theorem]{Remark}
\renewcommand\indexname{\sc \indexname}

\newenvironment{proof2}[1][Proof]{\noindent\textbf{#1\;}}{\ \hfill$\square$}
\renewcommand\indexname{\sc \indexname}

\def\arrows{{\rightrightarrows}}
\def\e{{\varepsilon}}
\def\l{{\lambda}}
\def\i{{\iota}}
\def\t{{\tau}}
\def\calF{{\mathcal{F}}}
\def\calG{{\mathcal{G}}}

\def\G{{\mathcal{G}}}
\def\calL{{\mathcal{L}}}    
\def\calM{{\mathcal{M}}}

\tikzset{commutative diagrams/.cd,
mysymbol/.style={start anchor=center,end anchor=center,draw=none} }

\newcommand\MySymb[2][?]{\arrow[mysymbol]{#2}[description]{#1}}

\DeclareFontFamily{U}{MnSymbolC}{}%%Para el símbolo de producto interior/contracción de formas por campos de vectores.
\DeclareSymbolFont{MnSyC}{U}{MnSymbolC}{m}{n}
\DeclareFontShape{U}{MnSymbolC}{m}{n}{
    <-6>  MnSymbolC5
   <6-7>  MnSymbolC6
   <7-8>  MnSymbolC7
   <8-9>  MnSymbolC8
   <9-10> MnSymbolC9
  <10-12> MnSymbolC10
  <12->   MnSymbolC12}{}
\DeclareMathSymbol{\intprod}{\mathbin}{MnSyC}{'270}

\begin{document}

\title{Deformations of symplectic groupoids}

% author one information
\author{Cristian Camilo C\'ardenas}
\address{Universidade Federal Fluminense, Instituto de Matemática e Estatística,Rua Prof. Marcos Waldemar de Freitas Reis, S/n, 24210-201, Niterói, RJ, Brazil}
\email{ccardenascrist@gmail.com}

% author two information
\author{Jo\~ao Nuno Mestre}
\address{University of Coimbra, CMUC, Department of Mathematics, Apartado 3008,
EC Santa Cruz,
3001-501 Coimbra, Portugal}
\email{jnmestre@gmail.com}

% author three information
\author{Ivan Struchiner}
\address{Universidade de S\~ao Paulo, Instituto de Matem\'atica e Estat\'istica, Rua do Mat\~ao 1010, 05508-090, S\~ao Paulo, SP, Brazil}
\email{ivanstru@ime.usp.br}

\begin{abstract}

We describe the deformation cohomology of a symplectic groupoid, and use it to study deformations via Moser path methods, proving a symplectic groupoid version of the Moser Theorem.

Our construction uses the deformation cohomologies of Lie groupoids and of multiplicative forms, and we prove that in the symplectic case, deformation cohomology of both the underlying groupoid and of the symplectic groupoid have de Rham models in terms of differential forms. 

We use the de Rham model, which is intimately connected to the Bott-Shulman-Stasheff double complex, to compute deformation cohomology in several examples. We compute it for proper symplectic groupoids using vanishing results; alternatively, for groupoids satisfying homological 2-connectedness conditions we compute it using a simple spectral sequence.

Finally, without making assumptions on the topology, we constructed a map relating differentiable and deformation cohomology of the underlying Lie groupoid of a symplectic groupoid, and related it to its Lie algebroid counterpart via van Est maps.  
\end{abstract}

\maketitle

\setcounter{tocdepth}{1} %oculta subsecciones
\tableofcontents

\section{Introduction}
In this paper we study deformations of a symplectic groupoid $(\G,\omega)$, constructing the deformation cohomology $H_\mathrm{def}^\bullet(\G,\omega)$ controlling them. Any deformation of $(\G,\omega)$ gives rise to a deformation class $[\eta]$ in the cohomology, which can be used in Moser path methods. We carry out computations for the cohomology in examples, comparing it with cohomologies associated to related objects.

Symplectic groupoids were introduced by Karas\"ev \cite{Karasev}, Weinstein \cite{alan-GS} and Zakrzewski \cite{Zakrzewski}, motivated by the search for a quantization procedure for Poisson structures. They are first of all Lie groupoids, objects introduced by Ehresmann in the 1950's \cite{Ehresmann}, which generalize Lie groups and which, among other uses, permit the unified study of differential geometric objects such as foliations or Lie group actions. A Lie groupoid equipped with a compatible, i.e., \textit{multiplicative} symplectic structure, becomes well adapted to the study of Poisson structures. Poisson structures are Lie algebra structures on the algebra of smooth functions on a manifold, which in addition are biderivations of the associative algebra structure. Although these Lie algebras are infinite dimensional, they are on the infinitesimal side of a rich Lie theory, with symplectic groupoids (which are finite dimensional!) playing the part of the corresponding global object.

In analogy with the relation between Lie groups and Lie algebras, any symplectic groupoid $\G\rightrightarrows M$ induces a Poisson structure $\pi$ on its space of units $M$, such that the Lie algebroid of $\G$ is naturally isomorphic to the cotangent algebroid $(T^*M)_\pi$ associated to the Poisson manifold $(M,\pi)$.  Not all Poisson structures arise in this way, but for the ones which do, called \textit{integrable}, it was shown in in \cite{MX3} that there is always a natural symplectic groupoid structure on the source $1$-connected integration $\G$ of $(T^*M)_\pi$. The symplectic groupoid structure on $\G$ was also obtained by an infinite dimensional reduction procedure in \cite{Cattaneo_Felder}, and the characterization of which Poisson manifolds are integrable was finally settled in \cite{Integrability_Poisson}.

The survey \cite{Alan-survey} gives an account of Poisson geometry at the early stages of development of symplectic groupoids, and already then they feature prominently in the results and the problems presented, some of them now settled, some still open.

First and foremost (from our specific viewpoint) symplectic groupoids are useful in the study of the Poisson geometry of the base, both in its local and global properties.
But we should mention that symplectic groupoids have also feature in a variety of other treatments, for example in geometric quantization \cite{Hawkins}, deformation quantization (cf.\ e.g.\ \cite{Cabrera_formal, Cattaneo_Felder}), the study of the geometry of varieties appearing in Poisson-Lie theory (for example \cite{Lu_Mouquin}) and of moduli spaces of flat connections (e.g.\ \cite{Li-Bland-Severa2, Severa-moduli}).

Particulary motivating for this work, was the role played by symplectic groupoids, both in the study of local and of global Poisson geometry,  via Moser path arguments, also known as ``the Moser trick''. We mention how it came to influence our work. Starting with the famous Moser Theorem, the Moser trick has been used to prove many rigidity and normal form results in symplectic geometry, and related topics (see e.g.\ \cite{ACS}), by reducing deformation problems to cohomological conditions. It was first used in the context of Lie groupoids, to our knowledge, by Weinstein \cite{Alan_linearization}, in the pursuit of linearization theorems for proper Lie groupoids. Next, in the geometric proof of Conn's Theorem of \cite{Geometric_Conn}, a \textit{Poisson} version of the Moser Theorem was used, together with symplectic groupoids that solved the cohomological conditions. Further, similar techniques were used around leaves of a Poisson manifold \cite{Marius_Ionut_leaves}. These techniques eventually led to the geometric proof of the linearization of proper groupoids around orbits \cite{CraiS}, and to the development of the deformation cohomology $H^\bullet_\mathrm{def}(\G)$ of an arbitrary Lie groupoid in \cite{CMS}.

 Recent developments related to deformation theory for Lie groupoids include a rigidity result for foliations \cite{Matias-Rui-foliations}, the study of deformations of vector bundles over Lie groupoids \cite{PierVitag}, and of morphisms of Lie groupoids \cite{CS}. Specific classes of deformations of Lie groupoids (deformations to the normal cone) are also of importance in noncommutative geometry (see e.g.\ \cite{DS1,DS2}).

Our work brings this new understanding on deformations of Lie groupoids full circle back to Poisson geometry, as exemplified by a symplectic groupoid version of the classical Moser Theorem (detailed further in Sections 7 and 8).

\begin{theorem}
Let $(\widetilde{\calG},\omega)$ be a deformation of a compact symplectic groupoid $(\calG,\omega_\calG)$ and let $[\eta]$ be the deformation class in $H^2_\mathrm{def}(\G, \omega)$ associated to it. Then, the deformation $(\widetilde{\calG},\omega)$ is trivial if and only if $[\eta]$ is exact.
\end{theorem}

This result also leads to local descriptions of moduli spaces of multiplicative symplectic forms on a compact Lie groupoid $\G$, given by an analogue of a Kodaira-Spencer map (See Theorem \ref{thm-moduli} for a more precise and detailed version).

\begin{theorem} Let $(\G,\omega)$ be a compact symplectic groupoid, and denote by $\mathcal{S}({\G})$ the space of multiplicative symplectic structures on $\G$. 

Then there is a neighbourhood $\mathcal{U}$ of $\omega$ in $\mathcal{S}({\G})$ and a neighbourhood $\mathcal{V}$ of $0$ in $H^2_\mathrm{def}(\G,\omega)$, such that there is an explicit 1:1 correspondence \[\kappa:\ \mathcal{U}\big/\sim\ \  \stackrel{\mathrm{1:1}}{\longrightarrow} \  \mathcal{V}\subset H^2_\mathrm{def}(\G,\omega),\] sending the equivalence class of $\omega$ to $0$. The equivalence relation is such that $\omega_1 \sim \omega_2$ if $\omega_1+\e(\omega_2-\omega_1)$ is a trivial family of multiplicative forms for all $\e$.
\end{theorem}

The key insight in order to construct the deformation cohomology $H^2_\mathrm{def}(\G, \omega)$, where the cohomological equations of the Moser Theorem naturally live, was that the pair of a Lie groupoid and an extra structure on it is more manageable when interpreted as a diagram. Then the new tools on deformations of morphisms of Lie groupoids from \cite{CS} could be used.

Moreover, in the specific case where the extra structure is symplectic, there is a description in terms of differential forms on the nerve $\G^{(\bullet)}$ of $\G$, of $H^\bullet_\mathrm{def}(\G)$, and of $H^\bullet_\mathrm{def}(\G, \omega)$. We call the latter the \textit{de Rham model for deformation cohomology} of the symplectic groupoid $(\G,\omega)$, and we summarize the results here (see Proposition \ref{prop:isomorph-nondegeneracy} and Theorem \ref{viewpoint2} for more detailed versions).

\begin{theorem} 
Let $(\calG,\omega)$ be a symplectic groupoid. Then $\omega^{b}$ induces a quasi-isomorphism  \[C^{\bullet}_\mathrm{def}(\calG)\sim \Omega^{1}(\calG^{(\bullet)}).\]

\ \

The deformation complex $C^{\bullet}_\mathrm{def}(\G,\omega)$ of the symplectic groupoid $(\G,\omega)$ is isomorphic to the mapping cone of the de Rham differential \[-d_{dR}:\Omega^{1}_\mathrm{def}(\G^{(\bullet)})\to\Omega_\mathrm{cl}^{2}(\G^{(\bullet)}).\]
\end{theorem}

In order to effectively use a symplectic groupoid to study the Poisson structure on the base, it is desirable to have a good description of the symplectic groupoid, or at least to know some of its properties. There are conditions to characterize whether a Poisson manifold is integrable \cite{Integrability_Poisson} and if it is, there is an infinite-dimensional reduction procedure for the construction of the source $1$-connected integration \cite{Cattaneo_Felder}. Nonetheless, constructing a concrete integration is often still challenging. 

Fortunately, symplectic groupoids for several classes of Poisson manifolds have been described, usually immediately bringing about new insight about the Poisson structures they integrate; we mention a few of them. Some of the simplest cases are symplectic manifolds, the zero Poisson structure and linear Poisson structures on vector spaces and vector bundles \cite{CosteDazordWeins.}. Even for the zero Poisson structure, symplectic groupoids may provide rich geometry in the form of integral affine structures, as shown recently in \cite{PMCT1,PMCT2}. Integrations for any Poisson-Lie group are described in \cite{Lu_thesis, Lu_Weinstein}, and these and other symplectic groupoids also have interpretations in terms of moduli spaces of flat connections, for example in \cite{Li-Bland-Severa2, Severa-moduli}. Integrations of Poisson fibrations are studied in \cite{Poisson_fibrations}; integrations of Log-symplectic manifolds are studied using blow-up and gluing techniques in \cite{Gualtieri-Li}; a description of the integration of neighborhoods of Poisson transversal submanifolds is given in \cite{Pedro-Ionut-PT2}; more examples appear in relation to the theory of Poisson-Lie groups - for example, symplectic leaves of double Bruhat cells are symplectic groupoids studied in \cite{Lu_Mouquin}; and this list is far from complete.

We carried out computations for the deformation cohomology in a few of these examples. Most computable examples $(\G,\omega)$ are typically of one of two kinds: either we assumed some sort of compactness condition, or vanishing of $H^1_\mathrm{dR}(\G)$ and $H^2_\mathrm{dR}(\G)$. Clearly we cannot ask for both  vanishing $H^2_\mathrm{dR}(\G)$ and compactness, because we deal with symplectic groupoids. In fact, even examples of compact symplectic groupoids with the more reasonable condition of having simply connected source fibres are very hard to find \cite{PMCT2, David_PMCT, Zwaan}.

Asking for the symplectic groupoid to be proper (a compactness-type condition), we obtained the following result (See Section 6 for more a detailed version of it), and applied it for linear Poisson and zero Poisson structures of proper type.

\begin{proposition}
         Let $(\G, \omega)$ be a proper symplectic groupoid, and let $\nu$ denote the normal distribution to the symplectic leaves. Then there is a 7-term exact sequence 
\begin{align*}
   0  & \to \Omega^1_\mathrm{cl}(M)^{\G-\mathrm{inv}} \to \Gamma(\nu^*)^{\G-\mathrm{inv}} \to \Omega^2_\mathrm{cl}(M)^{\G-\mathrm{inv}} \to H^{1}_\mathrm{def}(\G, \omega) \to \\
  &  \to \Gamma(\nu)^{\G-\mathrm{inv}}\to H_\mathrm{def}^{1}(\omega)\to H^{2}_\mathrm{def}(\G, \omega)\to 0
  \end{align*}       
  and $H^{k}_\mathrm{def}(\G,\omega)\cong H^{k-1}_\mathrm{def}(\omega)$ for $k> 2$.
  
In particular, expressed in terms of the de Rham model, \[H^{2}_\mathrm{def}(\G, \omega)\cong H_\mathrm{def}^{1}(\omega)/d(H^1_\mathrm{def}(\G)).\]
\end{proposition}

With the help of some simple spectral sequence arguments we also provide descriptions of the deformation cohomology of an arbitrary symplectic groupoid. Those descriptions, however, become much simpler in the case of proper symplectic groupoids, as detailed in Sections \ref{Sec-spectral-sequence-rows} and \ref{Sec-spectral-sequence-column}.

If on the other hand we ask for moderately strong topological conditions - vanishing of first and second de Rham cohomologies, we see a very close relation between deformation cohomology of $(\G,\omega)$ and Poisson cohomology $H_\pi(M)$ of the base. Poisson cohomology of $(M,\pi)$ is computed by the complex $\mathfrak{X}^\bullet(M)$ of multivector fields on $M$, with differential $d_\pi=[\pi,\cdot]_{SN}$, i.e., given by the Schouten-Nijenhuis bracket with $\pi$. It is the cohomology controlling deformations of the Poisson structure $\pi$, and in general it is hard to compute. We refer to \cite{Florian} for an account of techniques that can be used to compute it in different classes of Poisson manifolds, as well as for the very recent developments in \textit{loc. cit.}

Call a Lie groupoid $\G$ homologically $2$-connected if all the spaces $\G^{(n)}$ of its nerve, as well as its source-fibres, are homologically $2$-connected.
In this situation we obtained the following result (Section 9 contains precise and improved versions).

\begin{proposition} Let $(\G,\omega) \rightrightarrows (M,\pi)$ be a homologically $2$-connected symplectic groupoid.
Then \[H_\mathrm{def}^0(\G,\omega)\cong H^0_\pi(M)/\mathbb{R}, \ \ \ H_\mathrm{def}^1(\G,\omega)\cong H^1_\pi(M)\ \ \ H_\mathrm{def}^2(\G,\omega)\cong H^2_\pi(M),\] and there is an injective map \[H_\mathrm{def}^{3}(\G,\omega)\to H^{3}_\pi(M).\]
\end{proposition}

Although the conditions do sound restrictive, we focus in Section 9 on some situations in which they are satisfied, for example in the context of linear Poisson structures, Poisson-Lie groups, and cotangent VB-groupoids.

As a general result without further assumptions on topology of $(\G,\omega)$ other than connectivity of the source fibres, we construct a map $i_\G: H^\bullet(\G)\to H^\bullet_\mathrm{def}(\G)$ as a global counterpart to the map $i: H^\bullet_\pi(M)\to H^\bullet_\mathrm{def}(A)$ of \cite{CM} and prove the following (more details, and the statement at the level of cochains in Theorem \ref{global of i}).

\begin{theorem}
Let $(\calG,\omega)$ be a $s$-connected symplectic groupoid. Then  $i_{\calG}$ together with $i$ and the van Est maps for differentiable and deformation cohomology form the commutative diagram.
\begin{equation*}
%\[
  \begin{tikzcd}[column sep=4em, row sep=10ex]
    H^{k}(\calG) \MySymb[\circlearrowright]{dr} \arrow{d}[swap]{VE} \arrow{r}{i_{\calG}} & H^{k}_\mathrm{def}(\calG) \arrow{d}{VE_\mathrm{def}}\\
    H^{k}(T^{*}M) \arrow{r}{i} & H^{k}_\mathrm{def}(T^{*}M).
  \end{tikzcd}
%\]
\end{equation*}
\end{theorem}

We would like to make some final remarks.
We believe that the description of the deformation complex of a symplectic groupoid is important by itself, because this work should serve as prototype. It opens the road to similar studies for other structures on groupoids, which should shed light on deformation theoretic properties of objects described by them. We expect that similar techniques can be used, but revealing different phenomena, for multiplicative contact, Poisson, complex structures, foliations, etc.

The theory developed in this paper should have a version for local and formal symplectic groupoids, likely using the version of the Bott-Shulman-Stasheff for local Lie groupoids \cite{LBM}. In this direction, it would be interesting to study the classes of deformations of formal symplectic groupoids \cite{CDF} induced by star products arising from quantization, or pairs of star products, as in the works of \cite{Cabrera_formal} and \cite{Karabegov1, Karabegov2}. We expect that local and formal versions automatically satisfy many of the topological assumptions we used in computations, bringing the theory closer to the deformations of the underlying Poisson structures.

Finally, all the work developed in this paper is intimately connected with the Bott-Shulman-Stasheff double complex \cite{BSS}, therefore lending itself to generalizations towards twisted presymplectic, and higher (shifted) (pre)symplectic groupoids.

\subsection*{Outline of the paper}

In Section 2 we recall needed background material about Lie groupoid cohomology and deformations, VB-groupoids, multiplicative forms and symplectic groupoids, and Poisson manifolds.

In Section 3 we introduce deformations and equivalence of deformations of symplectic groupoids, and provide examples of trivial and non-trivial deformations.

In Section 4 we explain how the symplectic form of a symplectic groupoid $(\G,\omega)$ permits a description of deformation cohomology of $\G$ in terms of differential forms.

In Section 5 we come to the central object of this paper, the deformation cohomology $H_\mathrm{def}(\G, \omega)$ of a symplectic groupoid. We explain how to construct it out of the deformation cohomologies of the $\G$ and of $\omega$. We also prove that there is a de Rham model for it, in which cochains are differential forms, used for the computations in the rest of the paper.

In Section 6 we carry out the first computations of deformation cohomology, for symplectic manifolds, and proper symplectic groupoids; we then specialize to zero Poisson structures of proper type and linear Poisson structures of proper type. Using simple spectral sequence arguments we present some explicit computations for general symplectic groupoids.

In sections 7 and 8 we use the deformation cohomology to prove symplectic groupoid versions of the Moser theorem, first for specific (source-constant), and then for general deformations. We also study consequences moduli spaces of multiplicative symplectic forms.

In Section 9 we use another spectral sequence to carry out more computations of deformation cohomology, this time for symplectic groupoids for which the nerve has vanishing first and second de Rham cohomology. The relation with Poisson cohomology is very close in these cases; we include computations for linear Poisson structures, the zero Poisson structure, Poisson-Lie groups, and cotangent groupoids.

in Section 10 we study, now without topological assumptions on $(\G, \omega)$, relations between the differentiable and deformation cohomology of $\G$, the Poisson cohomology of the base, and the deformation cohomology of the corresponding algebroid. We do so by providing a map $i_\G: H^\bullet(\G)\to H^\bullet_\mathrm{def}(\G)$ and proving that it is related by van Est maps to the map $i: H^\bullet_\pi(M)\to H^\bullet_\mathrm{def}(A)$ of \cite{CM}.

Finally, we collect in an appendix the needed material on double structures, used to prove the Theorem of Section 10. 

\subsection*{Notation and conventions}

We denote a Lie groupoid $\G$ over a base $M$ by $\G\rightrightarrows M$, usually having source and target maps denoted by $s$ and $t$, multiplication denoted by $m$, inversion map $i$ and unit map $u$.

We say that a Lie groupoid $\G$ is $s$-connected (respectively $s$-$k$-connected) if all fibres of its source map are connected (respectively $k$-connected).

We denote the tangent map of a differentiable map $f:M\to N$ either by $df$ or $Tf$. 
To avoid confusion with other cohomologies, we denote de Rham cohomology of a manifold $M$ by $H^\bullet_{dR}(M)$.

\subsection*{Acknowledgements}

The authors would like to thank Alejandro Cabrera, Rui Loja Fernandes, Ioan M\u{a}rcu\c{t}, Cristi\'an Ortiz, and Luca Vitagliano for useful discussions related to the work of this paper.

C.C.C. was partially supported by a PNPD postdoctoral fellowship at UFF, and would also like to thank the University of S\~ao Paulo, where this work was initiated.
J.N.M. was partially supported by the Centre for Mathematics of the University of Coimbra - UIDB/00324/2020, funded by the Portuguese Government through FCT/MCTES, and by grant 2019/20789-4, S\~ao Paulo Research Foundation (FAPESP). I.S. was partially supported by FAPESP 2015/22059-2 and CNPq 307131/2016-5.

\section{Background}

We start by recalling some background on Lie groupoids: their cohomology and deformations, VB-groupoids, and symplectic groupoids. For a general introduction to Lie groupoids we refer to \cite{Lectures,M,MM}.

\subsection{Deformations of Lie groupoids} 
More details on the notions in this section can be found in \cite{CMS}, where  they are discussed in detail.
\begin{definition} A \textbf{family of Lie groupoids} over a manifold $B$ consists of a Lie groupoid $\widetilde{\calG}\arrows\tilde{M}$, together with a submersion $\textbf{p}:\tilde{M}\to B$, such that $\textbf{p}\circ \tilde{s}=\textbf{p}\circ \tilde{t}$.
\end{definition}

In other words, a family of Lie groupoids over $B$ consists of a submersive Lie groupoid map $\textbf{p}:(\widetilde{\calG}\arrows \tilde{M})\to(B \arrows B)$, from $\widetilde{\calG}$ to the unit groupoid of $B$. We denote by $\calG_b\arrows M_b$ the fibre of $\textbf{p}$ over $b\in B$, i.e.\ $(\textbf{p}\circ \tilde{s})^{-1}(b)\arrows \textbf{p}^{-1}(b)$.

\begin{definition}An \textbf{equivalence between families} $\widetilde{\calG}$ and $\widetilde{\calG}'$ over $B$ is given by a Lie groupoid isomorphism $F:\widetilde{\calG}\to\tilde{\calG'}$ such that $\textbf{p}'\circ F=\textbf{p}$.
A family is said to be \textbf{trivial} if it is equivalent to the product family given by  \[(\widetilde{\calG}\arrows \tilde{M})\times(B \arrows B)\stackrel{\textbf{p}=pr_B}{\longrightarrow}(B \arrows B).\]
\end{definition}

\begin{definition}\label{Deform.grpds}
A \textbf{deformation of a Lie groupoid} $\calG$ is a family of Lie groupoids $\widetilde{\calG}\arrows\tilde{M}$ over an open interval $I$ containing zero, such that $\calG_0$ is equal to $\calG$.

An \textbf{equivalence between deformations} $\widetilde{\calG}$ and $\widetilde{\calG}'$ of $\calG$ is given by an equivalence of families $F:\widetilde{\calG}_{|_J}\to\widetilde{\calG}_{|_J}'$ such that $F_0:=F_{|_{\calG_0}}=id_{\calG_0}$, where $J\subset I\cap I'$ is an open interval containing zero, and $\widetilde{\calG}_{|_J}$ denotes the restriction of the family $\widetilde{\calG}$ to the interval $J$.
\end{definition}

\begin{definition}
A deformation $\widetilde{\G}$ of a Lie groupoid $\calG$ for which $\widetilde{\G}=\G\times I$ as smooth manifolds is called a \textbf{strict} deformation.

It is called an \textbf{$s$-constant} deformation if additionally the source map is left constant, $i.e.$, $s_\varepsilon=s_0$ for all $\e\in I$.
\end{definition}

\begin{remark} A deformation of a Lie groupoid $\G$ determines a collection $\{\G_\epsilon\}$ of Lie groupoids parametrized by $\epsilon\in I$ (the fibres of the family), varying smoothly with respect to $\epsilon$ in the sense that they fit as fibres of a submersion. If the deformation is strict then the $\G_\e$ are equal as manifolds, and only the structural maps change.

Accordingly, we often use denote a deformation of $\G$ simply by $\{\G_\epsilon\}$, keeping in mind that there is a specified family that the members of this family fit into. 

Similarly, we often denote an equivalence $F:\widetilde{\G}\to\widetilde{\G}'$ of deformations by the associated family of maps $F_\epsilon:\G_\epsilon\to\G_\epsilon '$.
\end{remark}

\subsection{The nerve of a Lie groupoid and differentiable cohomology}\label{subsec-Nerve}

Let $\G\arrows M$ be a Lie groupoid and denote its space of strings of $k$ composable arrows by 
\[\calG^{(k)} =\{(g_1,\ldots,g_k) : s(g_i) = t(g_{i+1}) \text{ for all } 1\leq i\leq k-1\},\] 
and let $\G^{(0)}=M$. \textbf{The nerve of $\G$} is the simplicial manifold $\G^{(\bullet)}$ for which the space of $k$-simplices is $\calG^{(k)}$, the face maps are $d_i:\G^{(n)}\to\G^{(n-1)}$ given by 
\[d_i(g_1,\ldots,g_n) = 
\begin{cases}
(g_2,\ldots,g_n) & \text{ if }  i=0\\
(g_1,\ldots,g_i g_{i+1}, \ldots,g_n)& \text{ if }   1\leq i\leq n-1\\
(g_1,\ldots,g_{n-1})& \text{ if } i=n,\end{cases}
\] and degeneracies $s_i:\G^{(n)}\to\G^{(n+1)}$ are given by inserting the appropriate unit in the $i$-th position, $s_i(g_1,\ldots,g_n)=(g_1,\ldots,g_{i-1},1_,g_i,\ldots, g_n)$, for $1\leq i\leq n+1$.

Smooth functions on the nerve form a cosimplicial space, out of which we can construct the cochain complex $C^\bullet_\mathrm{diff}(\G)$ (also denoted just by $C^\bullet(\G)$) computing the \textbf{differentiable cohomology} of $\G$, denoted by $H^\bullet_\mathrm{diff}(\G)$. Explicitly, its spaces of cochains are $C^k_\mathrm{diff}(\G)=C^\infty(\G^{(k)})$, and the differentials are given by  \[\delta=\sum_{i=0}^{k+1} (-1)^i d_i^*: C^k_\mathrm{diff}(\G)\to C^{k+1}_\mathrm{diff}(\G).\] The subcomplex of \textbf{normalized} cochains consists of those cochains $c\in C^\bullet_\mathrm{diff}(\G)$ which vanish on all degeneracies, i.e., $s_i^*c=0$ for all $i$; it is quasi-isomorphic to $C^\bullet_\mathrm{diff}(\G)$.

More generally, $k$-forms on the nerve of $\G$ together with a differential $\delta$ defined by the same formula form the cochain complex $\Omega^k(\G^{\bullet})$. These complexes form the lines of the \textbf{Bott-Shulman-Stasheff double complex} $\left(\Omega^{\bullet}(\calG^{(\bullet)}),\delta, d\right)$ \cite{BSS}; the horizontal differentials $\delta:\Omega^{q}(\calG^{(\bullet)})\longrightarrow\Omega^{q}(\calG^{(\bullet+1)})$ are determined by the simplicial structure of the nerve of $\calG$, as described above, and vertical differentials $d:\Omega^{\bullet}(\calG^{(p)})\longrightarrow\Omega^{\bullet+1}(\calG^{(p)})$ are given by the de Rham differential of forms.

\subsection{The deformation complex of a Lie groupoid}

We recall the definition from \cite{CMS} of the deformation complex of a Lie groupoid $\calG$.
To that end, let $\bar{m}$ denote the division map  of $\calG$, defined as
\[\bar{m}(p,q) = pq^{-1}, \text{ for all } p,q \in \calG\ \text{ such that }\ s(p) = s(q).\]

\begin{definition}
\label{dfn-def-coh}
The deformation complex $(C_\mathrm{def}^\bullet(\mathcal{G}),\delta)$ of a Lie groupoid $\calG$, whose cohomology is denoted $H_\mathrm{def}^\bullet(\calG)$, is defined as follows. For $k\geq 1$, the $k$-cochains $c\in C_\mathrm{def}^k(\mathcal{G})$ are the smooth maps 
\[ c:\calG^{(k)}\longrightarrow T\calG, \ \ (g_1,\ldots,g_k)\mapsto c(g_1,\ldots,g_k)\in T_{g_1}\calG,\] 
which are $s$-projectable in the sense that 
\[ ds\circ c(g_1,g_2,\ldots,g_k)\]
does not depend on $g_1$.

 The differential of $c\in C_\mathrm{def}^k(\mathcal{G})$ is defined by
\begin{align*}(\delta c)(g_1,\ldots,g_{k+1})  = &- d\bar{m}(c(g_1g_2,\ldots,g_{k+1}),c(g_2,\ldots,g_{k+1}))  \\
&+\sum_{i=2}^k(-1)^{i} c(g_1,\ldots, g_ig_{i+1},\ldots, g_{k+1}) +(-1)^{k+1}c(g_1,\ldots, g_k).
\end{align*}
For $k= 0$, $C_\mathrm{def}^0(\mathcal{G}): = \Gamma(A)$ and the differential of $\alpha\in \Gamma(A)$ is defined by
\[ \delta(\alpha)=\overrightarrow{\alpha}+\overleftarrow{\alpha}\in C_\mathrm{def}^1(\mathcal{G}),\]
where $\overrightarrow{\alpha}$ and $\overleftarrow{\alpha}$ are the right-invariant and the left-invariant vector fields on $\G$ induced by $\alpha$. \end{definition}

Note that, to interpret the case $k= 0$, one way to think of a section of $A$ is as a map $c:\calG^{(0)}=M\longrightarrow T\calG$, with $c(1_x)=c_x \in T_{1_x}\calG$, such that $ds\circ c_x=0_x$.

\subsection{The deformation class of a deformation of Lie groupoids}\label{section-def-class-Lie-gpd}

We recall the definition from \cite{CMS} of the deformation class of a deformation $\widetilde{\G}\arrows M\stackrel{\textbf{p}}{\to} I$ of a Lie groupoid $\G$.

If $\widetilde{\G}$ is an $s$-constant deformation, then the expression 

\begin{equation}\label{def-s-constant-cocycle}
\xi(g,h):=\left.\frac{d}{d\e}\right|_{\e=0}\bar{m}_\e(gh,h)
\end{equation} 
defines a $2$-cochain in deformation cohomology. It is a cocycle and its cohomology class only depends on the equivalence class of the deformation \cite[Lemma 5.3]{CMS}. Moreover, If we let $\tilde{\xi}=-\delta(\frac{\partial}{\partial\e})\in C_\mathrm{def}^1(\widetilde{\G})$, then $\xi=\tilde{\xi}_{|\G_0}$ \cite[Prop. 5.7]{CMS}. 

The expression (\ref{def-s-constant-cocycle}) no longer makes sense if the deformation $\widetilde{\G}$ is not $s$-constant, but defining $\xi$ in terms of $\tilde{\xi}$ is still possible. For that, we need an appropriate analogue for the vector field $\frac{\partial}{\partial\e}$ on $\widetilde{\G}$.

\begin{definition}Let $\widetilde{\G}\arrows M\stackrel{\textbf{p}}{\to} I$ be a deformation of a Lie groupoid $\G$. A transverse vector field for $\widetilde{\G}$ is a vector field $\tilde{X}\in\mathfrak{X}(\widetilde{\G})$ which is $s$-projectable to a vector field $\tilde{V}\in\mathfrak{X}(\tilde{M})$, which is in turn $\textbf{p}$-projectable to $\frac{\partial}{\partial\e}\in\mathfrak{X}(I)$.
\end{definition}

\begin{proposition}[\cite{CMS}, Prop. 5.12]
Let $\widetilde{\G}$ be a deformation of $\G$. Then
\begin{enumerate}
\item There exist transverse vector fields for $\widetilde{\G}$;
\item For  $\tilde{X}\in\mathfrak{X}(\widetilde{\G})$ transverse, the restriction of $\delta\tilde{X}$ to $\G_0$ defines a cocycle $\xi_0\in C^2_\mathrm{def}(\G)$;
\item The cohomology class of $\xi_0$ does not depend on the choice of $\tilde{X}$.
\end{enumerate}
\end{proposition}

\begin{definition} The resulting cohomology class $[\xi_0]\in H^2_\mathrm{def}(\G)$ is called the deformation class associated to the deformation $\widetilde{\G}$ of $\G$.
\end{definition}

Again, the deformation class is invariant under equivalence of deformations, and as remarked in \cite{PierVitag}, this follows already from the independence of the chosen transverse vector field.

Note that $\widetilde{\G}$ can also be seen as a deformation of its other fibres $\G_\e$, and for such, the deformation class will be $[\xi_\e]=[\delta\tilde{X}_{|\G_\e}]$.

\subsection{Deformations of Lie groupoid morphisms and diagrams}

We recall some background on deformations of Lie groupoid morphisms, which are studied in detail in \cite{CS}. 

\begin{definition}Let $F:\G\to\mathcal{H}$ be a morphism of Lie groupoids. A \textbf{deformation of the morphism $F$} is a smooth map $\tilde{F}: \G\times I \to \mathcal{H}$ such that $\tilde{F}(\cdot, 0)=F$ and for each $\epsilon\in I$, the map $F_\epsilon: \mathcal{G}\to \mathcal{H}$ is a morphism.
\end{definition}

More generally, and this will be useful in our study, one can at once deform the whole diagram $F: \G \to \mathcal{H}$, which for simplicity we denote by $F$ as well. 

\begin{definition}A \textbf{deformation of the diagram $F$} consists of a triple $(\widetilde{\G},\tilde{F},\tilde{\mathcal{H}})$, where $\widetilde{\G}$ and $\tilde{\mathcal{H}}$ are deformations of $\G$ and $\mathcal{H}$, and the map $\tilde{F}:\widetilde{\G}\to \tilde{\mathcal{H}}$ is a Lie groupoid morphism sending $\G_\epsilon$ to $\mathcal{H}_\epsilon$, such that $F_0:=\tilde{F}_{|_{\G_0}} = F$.
\end{definition}

Note that the requirement that $\tilde{F}$ is a morphism implies in particular that each $F_\epsilon:=\tilde{F}_{|_{\G_\epsilon}}:\G_\epsilon\to \mathcal{H}_\epsilon$ is a Lie groupoid morphism as well.

For the purposes of this paper we will be interested in simultaneous deformations of the Lie groupoid $\G$ and the morphism $F$, or in other words, deformations of the diagram $F$ where $\mathcal{H}$ is kept fixed.

\subsection{The deformation complex of a Lie groupoid morphism}\label{def_complex_morphism}

The deformation complex $C^\bullet_\mathrm{def}(F)$ of a Lie groupoid morphism $F:\G\to \mathcal{H}$ was briefly described in \cite{CMS} and further explored in \cite{CS}. It has the following description, similar to that of $C^\bullet_\mathrm{def}(\G)$.
For $k\geq 1$, \[C_\mathrm{def}^k(F)=\left\{c:\G^{(k)}\to T\mathcal{H}\ \middle|\ c(g_1,\ldots,c_k)\in T_{F(g_1)}\mathcal{H}\ \text{and}\ c \ \text{is}\ s_\mathcal{H}-\text{projectable} \right\},\] where $c$ being 
$s_\mathcal{H}-$projectable means that $ds_\mathcal{H}\circ c(g_1, \ldots, g_k)$ is independent of $g_1$.

The differential of $c$ is defined by the same formula as the differential for $C^*_\mathrm{def}(\G)$ (see Definition \ref{dfn-def-coh}), except using the division $d\bar{m}_\mathcal{H}$ of $T\mathcal{H}$ instead of the division $d\bar{m}$ of $T\G$.

In degree 0, we have $C_\mathrm{def}^0(F)=\Gamma({F^*A_\mathcal{H}})$, and $\delta \alpha = \overrightarrow{\alpha}+\overleftarrow{\alpha}\in C_\mathrm{def}^1(F)$, where $\overrightarrow{\alpha}(g)=dR_{F(g)}(\alpha_{t(g)})$ and $\overleftarrow{\alpha}(g)=dL_{F(g)}(di (\alpha_{s(g)}))$.

There is a natural cochain complex map $F_*:C^\bullet_\mathrm{def}(\G)\to C^\bullet_\mathrm{def}(F)$, given by $F_*c(g_1,\ldots,g_k):=dF_{g_1}(c(g_1,\ldots,g_k))$.

Analogously, for example, to the cases of deformations of  morphisms of associative algebras \cite{GerSchack}, of Lie algebras \cite{yael}, and of deformations of holomorphic maps (see \cite{Iacono} and references therein), it is shown in \cite{CS} that simultaneous deformations of $\G$ and $F$ are controlled by the mapping cone complex of $F_*$, \[\text{Cone}(F_*)^n=C^{n}_\mathrm{def}(\G)\oplus C^{n-1}_\mathrm{def}(F),\ \ \delta(c,Y)=(\delta c, F_*c-\delta Y).\] 

\subsection{VB-groupoids}\label{VB-grpds}

VB-groupoids can be thought of as groupoid objects in the category of vector bundles. They provide alternative ways to look at the representation theory and the deformation theory of Lie groupoids (See \cite{CMS}, \cite{GrMe-grpds} and Proposition \ref{deformation and vb-complexes} below).
We review some basics on VB-groupoids which will be useful in the following sections. For more details we refer to \cite{M}, \cite{GrMe-grpds} and \cite{BCdH}.

\begin{definition}
 A \textbf{VB-groupoid} $(\Gamma, E, \calG, M)$ consists of two Lie groupoids together with two vector bundle structures as in the diagram
\begin{equation}\label{VB-grpd}
%\[
  \begin{tikzcd}[column sep=4em, row sep=10ex]
    \Gamma \arrow[yshift=2pt]{r}{\tilde{s}} \arrow[yshift=-2pt]{r}[swap]{\tilde{t}} \arrow{d}[swap]{\tilde{q}} & E \arrow{d}{q}\\
    \calG \arrow[yshift=2pt]{r}{s} \arrow[yshift=-2pt]{r}[swap]{t}    & M,
  \end{tikzcd}
%\]
\end{equation}
where $\tilde{q}$ and $q$ are vector bundles and $(\Gamma\rightrightarrows E)$, $(\calG\rightrightarrows M)$ are Lie groupoids such that the structure maps of the groupoid $\Gamma$ 
are vector bundle morphisms over the corresponding structure maps of the groupoid $\calG$.

A \textbf{morphism of VB-groupoids} \[(\phi_\Gamma,\phi_E,\phi_\calG,\phi_M):(\Gamma, E, \calG, M)\longrightarrow (\Gamma', E', \calG', M')\] is a Lie groupoid map $(\phi_\Gamma,\phi_E):\Gamma\rightrightarrows E\to\Gamma'\rightrightarrows E'$, such that $\phi_\Gamma$ and $\phi_E$ are vector bundle morphisms covering $\phi_\calG:\calG\longrightarrow \calG'$ and $\phi_M:M\longrightarrow M'$, respectively. Restricting $\phi_\Gamma$ to the zero section, $\phi_\calG$ turns out to be a Lie groupoid morphism.
\end{definition}

\begin{remark}
The requirement that $m_\Gamma:\Gamma^{(2)}\longrightarrow\Gamma$ is a vector bundle morphism, is taken with respect to the vector bundle structure of $\Gamma^{(2)}$ over $\calG^{(2)}$.
\end{remark}

\begin{example}(Tangent VB-groupoid)
 
Given a Lie groupoid $\calG\rightrightarrows M$ with source, target and multiplication maps $s$, $t$ and $m$, by applying the tangent functor one gets the tangent groupoid $T\calG\rightrightarrows TM$ with structure maps $ds$, $dt$, $dm$ and so on. This tangent groupoid is further a VB-groupoid over $\calG\rightrightarrows M$ (with respect to the tangent projections).
\end{example}

\begin{remark}\label{rmk-core-exact-seq-TG}
Note that in the previous example one has the following short exact sequences of vector bundles over $\calG$,
\begin{equation}\label{left}
s^{*}(A_\calG)\stackrel{-l\circ di}{\longrightarrow} T\calG\stackrel{(dt)^{!}}{\longrightarrow}t^{*}(TM)
\end{equation} and
\begin{equation}\label{right}
t^{*}(A_\calG)\stackrel{r}{\longrightarrow} T\calG\stackrel{(ds)^{!}}{\longrightarrow}s^{*}(TM)
\end{equation}
where $r$ and $l$ are the right and left multiplication on vectors tangent to the $s$-fibres and $t$-fibres of $\calG$, respectively; $(ds)^{!}$ and $(dt)^{!}$ are the maps induced by $ds$ and $dt$ with image in the corresponding pullback bundles. 
\end{remark}

\begin{example}(Cotangent groupoid)\label{Ex-cot-grpd}
As noticed in \cite{CosteDazordWeins.}, the cotangent bundle of a Lie groupoid $\calG$ inherits a groupoid structure over the dual of the Lie algebroid of $\calG$,
$$T^{*}\calG\rightrightarrows A_\calG^{*},$$ 
with source and target maps induced, respectively, from the dual of the exact sequences (\ref{left}) and (\ref{right}). Explicitly, for $\alpha_g\in T^{*}_{g}\calG$ and $a\in\Gamma(A_\calG)$,
$$\left\langle \tilde{s}(\alpha_g),a_{s(g)}\right\rangle=-\left\langle \alpha_g, l_g\circ di(a)\right\rangle$$
and
$$\left\langle \tilde{t}(\alpha_g), a_{t(g)}\right\rangle=\left\langle \alpha_g, r_g(a)\right\rangle.$$
With multiplication determined by $$\left\langle \tilde{m}(\alpha_g,\beta_h),Tm(v_g,w_h)\right\rangle=\left\langle \alpha_g,v_g\right\rangle+\left\langle \beta_h,w_h\right\rangle,$$
for $(v_g.w_h)\in (T\calG)^{(2)}$.
\end{example}

\paragraph{\textbf{VB-groupoid cohomology}}
Let $\Gamma$ be a VB-groupoid. The differentiable complex of $\Gamma$ (as Lie groupoid) has a natural subcomplex $C^\bullet_\mathrm{lin}(\Gamma)$ given by the fibrewise linear cochains of $\Gamma$.

The VB-groupoid complex  $C^\bullet_\mathrm{VB}(\Gamma)$ of $\Gamma$ (\cite{GrMe-grpds}), is a subcomplex of $C^\bullet_\mathrm{lin}(\Gamma)$ which takes into account simultaneously the Lie groupoid and vector bundle structures of $\Gamma$. It consists of the left-projectable elements, that is, those $c\in C^\bullet_\mathrm{lin}(\Gamma)$ for which
\begin{enumerate}
	\item $c(0_{g_1},\gamma_{g_2},...,\gamma_{g_k})=0$,
	\item $c(0_g\cdot\gamma_{g_1},\gamma_{g_2},...,\gamma_{g_k})=c(\gamma_{g_1},\gamma_{g_2},...,\gamma_{g_k})$.
\end{enumerate}
This complex turns out to compute the cohomology of the base groupoid with coefficients in a representation up to homotopy. For the case of the tangent VB-groupoid it  yields an interpretation of the adjoint representation (\cite{GrMe-grpds}). We will make use of the following fact.

\begin{proposition}[\cite{GrMe-grpds}, Prop. 5.5, Thm. 5.6]\label{deformation and vb-complexes}

The deformation complex of a Lie groupoid $\calG$ is isomorphic to the VB-groupoid complex of its cotangent VB-groupoid, $C^\bullet_\mathrm{VB}(T^{*}\calG)$. The isomorphism is given by $C^\bullet_\mathrm{def}(\calG)\longrightarrow C^\bullet_\mathrm{VB}(T^{*}\calG)$, $c\mapsto c'$ with
$$c'(\eta_{g_1},...,\eta_{g_k})=\left\langle\eta_{g_1}, c(g_1,...,g_k)\right\rangle.$$ 
\end{proposition}

We recall also the following result of Cabrera and Drummond, that will let us substitute the VB-complex by the complex of linear cochains for most purposes.

\begin{lemma}[\cite{CD}, Lemma 3.1]\label{LemmaCD} Let $(\Gamma, E, \calG, M)$ be a VB-groupoid. The inclusion $\iota:C^\bullet_\mathrm{VB}(\Gamma) \hookrightarrow C^\bullet_\mathrm{lin}(\Gamma)$ induces an isomorphism of right $H^\bullet(\calG)$-modules in cohomology.
\end{lemma}

\subsection{Multiplicative forms, symplectic groupoids, and Poisson manifolds}

Given a Lie groupoid $\G$, a differential $k$-form $\alpha\in\Omega^k(\G)$ is called \textbf{multiplicative} if it satisfies \[m^*\alpha=pr_1^*\alpha + pr_2^*\alpha,\] where $m$ is the multiplication of $\G$ and $pr_1,pr_2:\G^{(2)}\to\G$ are the two projections. We denote by $\Omega^\bullet_\mathrm{mult}(\G)\subset \Omega^\bullet(\G)$ the subcomplex of multiplicative forms on $\G$.

\begin{remark}\label{rmk-alt-mult} Multiplicativity can be expressed in yet a few other equivalent ways, some of which will be useful in this study. By definition, multiplicativity of $\alpha$ amounts to $\delta\alpha=0$, where $\delta=pr_1^* - m^* + pr_2^*$ is the horizontal differential of the first column in the Bott-Shulman-Stasheff double complex (cf.\ Section \ref{subsec-Nerve}).

Moreover, note that to a $k$-form $\alpha\in \Omega^k(\G)$ one can associate the map \[\hat\alpha:\oplus^k_\G T\G\to\mathbb{R},\] given by contraction of $\alpha$ with tangent vectors. The map $\hat{\alpha}$ is skewsymmetric and multilinear (with respect to the linear structure of $\oplus^k_\G T\G$ over $\G$). We can now determine multiplicativity of $\alpha$ in terms of $\hat{\alpha}$.
\end{remark}

\begin{lemma}\cite[Lemma 4.1]{BC} The form $\alpha$ is multiplicative if and only if the map $\hat{\alpha}$ is a groupoid morphism (where the groupoid structure on $\oplus^k_\G T\G$ is the natural extension of the structure of $T\G$, and $\mathbb{R}$ is seen as an additive group).
\end{lemma}

A \textbf{symplectic groupoid} is a pair $(\G,\omega)$, where $\G$ is a Lie groupoid and $\omega$ is a multiplicative symplectic form on $\G$.
As seen above this can be interpreted as having a pair of a Lie groupoid $\G$ and a symplectic form $\omega$ on $\G$ for which $\hat\omega:\oplus^2_\G T\G\to \mathbb{R}$ is a groupoid morphism. This point of view will be useful, because it lets us draw from the study of deformations of Lie groupoid morphisms \cite{CS} in order to study deformations of symplectic groupoids.

Symplectic groupoids are very closely related to Poisson manifolds. Recall that a \textbf{Poisson structure} on a manifold $M$ is a Lie bracket 
\[\{\cdot,\cdot\}:C^\infty(M)\times C^\infty(M)\to C^\infty(M)\]

which is a derivation in each entry, i.e.\ it satisfies the Leibniz identity \[\{f,gh\}=\{f,g\}h + g\{f,h\}.\]

Being a biderivation, it can be seen as a bivector field $\pi\in \Gamma(\wedge^2(TM))$. 
 The pair $(M,\pi)$ is called a \textbf{Poisson manifold}.
 
\newpage 
Any Poisson manifold $(M,\pi)$ gives rise to a Lie algebroid structure on the bundle $T^*M\to M$, called the \textbf{cotangent Lie algebroid}. The anchor $\pi^\sharp: T^*M\to TM$ is given by contraction with the Poisson bivector, $\pi^\sharp(\alpha)=\pi(\alpha,\cdot)$, while the Lie bracket on $\Omega^1(M)$ is defined by \[[\alpha, \beta]_\pi=\mathcal{L}_{\pi^\sharp(\alpha)}\beta-\mathcal{L}_{\pi^\sharp(\beta)}\alpha-d(\pi(\alpha,\beta)).\]

Returning to symplectic groupoids, let us recall a few of the many notable properties they satisfy (see for example \cite{CosteDazordWeins.} for more).
Let $(\G, \omega)$ be a symplectic groupoid over $M$. Then 
\begin{enumerate}[leftmargin=*]\item The dimension of $\G$ is twice the dimension of $M$.
\item $M$ (embedded via the unit map) is a Lagrangian submanifold of $\G$. This is a consequence of the multiplicativity of $\omega$ and of the relation between the dimensions.
\item Let $A$ be the Lie algebroid of $\G$. There is a natural splitting at the units\[T\G_{|M}=A\oplus TM,\]
and since $(\G, \omega)$ is symplectic, it follows from the dimension relation that $A$ and $TM$ have the same rank. Moreover, the fact that $M$ is Lagrangian, together with the splitting $T\G_{|M}=A\oplus TM$, implies that $\omega:\bigoplus^2 T\G\to \mathbb{R}$ induces a non-degenerate pairing $TM\oplus A\to \mathbb{R}$, and hence an isomorphism $\omega^\flat_{|TM}:TM\to A^*$. Its dual map $-\omega^\flat_{|A}$ induces an isomorphism $A\cong T^*M$.

\item The anchor of $A$, viewed via the identification above as a map $T^*M\to TM$, is anti-selfadjoint, giving us a bivector $\pi\in \mathfrak{X}^2(M)$. Moreover, $\pi$ is in fact Poisson, and the induced algebroid structure on $T^*M$ is that of the cotangent algebroid of $(M, \pi)$.
\end{enumerate}

A Poisson manifold $(M, \pi)$ is called integrable if there is a symplectic groupoid $(\G, \omega)$ over $M$ inducing the Poisson structure on the base. In that case, as mentioned above,  $\G$ integrates the cotangent Lie algebroid $T^*M$. There is a converse to this statement.

\begin{theorem}[\cite{MX3}, Thm. 5.2] Let $(M,\pi)$ be a Poisson manifold such that the cotangent Lie algebroid $T^*M$ is integrable. Then the source $1$-connected integration of $T^*M$ admits a natural multiplicative symplectic structure, which induces the Poisson structure $\pi$ on $M$.
\end{theorem}

Let us mention a few examples of integrable Poisson manifolds, and of symplectic groupoids integrating them.

\ \

\paragraph{\textbf{Symplectic manifolds}}
Any  symplectic manifold $(M, \omega)$ can be seen as a Poisson manifold,  with the Poisson bivector being given by the inverse of the symplectic form, $\pi=\omega^{-1}$. The pair groupoid $M\times M \rightrightarrows M$, equipped with the symplectic form $pr_2^*\omega-pr_1^*\omega$ is a symplectic groupoid integrating $\pi$. The source $1$-connected symplectic groupoid integrating $(M,\omega)$ is the symplectic fundamental groupoid $(\Pi_1(M), s^*\omega-t^*\omega)$. 

Any other $s$-connected Lie groupoid $\G$ integrating $T^*M$ lies between these two extremes, as there are groupoid submersions $\Pi_1(M)\to \G \to M\times M$. Moreover, $\G$ is a symplectic groupoid when equipped with the form $s^*\omega-t^*\omega$.

\ \

\paragraph{\textbf{Zero Poisson structure}}

Let $(M,0)$ be a manifold equipped with the zero Poisson structure, meaning that the Poisson bracket of any two functions is zero. The source $1$-connected symplectic groupoid integrating $(M,0)$ is the cotangent bundle of $M$, with the canonical symplectic structure $(T^*M, \omega_\mathrm{can})$. The source and target maps of $T^*M$ are equal to the vector bundle projection, and the groupoid multiplication is the fibrewise addition.

\ \

\paragraph{\textbf{Linear Poisson manifolds}}

A linear Poisson structure on a vector space $V$ is one for which the Poisson bracket of linear functions is again linear. This happens if and only if there is a Lie algebra $\mathfrak{g}$ such that $V=\mathfrak{g}^*$, and the Poisson bracket is the Kirillov-Konstant-Souriau bracket, which extends the Lie bracket on $\mathfrak{g}$ (seen as the space of linear functions on $\mathfrak{g}^*$) to all smooth functions on $\mathfrak{g}^*$. We denote the corresponding Poisson bivector by $\pi_\mathfrak{g}$. Explicitly, \[\{f,g\}(\eta):=\langle\eta, [d_\eta f,d_\eta g]_\mathfrak{g}\rangle, \ \ \forall \eta\in \mathfrak{g}^*, \ f,g\in C^\infty(\mathfrak{g}^*).\]

The $s$-connected symplectic groupoids integrating the linear Poisson manifold $\mathfrak{g}^*$ are of the form $(T^*G, \omega_\mathrm{can})$, for any connected Lie group $G$ integrating $\mathfrak{g}$. The groupoid structure of $T^*G$ is that of the cotangent groupoid of $G$, as explained in Example \ref{Ex-cot-grpd}.

The source 1-connected symplectic groupoid integrating $(\mathfrak{g}^*, \pi_\mathfrak{g})$ is the cotangent groupoid $(T^*G, \omega_\mathrm{can})$ of the $1$-connected group $G$ integrating $\mathfrak{g}$.

The cotangent groupoid $T^*G$ is also isomorphic to the action groupoid $G\ltimes \mathfrak{g}^*$ of the coadjoint action of $G$ on $\mathfrak{g}^*$. An isomorphism is given by the trivialization of the cotangent bundle via right translations, $T^*G\cong G\times \mathfrak{g}^*$.

\ \

\paragraph{\textbf{Cotangent VB-groupoids}}

Generalizing the previous case, Lie algebroid structures on a vector bundle $A$ are in $1:1$ correspondence with Poisson structures on its dual $A^*$ which are linear along the fibres \cite{Courant}. By that we mean a Poisson structure such that the Poisson bracket of two fibrewise linear functions on $A^*$ is linear, and the bracket of a fibrewise linear function with a basic function is basic, and the bracket of two basic functions is zero.

Given any Lie groupoid $\G \rightrightarrows M$ with Lie algebroid $A=\mathrm{Lie}(\G)$, the cotangent Lie groupoid $T^*\G\rightrightarrows A^*$ (Example \ref{Ex-cot-grpd}) is a symplectic groupoid integrating $A^*$. The source 1-connected integration of $(A^*, \pi_\mathrm{lin})$ is the cotangent groupoid $(T^*\G, \omega_\mathrm{can})$ of the source $1$-connected groupoid $\G$ integrating $A$.

\section{Deformations of symplectic groupoids}

We now come to the notion of a deformation of a symplectic groupoid $(\mathcal{G},\omega)$. It naturally consists of an underlying deformation of Lie groupoids of $\mathcal{G}$ (Definition \ref{Deform.grpds}), with a compatible deformation of the multiplicative symplectic form $\omega$. For simplicity let us start with strict deformations.

\begin{definition}\label{Deformationsympl.grpds} (Strict deformations of symplectic groupoids)\\
Let $(\calG,\omega)$ be a symplectic groupoid, let $\calG_\e=(\calG,\bar{m}_{\e})$ be a strict deformation of $\calG$ and let $\omega_\e\in\Omega^{2}_\mathrm{cl}(\calG)$ be a deformation of $\omega$ by symplectic forms. The family $(\calG_\e,\omega_\e)$ is called a \textbf{strict deformation} of $(\calG,\omega)$ if each $(\calG_{\e},\omega_\e)$ is a symplectic groupoid. 

Similarly, we define an \textbf{ $s$-constant deformation of symplectic groupoids} as one for which the underlying Lie groupoid deformation is $s$-constant.

An $s$-constant deformation of $(\G, \omega)$ such that $(\bar{m}_{\e},\omega_{\e})=(\bar{m},\omega)$ for all $\e$ is called a \textbf{constant deformation} of $(\calG,\omega)$.
\end{definition}

Examples of deformations can be easily constructed (at least) in two simple ways.

\begin{examples}\label{Deformsympgrpdsexamples}
Consider $(\calG,\omega)$ a symplectic groupoid with $(M,\pi)$ the Poisson structure induced on the base $M$.

\begin{enumerate}
\item (Diffeomorphisms) Let $\phi_\e:\calG\longrightarrow\calG$ be a smooth family of diffeomorphisms with $\phi_0=id_{\calG}$. For every $\e$, induce the groupoid structure $\calG_\e:=(\calG,\bar{m}_\e)$ in such a way that $\phi_\e:\calG_\e\longrightarrow\calG$ is an isomorphism of Lie groupoids. Then $(\calG_\e,\phi^{*}_\e\omega)$ turns out to be a (strict) deformation of $(\calG,\omega)$.

\item (Basic gauge transformations) Let $\alpha_\e\in\Omega^{2}_\mathrm{cl}(M)$ be a smooth family of closed 2-forms, with $\alpha_0=0$. Denote by $\Omega_{\alpha_\e}=\omega+s^{*}\alpha_\e-t^{*}\alpha_\e$ the family of multiplicative 2-forms on $\G$ induced by $\alpha_\e$. Then $(\calG,\Omega_{\alpha_\e})$ is an $s$-constant deformation of $(\G,\omega)$ as long as $\Omega_{\alpha_\e}$ stays non-degenerate for small enough $\e$ (for example, this happens if $\G$ is compact).
In particular, $(\calG,\Omega_{\alpha_\e})$ is the symplectic groupoid integrating the gauge transformation $\pi^{\alpha_\e}$ of $\pi$ by $\alpha_\e$.
\end{enumerate}
\end{examples}

These classes of deformations will be considered trivial, and lead us to the notion of equivalence of deformations.

\begin{definition}\label{Equivdeform}%(Equivalence of deformations)
We say that a deformation $(\calG'_{\e},\omega'_{\e})$ of $(\G,\omega)$ is \textbf{equivalent} to another deformation $(\calG_{\e},\omega_{\e})$ of $(\G,\omega)$ if there exists a smooth family of isomorphisms $F_\e:\calG_\e\rightarrow\calG'_{\e}$ and a smooth family of closed 2-forms $\alpha_{\e}$ on $M$, such that $F_0=id_{\calG}$, $\alpha_0=0$, and
$$F_\e^{*}\omega'_\e=\omega_\e+\delta_\e(\alpha_\e), \text{ for every }\e\in J,$$
where $\delta_\e(\alpha_\e)=s^{*}_{\e}(\alpha_\e)-t^{*}_{\e}(\alpha_\e)$, and $J\subset I\cap I'$ is an open interval containing zero. A deformation is said to be \textbf{trivial} if it is equivalent to a constant deformation.
\end{definition}

General deformations of symplectic groupoids, and their equivalences,  are defined similarly, allowing for non-strict deformations of the underlying Lie groupoid.

\begin{definition} A \textbf{deformation of the symplectic groupoid} $(\calG,\omega_\calG)$ consists of a deformation $\widetilde{\calG}\rightrightarrows\tilde{M}\stackrel{\textbf{p}}{\rightarrow} I$ of $\calG$, together with a multiplicative 2-form $\omega$ on $\widetilde{\calG}$ such that, for every $\e$, the restriction $\omega_\e$ of $\omega$ to $\calG_\e$ makes $(\G_\e, \omega_\e)$ into a symplectic groupoid, and $(\calG_0,\omega_0)=(\calG,\omega_\calG)$. 
\end{definition}

We recall that an equivalence between deformations $\widetilde{\G}$ and $\widetilde{\G}'$ of a Lie groupoid $\G$ is a Lie groupoid isomorphism $F:\widetilde{\G}_{|_J}\to\widetilde{\G}_{|_J}'$ such that $\textbf{p}'\circ F=\textbf{p}$, and $F_{|\G_0}=\mathrm{id}_{\G_0}$, where $\widetilde{\G}_{|_J}\rightrightarrows \widetilde{M}_{|_J}$ denotes the restriction of the family $\widetilde{\G}\rightrightarrows \widetilde{M}$ to some open interval $J\subset I\cap I'$ containing zero.

Let $\mathcal{F}$ and $\mathcal{F}_M$ denote the foliations of $\widetilde{\G}$ and $\widetilde{M}$ by the fibres of $\textbf{p}$. Given a form $\omega\in\Omega^k(\widetilde{\G})$ denote by $\omega_\mathcal{F}\in \Omega^k_\mathcal{F}(\widetilde{\G})$ its corresponding foliated form.

\begin{definition}
An \textbf{equivalence between deformations} $(\widetilde{\calG},\omega)$ and $(\widetilde{\calG}',\omega')$ of $(\calG,\omega_\calG)$ is given by an equivalence of deformations of Lie groupoids $F:\widetilde{\G}_{|_J}\to\widetilde{\G}_{|_J}'$, and a closed foliated $2$-form $\alpha_\mathcal{F}\in\Omega^{2}_{\mathcal{F}_M}(\widetilde{M}_{|_J})$ such that $\alpha_\mathcal{F}$ vanishes along $\G_0$, and 
\[(F^{*}\omega')_\mathcal{F}=\omega_\mathcal{F}+\delta(\alpha_\mathcal{F}),\] for some open interval $J\subset I\cap I'$ containing zero.

\end{definition}

\begin{example}\label{Example cotangent-gpd-deformation}Given a strict deformation $\calG_\e$ of a Lie groupoid $\calG$ there is an induced deformation of Lie algebroids $A_\e=\mathrm{Lie}(\calG_\e)$. It in turn induces a deformation of Poisson manifolds $(A_\e^*, \pi_\e)$ such that each $\pi_\e$ is the fibrewise linear Poisson structure on $A^*_\e$ corresponding to the algebroid $A_\e$. Finally, each of these Poisson manifolds is integrated by the symplectic groupoid $(T^*\calG_\e, \omega_\mathrm{can})$. Since the deformation of Lie groupoids $\calG_\e$ is strict (in fact, a strict deformation of $\mathrm{VB}$-groupoids, cf.\ \cite{PierVitag}), $(T^*\calG_\e, \omega_\mathrm{can})=(T^*\calG,\omega_\mathrm{can})$ as symplectic manifolds, but the Lie groupoid structure of $T^*\calG_\e$ can vary.

In short, any strict deformation of Lie groupoids induces deformations of Lie algebroids, fibrewise linear Poisson manifolds, and symplectic groupoids.
\[\G_\e \rightsquigarrow A_\e \rightsquigarrow (A_\e^*, \pi_\e)\rightsquigarrow (T^*\calG_\e, \omega_\mathrm{can})\]
\end{example}

\section{The deformation complex of a (symplectic) groupoid, revisited}\label{Def.comp.symp.grpds}

We revisit the deformation complex of a Lie groupoid $\calG$ in the case that $(\calG,\omega)$ is a symplectic groupoid. We will use the symplectic form $\omega$ to identify $C^\bullet_\mathrm{def}(\calG)$ with a subcomplex $\Omega^{1}_\mathrm{def}(\calG^{(\bullet)})$ of $\Omega^{1}(\calG^{(\bullet)})$, the complex of 1-forms on the nerve of $\calG$.

\ \

\paragraph{\textbf{Step 1}} Recall that the deformation complex $C^\bullet_\mathrm{def}(\calG)$ can be identified with the VB-groupoid complex $C_\mathrm{VB}^\bullet(T^{*}\calG)$ of $T^{*}\calG$ (Prop. \ref{deformation and vb-complexes}), which is a subcomplex of $C^\bullet_\mathrm{lin}(T^{*}\calG)$. 

\paragraph{\textbf{Step 2}}We use the fact that the vector bundle isomorphism $\omega^{b}:T\calG\rightarrow T^{*}\calG$ is actually a VB-groupoid isomorphism, due to the multiplicativity of $\omega$ (Lemma 3.6 in \cite{BCO}). Therefore it determines an isomorphism between $C^\bullet_\mathrm{lin}(T^{*}\calG)$ and  $C^\bullet_\mathrm{lin}(T\calG)$. We use it to further identify the deformation complex of $\calG$ with $C_\mathrm{VB}^\bullet(T\calG)\subset C^\bullet_\mathrm{lin}(T\calG)$.

\paragraph{\textbf{Step 3}} The inclusion $C_\mathrm{VB}^\bullet(T\calG) \subset C^\bullet_\mathrm{lin}(T\calG)$ is a quasi-isomorphism (Lemma 3.1 in \cite{CD}, Lemma \ref{LemmaCD}).

\paragraph{\textbf{Step 4}} For each integer $k$, there is a natural isomorphism $(T\calG)^{(k)}\cong T\calG^{(k)}$. In this way, the elements of $C^{k}_\mathrm{lin}(T\calG)$ can be viewed as the fibrewise linear functions on $T\calG^{(k)}$, i.e., as elements of $C^{\infty}_\mathrm{lin}(T\calG^{(k)})\cong\Omega^{1}(\calG^{(k)})$.

Thus, at the level of cochains, we have the following chain of identifications and quasi-isomorphisms:
\[C^{k}_\mathrm{def}(\calG)\cong C_\mathrm{VB}^{k}(T^{*}\calG) \stackrel{\omega^{b}}{\cong} C_\mathrm{VB}^{k}(T\calG) \stackrel{\mathrm{q. i.}}{\hookrightarrow} C^{k}_\mathrm{lin}(T\calG):=C^{\infty}_\mathrm{lin}(T\calG^{(k)})\cong\Omega^{1}(\calG^{(k)})\,\text{.}\]

The last identification of $C^\bullet_\mathrm{lin}(T\calG)$ with $\Omega^\bullet(\calG)$ is compatible with the differentials of the complexes involved. In fact, the two differentials correspond to the two different ways of considering the alternated sum of pullbacks of 1-forms $\alpha$ by the face maps $d_i:\calG^{(n)}\to\calG^{(n-1)}$ of the nerve of $\calG$: as the usual pullback of forms $d_i^{*}\alpha$, or as the pullback of functions $(Td_i)^{*}:C^{\infty}_\mathrm{lin}(T\calG^{(n+1)})\longrightarrow C^{\infty}_\mathrm{lin}(T\calG^{(n)})$ by $Td_i$.

We summarize now the description of the deformation complex of $\calG$ obtained above, in the symplectic case. 

\begin{proposition}\label{prop:isomorph-nondegeneracy}
Let $(\calG,\omega)$ be a symplectic groupoid. Then $\omega^{b}$ induces an isomorphism between the deformation complex $C^{*}_\mathrm{def}(\calG)$ of $\calG$ and a subcomplex $\Omega^{1}_\mathrm{def}(\calG^{(\bullet)})$ of $\Omega^{1}(\calG^{(\bullet)})$. Moreover, $\Omega^{1}_\mathrm{def}(\calG^{(\bullet)})\hookrightarrow\Omega^{1}(\calG^{(\bullet)})$ is a quasi-isomorphism.
\end{proposition}

\begin{remark}\label{remark:explicite isomorph}

The complex $\Omega^{1}_\mathrm{def}(\calG^{(\bullet)})$ obtained under the identification above can be explicitly described by translating the conditions of the VB-groupoid complex to conditions on 1-forms on the nerve of $\calG$. Namely, 
 a 1-form $\alpha$ on $\calG^{(k)}$, with $k>0$ will belong to $\Omega^{1}_\mathrm{def}(\calG^{(k)})$ if the following two conditions are satisfied (when $k=0$ there are no conditions imposed).
\begin{enumerate}
\item $\alpha\in\Gamma_{\calG^{(k)}}(T^{*}\calG^{(k)})$ comes from a section $\hat{\alpha}$ of the bundle $pr_{1}^{*}T^{*}\calG\rightarrow\calG^{(k)}$, where $pr_{1}:\calG^{(k)}\rightarrow\calG$ is the projection on the first component. That is, there exists a necessarily unique section $\hat{\alpha}$ such that the left triangle below commutes;
\begin{equation}\label{triangledefiningsigma}
  \begin{tikzcd}[column sep=3em, row sep=10ex]
    T^{*}\calG^{(k)} \arrow{d}[swap]{c_{\calG^{(k)}}} & pr_{1}^{*}(T^{*}\calG) \arrow{l}[swap]{(d\,pr_1)^{*}} \arrow[xshift=-3pt]{ld}{ } \arrow{r}{ } & T^{*}\calG \arrow{d}{c_{\calG}}\\
		\calG^{(k)}	\arrow[xshift=3pt]{ur}[swap]{\hat{\alpha}} \arrow[xshift=3pt]{u}[swap]{\alpha} \arrow{rr}{pr_1} &  & \calG
		%\calG \arrow[xshift=2pt]{d}{t} \arrow[xshift=-2pt]{d}[swap]{s}    & M,
  \end{tikzcd}
\end{equation}
\item $\alpha$ is left-invariant for the action $g\cdot(g_1,\ldots,g_k)\mapsto (gg_1,\ldots,g_k)$ of $\calG$ on $\calG^{(k)}$.
\end{enumerate}

In terms of cochains, the identification of Proposition \ref{prop:isomorph-nondegeneracy} is the composition of the following maps
\[
  \begin{tikzcd}[column sep=4em, row sep=0.1em]
    \parbox[0]{4em}{$C^{k}_\mathrm{def}(\calG)$} \arrow{r}{} & \parbox[2em]{4em}{$C^{k}_\mathrm{lin}(T^{*}\calG)$} \arrow{r}{(\omega^{b})^{*}} & \parbox[0]{4em}{$C^{k}_\mathrm{lin}(T\calG)$} \arrow[<->]{r}{} & \parbox[0]{4em}{$\Omega^{1}_\mathrm{def}(\calG^{(k)})$}\\
    \parbox[t]{4em}{\hspace{1em}$c$} \arrow[|->]{r}{} & \parbox[t]{4em}{\hspace{1.5em}$c'$} \arrow[|->]{r}{} & \parbox[t]{4em}{$(\omega^{b})^{*}c'$} \arrow[|->]{r}{} & \parbox[t]{4em}{\hspace{1em}$\hat{c}$,}
  \end{tikzcd}
	\]

where $c'(\eta_{g_{1}},\ldots,\eta_{g_{k}}):=\left\langle \eta_{g_{1}}, c(g_{1},\ldots, g_{k})\right\rangle$ and, by construction, 
	
	\begin{equation}\label{identification}
	\begin{split}
	\hat{c}(v_{g_{1}},\ldots,v_{g_{k}})&=(\omega^{b})^{*}c'(v_{g_{1}},\ldots,v_{g_{k}})\\
	&=c'(\omega^{b}(v_{g_{1}}),\ldots,\omega^{b}(v_{g_{k}}))\\
	&=-\left\langle\omega^{b}(c(g_{1},\ldots, g_{k})),v_{g_{1}}\right\rangle\\
	&=-\left[(dpr_{1})^{*}(\omega^{b}(c(g_{1},\ldots, g_{k})))\right](v_{g_{1}},\ldots,v_{g_{k}}).
	\end{split}
	\end{equation}

In other words, $c\in C^{k}_\mathrm{def}(\calG)$ corresponds to the section $-\omega^{b}\circ c:\calG^{(k)}\rightarrow pr_{1}^{*}T^{*}\calG$ of condition $(1)$ above. Condition $(2)$ holds for $\hat{c}$ as it amounts to the $s$-projectability condition of $c$. In degree zero, this correspondence $C^0_\mathrm{def}(\G)=\Gamma(A)\cong \Omega^1(M)$ amounts to the identification $-\omega^\flat_{|A}:A\cong T^*M$, at the level of sections.
\end{remark}

\begin{remark}
Note that, in the identification above, we only use the multiplicativity and non-degeneracy conditions of $\omega$. Then, in particular, such an identification also holds for twisted symplectic groupoids (cf.\ \cite{BCWZ}).
\end{remark}

\begin{remark}\label{rmk:sigma_e}
Denote by $\sigma$ the vector bundle maps $(d\,pr_1)^{*}$ given by the diagram \eqref{triangledefiningsigma} above, for any $k$. If $(\calG_\e,\omega_\e)$ is an $s$-constant deformation of $(\calG,\omega)$, then under the identification above the deformation cocycle $\xi_0$ of $\calG_\e$ corresponds to the 1-form $\zeta_0=-\sigma(\omega(\xi_0,\cdot))\in\Omega^{1}_\mathrm{def}(\calG^{(2)})$. And, since $\sigma=Id$ for $k=1$, it follows that for $X\in C^{1}_\mathrm{def}(\calG)$ the corresponding 1-form on $\calG$ will be $-i_{X}\omega=-\omega^{b}(X)$.
\end{remark}

\section{The deformation complex of a symplectic groupoid}

Our approach to find the deformation complex of a symplectic groupoid is to construct it out of the deformation complex of a Lie groupoid \cite{CMS} (taking also into account the discussion of Section \ref{Def.comp.symp.grpds}), and the deformation complex for multiplicative forms, introduced in \cite{CS}. 

In section \ref{section:cocycle} we will study the resulting complex, looking at cocycles and coboundaries related to deformations and equivalences, confirming that this approach leads indeed to the correct complex.

\subsection*{Deformations of symplectic groupoids in terms of morphisms}

Let $(\calG_\e,\omega_\e)$ be a deformation of the symplectic groupoid $(\calG,\omega)$. The first step is to consider the associated morphism $\hat\omega:\bigoplus^2T\G\to\mathbb{R}$, described in Remark \ref{rmk-alt-mult}. Doing so,
$(\calG_\e,\omega_\e)$ can be replaced by the simultaneous deformation $(\bigoplus^{2}T\calG_\e,\hat{\omega}_\e)$ of the Lie groupoid $\bigoplus^{2}T\calG$ and of the morphism $\hat{\omega}$.

Note that this is a very particular type of simultaneous deformation of a Lie groupoid and a morphism having that groupoid as domain. The deformation of $\bigoplus^{2}T\G$ is of the form $\bigoplus^{2}T(\calG_\e)$ and the deformation of $\hat\omega$ is  by (closed, non-degenerate) bilinear and skewsymmetric morphisms $\hat{\omega}_\e$. Let us then first examine these separate pieces.

\subsection*{Deformations of multiplicative 2-forms}

Let us first describe the deformation complex $\hat{C}_\mathrm{def}^\bullet(\omega)$ of a multiplicative 2-form $\omega$, defined in \cite{CS}. It is the subcomplex of \textbf{bilinear} and \textbf{skewsymmetric} cochains of the deformation complex $C^\bullet_\mathrm{def}(\hat{\omega})$ of the associated Lie groupoid morphism $\hat{\omega}$. We remark that our notation differs from the one in \cite{CS}, where the deformation complex $\hat{C}_\mathrm{def}^\bullet(\omega)$ was denoted simply by $C_\mathrm{def}^\bullet(\omega)$. We reserve the latter notation for the deformation complex of a \textit{closed} multiplicative form $\omega$.

As shown in \cite{CS}, it is isomorphic to the complex $\Omega^{2}(\G^{(\bullet)})$ of 2-forms on the nerve of $\G$, and it can also be identified with the subcomplex $C^{\bullet}(\bigoplus^{2}_{p_\G}T\G)_\mathrm{2-lin,\, sk}$ of 2-linear skewsymmetric cochains of the differentiable complex of the groupoid $\bigoplus^{2}_{p_\G}T\G$.
As remarked in \cite{CS}, a deformation $\omega_\e$ of $\omega$ induces a family of cocycles $\frac{d}{d\e}\hat{\omega}_\e$ in $\hat{C}^{1}_\mathrm{def}(\omega_\e)$.

In our setting, $\omega$ is closed and we are interested in deformations of it via closed multiplicative forms. The deformation complex to consider is then the subcomplex $C_\mathrm{def}^\bullet(\omega)\subset \hat{C}_\mathrm{def}^\bullet(\omega)$, corresponding to closed $2$-forms on the nerve of $\G$.

\subsection*{Tangent lifts of deformation cochains}

For any Lie groupoid $\G$, the \textbf{tangent lift} of deformation cochains, described in \cite{PierVitag}, is the inclusion of cochain complexes   $T_\mathrm{def}:C^{\bullet}_\mathrm{def}(\calG)\to C^{\bullet}_\mathrm{def}(T\calG)$  determined by composing the canonical involution (also called canonical flip, cf.\ \cite[Section 8.13]{Michor}) $J_\calG:TT\calG\to TT\calG$  and the differential $d$ of a smooth function:
$$T_\mathrm{def}c:=J_\calG\circ dc.$$

We also call tangent lift to its natural extension $\oplus^{2}T_\mathrm{def}: C^\bullet_\mathrm{def}(\calG)\to C^\bullet_\mathrm{def}(\oplus^{2}T\calG)$ defined as $\oplus^{2}T_\mathrm{def}(c):=\oplus^{2}(T_\mathrm{def}c),$ of which we will make use.

As remarked in \cite{ CS, PierVitag}, the tangent lift includes the deformation cocycles of $\G$ into the deformation complex of $\oplus^2 T\G$, as cocycles associated to deformations of the form $\bigoplus^{2}T(\G_\e)$.

\subsection*{The deformation complex}

We come back to the idea that deformations of symplectic groupoids can be seen as particular simultaneous deformations of $\bigoplus^{2}T\calG$, of the form $\bigoplus^{2}T\calG_\e$, and of a morphism $\hat{\omega}$, via morphisms associated to (closed, non-degenerate) multiplicative forms.

In this situation, we can use the idea from \cite{CS} of using the mapping cone construction to control simultaneous deformations.
As described in Section \ref{def_complex_morphism}, $\hat{\omega}$ induces a map $\hat{\omega}_{*}:C_\mathrm{def}^{\bullet}(\bigoplus^{2}T\G)\to C_\mathrm{def}^{\bullet}(\hat\omega)$ between deformation complexes. In fact, the bilinearity and skewsymmetry of $\omega$ imply, by direct computation, that its image is contained in the subcomplex $C_\mathrm{def}^{\bullet}(\omega)$.

\begin{definition}The \textbf{deformation complex of the symplectic groupoid $(\calG, \omega)$}, denoted by $C^{\bullet}_\mathrm{def}(\G,\omega)$, is the mapping cone complex of the map
$$\hat{\omega}_{*}\circ\oplus^{2}T_\mathrm{def}:C^{\bullet}_\mathrm{def}(\G)\to C_\mathrm{def}^{\bullet}(\omega).$$
\end{definition}

\begin{remark}\label{remark:cocycle}

Considering deformations of symplectic groupoids as simultaneous deformations of a morphism of Lie groupoids and of its domain, the work of \cite{CS}, already indicates that the element
$$\xi_\e-\frac{d}{d\e}\hat{\omega}_\e\in C^{2}_\mathrm{def}(\G_\e)\bigoplus C^{1}_\mathrm{def}(\omega_\e)$$
should be the deformation cocycle associated to the deformation $(\G_\e,\omega_\e)$. Here $\xi_\e$ is the deformation class of the underlying deformation of Lie groupoids, described in Section \ref{section-def-class-Lie-gpd}. We will reach this same construction explicitly (see Section \ref{section:cocycle}) using the alternative description of the deformation complex given by the Theorem \ref{viewpoint2} below.
\end{remark}

The next Theorem shows that using identifications given by the symplectic form $\omega$, the mapping cone complex defined above has a more familiar presentation, given in terms of differential forms.

\begin{theorem}\label{viewpoint2}
Let $(\G,\omega)$ be a symplectic groupoid. Under the correspondence $C_\mathrm{def}^\bullet(\calG)\cong \Omega_\mathrm{def}^{1}(\calG^{(\bullet)})$, induced by the isomorphism $\omega^{b}$ (see Proposition \ref{prop:isomorph-nondegeneracy}), the map $\hat{\omega}_{*}\circ\oplus^{2}T_\mathrm{def}$ agrees with the de Rham differential $d_{dR}:\Omega_\mathrm{def}^{1}(\calG^{(\bullet)})\to\Omega^{2}_\mathrm{cl}(\calG^{(\bullet)})$ up to a sign. That is, the following diagram is commutative.

\begin{equation}\label{diagrama:deRham-mapping}
\xymatrix@C+1pc{C_\mathrm{def}^\bullet(\calG) \ar[r]^{\hat{\omega}_{*}\circ\oplus^{2}T_\mathrm{def}} \ar[d]_{-\sigma\circ\omega^{b}} & C_\mathrm{def}^\bullet(\omega) \ar[d]^{\cong}\\
\Omega^{1}_\mathrm{def}(\G^{(\bullet)}) \ar[r]^{-d_{dR}} & \Omega_\mathrm{cl}^{2}(\G^{(\bullet)})}
\end{equation}

\end{theorem}

\begin{proof}
Let us first review the correspondence $C_\mathrm{def}^{\bullet}(\calG)\cong \Omega_\mathrm{def}^{1}(\calG^{(\bullet)})$ of the previous section in terms of vector fields. Given any splitting $\tau:t^{*}TM\to T\G$ of the target map $dt:T\G\to t^{*}TM$ of $\G$, any deformation cochain $c$ of $C^{k}_\mathrm{def}(\calG)$ induces a vector field $X_c\in\mathfrak{X}(\calG^{(k)})$ on $\calG^{(k)}$. For example, if $c\in C_\mathrm{def}^{2}(\G)$ is a degree 2 deformation cochain, the vector field $X_{c}$ is determined by
$$X_c(g,h):=(c(g,h),\tau_{h}(ds(c(g,h)))).$$
Analogously, for a higher degree cochain $c$, the components of the vector field $X_c$ are determined by successively $s$-projecting and $\tau$-lifting the deformation cochain $c$. The correspondence $C_\mathrm{def}^\bullet(\calG)\cong \Omega_\mathrm{def}^{1}(\calG^{(\bullet)})$ induced by the isomorphism $\omega^{b}$ can be written as
$$c\mapsto -i_{X_c}(pr_1^{*}\omega) \in \Omega^{1}_\mathrm{def}(\calG^{(k)});$$
where $pr_1:\G^{(k)}\to\G$ is the 1st-component projection map. Note that even though $X_c$ depends on $\tau$,  the form $-i_{X_c}(pr_1^{*}\omega)$ does not.

Therefore, when doing the same for the tangent groupoid $T\G$, we are free to choose a convenient splitting $\tilde{\tau}$ of the target map $d\tilde{t}:TT\G\to\tilde{t}^{*}TTM$, where $\tilde{t}:T\G\to TM$ is the target map of $T\G$.
We choose $\tilde{\tau}$ to be induced by the tangent lift $T\tau$ of the previous splitting $\tau$. Explicitly, under the canonical isomorphism of bundles $\tilde{t}^{*}T(TM)\cong T(t^{*}TM)$, the splitting $\tilde{\tau}$ is defined by
$$\tilde{\tau}:=J_\G\circ T\tau,$$
where $J_\G:TT\G\to TT\G$ is the canonical involution of the double tangent bundle.

We denote by  $\tilde{X}_{\tilde{c}}$ the vector field induced by the splitting $\tilde{\tau}$ and the deformation cochain $\tilde{c}\in C_\mathrm{def}^{k}(T\G)$.

A direct computation now shows that there is a compatibility between $\tilde{X}_{\tilde{c}}$ and the tangent lift $X_c^{T}\in\mathfrak{X}(T\G^{(k)})$ of the vector field $X_c$. Precisely, one obtains
$$X_{c}^{T}=\tilde{X}_{T_\mathrm{def}c}.$$

Finally, we are ready to check the commutativity of the diagram \eqref{diagrama:deRham-mapping} above. Let $c\in C^{k}_\mathrm{def}(\G)$ and $\tilde{pr}_1:(T\G)^{(k)}\to T\G$ be the projection on the first component. Then

\begin{align*}
\hat{\omega}_{*}\circ\oplus^{2}T_\mathrm{def}(c)&=d\hat{\omega}(T_\mathrm{def}c, T_\mathrm{def}c)\\
&=d\hat{\omega}(d\tilde{pr}_1\tilde{X}_{T_\mathrm{def}c}, d\tilde{pr}_1\tilde{X}_{T_\mathrm{def}c})\\
&=d\hat{\omega}(d\tilde{pr}_1 X^{T}_{c}, d\tilde{pr}_1 X^{T}_{c}).
\end{align*}

Let $F_\e$ be the flow of the vector field $X_c$. Then $TF_\e$ will be the flow of the tangent lift $X^{T}_{c}$ and the expression above becomes

\begin{align*}
d\hat{\omega}(d\tilde{pr}_1 X^{T}_{c}, d\tilde{pr}_1 X^{T}_{c})&=\frac{d}{d\e}\Bigr|_{\e=0}\ \hat{\omega}(\tilde{pr}_1TF_\e, \tilde{pr}_1TF_\e)\\
&=\frac{d}{d\e}\Bigr|_{\e=0}\ F_\e^{*}(pr_1)^{*}\omega\\
&=\mathcal{L}_{X_c}(pr_1)^{*}\omega\\
&=d_{dR}(i_{X_c}pr_1^{*}\omega),
\end{align*}

which proves the commutativity of the diagram.
\end{proof}

\begin{corollary}\label{corollary:deform.complex}
The deformation complex $C^{\bullet}_\mathrm{def}(\G,\omega)$ of the symplectic groupoid $(\G,\omega)$ is isomorphic to the mapping cone of the de Rham differential \[-d_{dR}:\Omega^{1}_\mathrm{def}(\G^{(\bullet)})\to\Omega_\mathrm{cl}^{2}(\G^{(\bullet)}).\]
Explicitly, according to our convention for the mapping cone complex, the differential $D:\Omega_\mathrm{def}^{1}(\calG^{(\bullet)})\oplus\Omega_\mathrm{cl}^{2}(\calG^{(\bullet-1)})\to\Omega_\mathrm{def}^{1}(\calG^{(\bullet+1)})\oplus\Omega_\mathrm{cl}^{2}(\calG^{(\bullet)})$ is given by
$$(\zeta,\omega)\mapsto (\delta\zeta,-d_{dR}\zeta-\delta\omega).$$ We call this complex the \textbf{de Rham model for deformation cohomology} of the symplectic groupoid $(\G,\omega)$.
\end{corollary}

\begin{remark}\label{Rem_sub_BSS}
Notice that $d_{dR}:\Omega^{1}(\G^{(\bullet)})\to\Omega^{2}(\G^{(\bullet)})$ is a part of the vertical differential of the Bott-Shulman-Stasheff double complex $(\Omega^{\bullet}(\G^{(\bullet)}),\delta,d_{dR})$. This allows us to identify the deformation complex $C^{\bullet}_\mathrm{def}(\G,\omega)$ with the total complex of a sub-double complex $B^{\bullet,\bullet}(\G)\subset\Omega^{\bullet}(\G^{(\bullet)})$. This sub-double complex consists of the subset $\Omega^{1}_\mathrm{def}(\G^{(\bullet)})$ of the second line $\Omega^{1}(\G^{(\bullet)})$ and of the subset $\Omega^{2}_\mathrm{cl}(\G^{(\bullet)})$ of the third line $\Omega^{2}(\G^{(\bullet)})$ of the Bott-Shulman-Stasheff double complex (and it is zero in all other bidegrees). In fact, the map
$$\varphi:\Omega^{1}_\mathrm{def}(\calG^{(n)})\oplus\Omega_\mathrm{cl}^{2}(\calG^{(n-1)})\to Tot^{n+1}(B^{\bullet,\bullet}(\G))$$
defined by
$$(\zeta,\omega)\mapsto (\zeta,(-1)^n\omega)$$
gives us the isomorphism of the complexes. The shift in degree appears naturally since multiplicative symplectic forms, while being degree two forms, give rise to degree three cocycles in the Bott-Shulman-Stasheff cohomology. More generally, a groupoid $\G$ equipped with a pair $(\phi, \omega)\in \Omega^2(\G)\oplus \Omega^3(M)$ such that $0 \oplus 0 \oplus \omega \oplus\phi \in Tot^3(\Omega^\bullet(\G^{(\bullet)}))$ is a cocycle is called a $\phi$-twisted presymplectic groupoid \cite{BCWZ}.

Another interpretation for Bott-Shulman-Stasheff cohomology is that it is the de Rham cohomology of the differentiable stack $M//\G$ presented by $\G$. We note also that degree three cocycles belonging to $B^{\bullet,\bullet}$ i.e.\ corresponding to deformation cocycles of $(\G,\omega)$, and satisfying an integrality condition are related to $S^1$-central extensions of $\G$ \cite[Thm. 5.1]{B-Xu}.

\end{remark}

\section{Examples and computations} 

\subsection{Deformation cohomology in small degrees}\label{subsec-small-degrees}

We start by interpreting cohomology classes in small degrees in a geometric way via direct computations; but see also the computations of small degrees in terms of a spectral sequence, in Section \ref{Sec-spectral-sequence-rows}. Using the de Rham model for the deformation cohomology, a simple computation shows that \[H_\mathrm{def}^0(\G, \omega)=\{ \alpha\in \Omega^1_\mathrm{def}(M)\mid  d\alpha=0, \ \delta\alpha=0\}=\Omega^1_\mathrm{cl}(M)^{\G-\mathrm{inv}},\] i.e.\ the space of closed invariant $1$-forms on $M$. 

In degree 1, again using the de Rham model, a deformation cocycle on $(\G,\omega)$ is given by a pair $(\zeta, \alpha)\in \Omega^1_\mathrm{def}(\G)\oplus \Omega^2_\mathrm{cl}(M)$ such that $\delta{\zeta}=0$ and $d\zeta = \delta \alpha$. Using the terminology of \cite{BCWZ, Cattaneo_Xu}, this means that $\zeta$ is a relatively $\alpha$-closed multiplicative $1$-form. 

Under the isomorphism $\omega^\flat: T\G\to T^*\G$, the form $\zeta$ corresponds to a vector field $X$ on $\G$ such that $-\omega^\flat(X)=\zeta$.
Multiplicativity of $\zeta$ and of $\omega$ mean that $X$ is  multiplicative as well, meaning that $X$ is a Lie groupoid morphism $X:\G\to T\G$. Therefore, $X$ is an infinitesimal automorphism of $\G$, its flow preserving the Lie groupoid structure. However, since $\zeta$ is not closed but only relatively $\alpha$-closed, the flow $\phi_\e$ of $X$ does not preserve the symplectic form $\omega$. Instead, it satisfies $\phi_\e^*\omega=\omega+\delta(\e\alpha)$, so it  applies a gauge transformation by $\delta(\e \alpha)$ to $\zeta$. Because of this we call 1-cocycles \textbf{twisted infinitesimal automorphisms of $(\G,\omega)$}.

A degree 1 coboundary is a pair of the form $(\delta\beta, d\beta)\in \Omega^1_\mathrm{def}(\G)\oplus \Omega^2_\mathrm{cl}(M)$. Adding it to a relatively $\alpha$-closed multiplicative $1$-form $\zeta$ we obtain a relatively $(\alpha+d\beta)$-closed multiplicative $1$-form $\zeta+\delta\beta$. We call these transformations of the form $(\zeta,\alpha)\mapsto (\zeta + \delta\beta, \alpha+d\beta)$ \textbf{trivial gauge transformations}. We conclude that

 \[H^1_\mathrm{def}(\G,\omega)=\frac{ \text{Twisted infinitesimal automorphisms of }(\G,\omega)}{\text{Trivial gauge transformations}}.\]

In degree 2, as mentioned in Remark \ref{remark:cocycle}, deformations of symplectic groupoids give rise, after taking a first-order approximation, to cocycles in $C^2_\mathrm{def}(\G,\omega)$; equivalences between deformations provide 1-cochains transgressing the difference between the corresponding deformation cocycles. Thus

 \[H^2_\mathrm{def}(\G,\omega)=\frac{ \text{Infinitesimal deformations of }(\G,\omega)}{\text{Infinitesimal equivalences}}.\] 
 
\noindent This is the main use of the deformation cohomology, and we will study it in detail in Sections \ref{section:cocycle} and \ref{section:general}.

\subsection{The cone exact sequence and proper Lie groupoids}\label{sec-proper}

Before we start analysing the deformation cohomology of a symplectic groupoid in particular examples, we recall a simple consequence of the definition of the deformation complex. As the mapping cone of $\hat{\omega}_{*}\circ\oplus^{2}T_\mathrm{def}:C^{\bullet}_\mathrm{def}(\G)\to C_\mathrm{def}^{\bullet}(\omega)$, it fits in a short exact sequence of complexes

\[0\to C_\mathrm{def}^{\bullet}(\omega)[-1] \to C^{\bullet}_\mathrm{def}(\G, \omega) \to C^{\bullet}_\mathrm{def}(\G) \to 0.\] This induces the long exact sequence in cohomology 
\begin{equation}\label{eq-long-ex-seq-cone}
\cdots \to H_\mathrm{def}^{k-1}(\omega) \to H^{k}_\mathrm{def}(\G, \omega) \to H^{k}_\mathrm{def}(\G)\stackrel{\partial}{\to} H_\mathrm{def}^{k}(\omega)\to \cdots
\end{equation} where the connecting homomorphism $\partial=H(\hat{\omega}_{*}\circ\oplus^{2}T_\mathrm{def})$ is the map induced in cohomology by $\hat{\omega}_{*}\circ\oplus^{2}T_\mathrm{def}$.

\begin{example}[Symplectic manifolds] 

A Lie groupoid $\G$ integrating a connected non-degenerate Poisson manifold $(M,\omega^{-1})$ is transitive. For a transitive Lie groupoid it holds \cite[Prop. 3.1]{CMS} that $H^\bullet_\mathrm{def}(\G)\cong H^\bullet(G,\mathfrak{g})$, where $G$ is the isotropy Lie group at any point of the base, acting on its Lie algebra $\mathfrak{g}$ via the adjoint representation. By simple dimension counting we know that the isotropy groups are discrete, so $\mathfrak{g}=0$.

Thus, for a symplectic groupoid $(\G, s^*\omega-t^*\omega)\rightrightarrows(M,\omega^{-1})$, the long exact sequence (\ref{eq-long-ex-seq-cone}) implies that \[H_\mathrm{def}^{\bullet}(\G, s^*\omega-t^*\omega)\cong H_\mathrm{def}^{\bullet-1}(s^*\omega-t^*\omega).\]

For the pair groupoid of $M$, any multiplicative $2$-form is cohomologically trivial, so $H_\mathrm{def}^2(M\times M, pr_1^*\omega-pr_2^*\omega)=0$. This translates to the fact that deformations of $(M\times M, pr_1^*\omega-pr_2^*\omega)$ are given by gauge transformation of $\delta\omega=pr_1^*\omega-pr_2^*\omega$ by $\delta \omega_\e$, where $\omega + \omega_\e$ is a deformation of the symplectic structure on the base.
\end{example}

Another situation in which we can make further use of this long exact sequence is for proper symplectic groupoids. 
A Lie groupoid is called \textbf{proper} if it is Hausdorff and \[(s,t):\G\to M\times M\] is a proper map. Examples of proper groupoids include compact Lie groupoids, pair groupoids, and the cotangent Lie groupoid $T^*G$ of a compact Lie group $G$.

Proper symplectic groupoids are the subject of recent extensive work \cite{PMCT1, PMCT2, Zung}, and satisfy many remarkable properties not shared by a general symplectic groupoid, and much less by non-integrable Poisson manifolds. They generalize many aspects of the geometry of compact symplectic manifolds and of Lie algebras of compact type. 

Proper symplectic groupoids are also very special among proper Lie groupoids, in a way which is analogous to how the adjoint action of a compact Lie group $G$ on its Lie algebra $\mathfrak{g}$ is special among actions of $G$ (compare for example \cite{PMCT1, PMCT2} with Chapters 2 and 3 of \cite{DK}). For example, for an arbitrary Lie groupoid $\G \rightrightarrows M$ the isotropy Lie algebra $\mathfrak{i}_x$ and the normal space $\nu_x$ to the orbit of $\G$ through $x\in M$ are the kernel and cokernel of the anchor map of $A=\mathrm{Lie}(\G)$ at $x$. In general there is no further relation between these two objects, which feature prominently in the study of the geometry of proper Lie groupoids; when $\G$ is a symplectic groupoid, however, $\nu_x$ inherits a linear Poisson structure from $M$ by linearization at $x$, such that its dual Lie algebra $\nu_x^*$ is naturally isomorphic to $\mathfrak{i}_x$.

We recall the following vanishing result for the deformation cohomology of proper Lie groupoids.

\begin{theorem}[\cite{CMS}, Thm 6.1]\label{thm-proper-vanishing}
         If $\G$ is a proper Lie groupoid, then  \[H^0_\mathrm{def}(\G)\cong\Gamma(\mathfrak{i})^{\G-\mathrm{inv}},\ \ \ H^1_\mathrm{def}(\G)\cong\Gamma(\nu)^{\G-\mathrm{inv}},\ \text{ and }\  H^{k}_\mathrm{def}(\G)=0 \text{ for } k\geq 2.\]
\end{theorem} 

Let us give a brief explanation of the objects in the statement of this result. The isotropy bundle $\mathfrak{i}$ is the kernel of the anchor map of the Lie algebroid of $\G$, \[\mathfrak{i}=\mathrm{Ker}(\rho:A\to TM).\] Differentiating at the units the conjugation by elements $g:x\to y$ induces a map $\mathrm{ad}_g:\mathfrak{i}_x\to \mathfrak{i}_y$.

Strictly speaking, $\mathfrak{i}$ is a distribution with possibly varying rank, so it is only a bundle if $\G$ is a regular groupoid. Nonetheless, it makes sense to consider $\Gamma(\mathfrak{i})$ as the space of smooth sections of $A$ which land in $\mathfrak{i}\subset A$. Among these, the space of invariant ones is \[\Gamma(\mathfrak{i})^{\G-\mathrm{inv}}=\{\alpha\in \Gamma(\mathfrak{i})\ |\ \mathrm{ad}_g(\alpha_x)=\alpha_y\}.\] It was shown in \cite{CMS} that in fact $H^0_\mathrm{def}(\G)\cong\Gamma(\mathfrak{i})^{\G-\mathrm{inv}}$ for any Lie groupoid, not only proper ones.

The normal bundle of $\G$ is \[\nu=\mathrm{Coker(\rho)}=TM/\rho(A),\] and once again it is only a smooth vector bundle for regular $\G$. But once more, there is a way to make sense of its space of smooth invariant sections $\Gamma(\nu)^{\G-\mathrm{inv}}$. First of all, the space of smooth sections is defined as $\Gamma(\nu):=\mathfrak{X}(M)/\rho(A)$. A section $[V]\in \Gamma(\nu)$ is called invariant if there is a vector field $X\in \mathfrak{X}(\G)$ which is both $s$ and $t$-projectable to $V$. This condition is independent of the choice of representative $V$. In the regular case, this notion recovers the usual notion of invariant sections of the normal representation of $\G$. We refer to Section $4.4$ of \cite{CMS} for more detail.

For an arbitrary Lie groupoid $\G$, deformation cocycles of degree one are multiplicative vector fields on $\G$, i.e., vector fields which are Lie groupoid morphisms $X:\G\to T\G$. Any such vector field $X$ is in particular $s$ and $t$-projectable to a vector field $V$ on the base. Mapping $X$ to the class of $V$ modulo $\mathrm{Im}(\rho)$ induces a linear map $H^1_\mathrm{def}(\G)\to\Gamma(\nu)^{\G-\mathrm{inv}}$ \cite[Lemma 4.7]{CMS}. Part of Theorem \ref{thm-proper-vanishing}  is that for proper Lie groupoids this map becomes an isomorphism.

\begin{proposition}\label{prop-symp-prop}
         Let $(\G, \omega)$ be a proper symplectic groupoid. Then there is a 7-term exact sequence
\begin{align*}
   0  & \to \Omega^1_\mathrm{cl}(M)^{\G-\mathrm{inv}} \to \Gamma(\nu^*)^{\G-\mathrm{inv}} \to \Omega^2_\mathrm{cl}(M)^{\G-\mathrm{inv}} \to H^{1}_\mathrm{def}(\G, \omega) \to \\
  &  \to \Gamma(\nu)^{\G-\mathrm{inv}}\to H_\mathrm{def}^{1}(\omega)\to H^{2}_\mathrm{def}(\G, \omega)\to 0
  \end{align*}       
  and $H^{k}_\mathrm{def}(\G,\omega)\cong H^{k-1}_\mathrm{def}(\omega)$ for $k> 2$.
  
In particular, in terms of the de Rham model for deformation cohomology, \[H^{2}_\mathrm{def}(\G, \omega)\cong \frac{H_\mathrm{def}^{1}(\omega)}{d(H^1_\mathrm{def}(\G))}\cong \frac{H_\mathrm{def}^{1}(\omega)}{d[\Omega^1_\mathrm{mult}(\G)/\delta(\Omega^1(M))]}.\]
\end{proposition}

\begin{proof}
Simply use Theorem \ref{thm-proper-vanishing} together with the long exact sequence (\ref{eq-long-ex-seq-cone}), the discussion of Section \ref{subsec-small-degrees} for deformation cohomology in small degrees, and recall that $\mathfrak{i}_x\cong  \nu_x^*$ for symplectic groupoids.
\end{proof}

\subsection{Zero Poisson structures of proper type}\label{sec-zero-proper}

The source $1$-connected integration of the zero Poisson structure on a manifold $M$ is given by the symplectic groupoid $(T^*M, \omega_\mathrm{can})$, viewed as a bundle of abelian Lie groups over $M$, and so it is never proper. However, $(M,0)$ may still be integrated by a non simply connected proper symplectic groupoid, having interesting consequences for the geometry of $M$. Proper $s$-connected symplectic integrations of $(M,0)$ are in 1:1 correspondence with integral affine structures on $M$ (\cite[Prop 4.2]{PMCT1}, see also \cite[3.1.6]{PMCT2}).

An integral affine structure on $M$ can be described by a Lagrangian smooth full rank lattice $\Lambda$ in $T^*M$ i.e., a subbundle $\Lambda\subset T^*M$, whose fibres are full rank lattices in the fibres of $T^*M$, and which is a Lagrangian submanifold of $T^*M$. The proper symplectic integration corresponding to $\Lambda$ is $(T^*M, \omega_\mathrm{can})/\Lambda$. The Lie groupoid $\mathcal{T}_\Lambda:=T^*M/\Lambda$ is a torus bundle and the form $\omega_\mathrm{can}$ descends to a symplectic form $\omega_\mathcal{T}$ on it, so $(\mathcal{T}_\Lambda,\omega_\mathcal{T})$ is a symplectic torus bundle over $M$.

Since $(\mathcal{T}_\Lambda,\omega_\mathcal{T})$ is a bundle of Lie groups, its source and target maps coincide, so using the de Rham model for deformation cohomology we see that \[H_\mathrm{def}^0(\mathcal{T}_\Lambda, \omega_\mathcal{T})=\Omega^1_\mathrm{cl}(M)^{\mathcal{T}_\Lambda-\mathrm{inv}}=\Omega^1_\mathrm{cl}(M).\] Moreover, \[H_\mathrm{def}^1(\mathcal{T}_\Lambda)\cong \Omega^1_\mathrm{mult}(\mathcal{T}_\Lambda)/\delta(\Omega^1(M))=\Omega^1_\mathrm{mult}(\mathcal{T}_\Lambda)\] and \[H_\mathrm{def}^1(\omega_\mathcal{T})= \Omega^2_\mathrm{cl,mult}(\mathcal{T}_\Lambda)/\delta(\Omega_\mathrm{cl}^2(M))=\Omega^2_\mathrm{cl, mult}(\mathcal{T}_\Lambda).\]

Therefore, in this case the 7-term sequence from Proposition \ref{prop-symp-prop} becomes
\begin{align*}
   0  & \to \Omega^1_\mathrm{cl}(M) \to \Omega^1(M) \stackrel{d}{\to} \Omega^2_\mathrm{cl}(M) \to H^{1}_\mathrm{def}(\mathcal{T}_\Lambda, \omega_\mathcal{T}) \to \\
  &  \to \Omega^1_\mathrm{mult}(\mathcal{T}_\Lambda)\stackrel{d}{\to} \Omega^2_\mathrm{cl, mult}(\mathcal{T}_\Lambda)\to H^{2}_\mathrm{def}(\mathcal{T}_\Lambda,\omega_\mathcal{T})\to 0
  \end{align*}       
  and $H^{k}_\mathrm{def}(\mathcal{T}_\Lambda,\omega_\mathcal{T})\cong H^{k-1}_\mathrm{def}(\omega_\mathcal{T})$ for $k> 2$.
  
In particular, \[H^{2}_\mathrm{def}(\mathcal{T}_\Lambda,\omega_\mathcal{T})\cong H^2(\Omega^\bullet_\mathrm{mult}(\mathcal{T}_\Lambda), d_{dR}),\]
the second de Rham cohomology of the subcomplex of multiplicative forms on $\mathcal{T}_\Lambda$.

\subsection{Linear Poisson manifolds of proper type}\label{sec-lin-proper}

The Lie groupoid $T^*G\cong G\ltimes \mathfrak{g}^*$ integrating the linear Poisson manifold $(\mathfrak{g}^*,\pi_\mathfrak{g})$ is proper if and only if the Lie group $G$ is compact. In this case, we know from Proposition \ref{prop-symp-prop} that \[H^{2}_\mathrm{def}(T^*G, \omega_\mathrm{can})\cong \frac{H_\mathrm{def}^1(\omega_\mathrm{can})}{d[\Omega^1_\mathrm{mult}(T^*G)/\delta(\Omega^1(\mathfrak{g}^*))]}.\]
But since $\mathfrak{g}^*$ is contractible, $H_\mathrm{dR}^2(\mathfrak{g}^*)=0$, so $\delta \Omega_\mathrm{cl}^2(\mathfrak{g}^*)=\delta d\Omega^1(\mathfrak{g}^*)=d \delta\Omega^1(\mathfrak{g}^*)$ and therefore

\[H^{2}_\mathrm{def}(T^*G, \omega_\mathrm{can})\cong \frac{\Omega^2_\mathrm{mult, cl}(T^*G)/\delta \Omega^2_\mathrm{cl}(\mathfrak{g}^*)}{d[\Omega^1_\mathrm{mult}(T^*G)/\delta(\Omega^1(\mathfrak{g}^*))]}=\frac{\Omega^2_\mathrm{mult, cl}(T^*G)/d \delta\Omega^1(\mathfrak{g}^*)}{d[\Omega^1_\mathrm{mult}(T^*G)/\delta(\Omega^1(\mathfrak{g}^*))]},\]
which is the same as $H^2(\Omega^\bullet_\mathrm{mult}(T^*G), d_{dR})$.

\subsection{Spectral sequence computations for deformation cohomology I}\label{Sec-spectral-sequence-rows}

For computational purposes, it is useful to use the identification of the deformation complex of $(\G,\omega)$ with the total complex of the very simple double complex

\[\begin{tikzcd}
0    & 0     & 0    & 0    &        \\
\Omega^{2}_\mathrm{cl}(M) \arrow[u] \arrow[r, "\delta"]                               & \Omega^{2}_\mathrm{cl}(\G) \arrow[u] \arrow[r, "\delta"]                                         & \Omega^{2}_\mathrm{cl}(\mathcal{G}^{(2)}) \arrow[u] \arrow[r, "\delta"]                                               & \Omega^{2}_\mathrm{cl}(\mathcal{G}^{(3)}) \arrow[u] \arrow[r]                                               & \cdots \\
\Omega^{1}_\mathrm{def}(M) \arrow[u, "d"] \arrow[r, "\delta"] & \Omega^{1}_\mathrm{def}(\mathcal{G}) \arrow[u, "d"] \arrow[r, "\delta"] & \Omega^{1}_\mathrm{def}(\mathcal{G}^{(2)}) \arrow[u, "d"] \arrow[r, "\delta"] & \Omega^{1}_\mathrm{def}(\mathcal{G}^{(3)}) \arrow[u, "d"] \arrow[r] & \cdots
\end{tikzcd}\]
as per definition of the mapping cone, and detailed in Remark \ref{Rem_sub_BSS}. We recall again that the inclusion $\Omega^{1}_\mathrm{def}(\G^{(\bullet)})\subset \Omega^{1}(\G^{(\bullet)})$ is a quasi-isomorphism. As such, we can also compute deformation cohomology of a symplectic groupoid using the cone of $-d_{dR}:\Omega^{1}(\G^{(\bullet)})\to\Omega_\mathrm{cl}^{2}(\G^{(\bullet)})$ instead of $-d_{dR}:\Omega^{1}_\mathrm{def}(\G^{(\bullet)})\to\Omega_\mathrm{cl}^{2}(\G^{(\bullet)})$. For the ``rows first'' spectral sequence associated to the corresponding double complex  \cite[Sec. III.14]{Bott-Tu}, the second page $E_2$ is 

%\begin{small}
\begin{tikzpicture}
  \matrix (m) [matrix of math nodes,
    nodes in empty cells,nodes={minimum width=5ex,
    minimum height=5ex,outer sep=-5pt},
    column sep=1ex,row sep=1ex]{
                &      &     &     &    &  \\
          1     &  H_\mathrm{dR}(\Omega^2_\mathrm{cl}(M)^{\G-\mathrm{inv}}) &  H_\mathrm{dR}H_\delta(\Omega^{2}_\mathrm{cl}(\G))  & H_\mathrm{dR}H_\delta(\Omega^{2}_\mathrm{cl}(\G^{(2)})) & \cdots & \\
          0     &  \Omega^1_\mathrm{cl}(M)^{\G-\mathrm{inv}} &  (H^1_\mathrm{def}(\G))_\mathrm{cl}  & (H^2_\mathrm{def}(\G))_\mathrm{cl}  &  \cdots & \\
    \quad\strut &   0  &  1  &  2  & \strut \\};
%  \draw[-stealth] (m-3-2.north) -- (m-2-2.south);
 %   \draw[-stealth] (m-3-3.north) -- (m-2-3.south);
  %\draw[-stealth] (m-3-4.north) -- (m-2-4.south);
\draw[thick] (m-1-1.east) -- (m-4-1.east) ;
\draw[thick] (m-4-1.north) -- (m-4-5.north) ;
\end{tikzpicture}
%\end{small}

\noindent where $(H^n_\mathrm{def}(\G))_\mathrm{cl}$ denotes $\mathrm{Ker}(d)\subset H_\delta(\Omega^{1}(\G^{(n)}))$. Let us also use the notation $H_{dR}^2(M)^{\G-\mathrm{inv}}=H^2(\Omega^\bullet(M)^{\G-\mathrm{inv}},d_{dR})$. Since the double complex was concentrated in the first two rows, the spectral sequence stabilizes at this page. Therefore $H_\mathrm{def}^k(\G,\omega)\cong E_2^{k,0}\oplus E_2^{k-1,1}$ as vector spaces, and more explicitly \begin{equation}\label{eq:E_2}
E_{2}^{k-1,1}=\frac{\Omega^{2}_\mathrm{cl}(\G^{(k-1)})_{\delta-\mathrm{cl}}}{\delta(\Omega^{2}_\mathrm{cl}(\G^{(k-2)}))+d_{dR}(\Omega^{1}(\G^{(k-1)})_{\delta-\mathrm{cl}})}.
\end{equation} When we know more about $\G$, it is possible to have a better description of these pieces. For example, either equation \eqref{eq:E_2} for $k=2$ or the same argument used for linear Poisson manifolds of proper type (Section \ref{sec-lin-proper}) shows the following.

\begin{proposition}\label{rows-spectral-seq} If $\G\rightrightarrows M$ is a symplectic groupoid and $H^2_\mathrm{dR}(M)=0$, then \[H^2_\mathrm{def}(\G,\omega)\cong H^2_\mathrm{mult}(\G)\oplus (H^2_\mathrm{def}(\G))_\mathrm{cl}.\] If additionally $\G$ is proper, then \[H^2_\mathrm{def}(\G,\omega)\cong H^2_\mathrm{mult}(\G).\]
\end{proposition}

By the previous discussion of this section and of Section \ref{subsec-small-degrees} we can also compute the cohomology in degrees zero and one, so for the sake of completeness, we recall that $H_\mathrm{def}^0(\G,\omega)=\Omega^1_\mathrm{cl}(M)^{\G-\mathrm{inv}}$ and $H^1_\mathrm{def}(\G,\omega)\cong H_{dR}^2(M)^{\G-\mathrm{inv}}\oplus (H^1_\mathrm{def}(\G))_\mathrm{cl}$.

\subsection{Spectral sequence computations for deformation cohomology II}\label{Sec-spectral-sequence-column}

Let us now use the ``columns first'' spectral sequence for the double complex associated to the mapping cone of $-d_{dR}:\Omega_\mathrm{def}^{1}(\G^{(\bullet)})\to\Omega_\mathrm{cl}^{2}(\G^{(\bullet)})$. Its first page is

\begin{small}
\begin{tikzpicture}
  \matrix (m) [matrix of math nodes,
    nodes in empty cells,nodes={minimum width=5ex,
    minimum height=5ex,outer sep=2pt},
    column sep=3ex,row sep=1ex]{
                &      &     &     &  & \\
          1     &  H^2_{dR}(M) &  \Omega^{2}_\mathrm{cl}(\G)/d_{dR}(\Omega^1_\mathrm{def}(\G))  & \Omega^{2}_\mathrm{cl}(\G^{(2)})/d_{dR}(\Omega^1_\mathrm{def}(\G^{(2)}))  & \cdots &\\
          0     &  \Omega^{1}_\mathrm{def}(M)_\mathrm{cl}  &  \Omega^{1}_\mathrm{def}(\G)_\mathrm{cl}  & \Omega^{1}_\mathrm{def}(\G^{(2)})_\mathrm{cl}  & \cdots &\\
    \quad\strut &   0  &  1  &  2  & 3 & \strut \\};
  \draw[-stealth] (m-2-2.east) -- (m-2-3.west);
    \draw[-stealth] (m-2-3.east) -- (m-2-4.west);
%    \draw[-stealth] (m-2-4.east) -- (m-2-5.west);
    \draw[-stealth] (m-3-2.east) -- (m-3-3.west);
    \draw[-stealth] (m-3-3.east) -- (m-3-4.west);
%    \draw[-stealth] (m-3-4.east) -- (m-3-5.west);
\draw[thick] (m-1-1.east) -- (m-4-1.east) ;
\draw[thick] (m-4-1.north) -- (m-4-5.north) ;
\end{tikzpicture}
\end{small}
where $\Omega^1_\mathrm{def}(\G^{(k)})_\mathrm{cl}$ denotes the kernel of $d_{dR}:\Omega_\mathrm{def}^{1}(\G^{(k)})\to\Omega_\mathrm{cl}^{2}(\G^{(k)})$.
This spectral sequence stabilizes at the third page, and so $H_\mathrm{def}^k(\G,\omega)\cong E^{k,0}_3\oplus E^{k-1,1}_3$, where
\[E^{k,0}_3=H^{k}_{\delta_{\mathrm{def}}}(\Omega^{1}_{\mathrm{def}}(\G^{\bullet})_\mathrm{cl})\]
and
\[E^{k-1,1}_3=\mathrm{Ker}\left(d^{k-1,1}_2:H^{k-1}_{\delta}\left(\frac{\Omega^{2}_\mathrm{cl}(\G^{(k-1)})}{d_{dR}(\Omega^1_\mathrm{def}(\G^{(k-1)}))}\right)\to H^{k+1}_{\delta_{\mathrm{def}}}(\Omega^{1}_{\mathrm{def}}(\G^{\bullet})_\mathrm{cl})\right).\]

Again, if we know more about $\G$, these formulas become simpler. If $\G$ is proper, we can make use of the following variation of Theorem \ref{thm-proper-vanishing}.

\begin{proposition} If $(\G,\omega)$ is a proper Lie groupoid, then \[H_\delta^k(\Omega_\mathrm{def}(\G^{(\bullet)})_\mathrm{cl})=0\] for all $k\geq 2$.
\end{proposition}
\begin{proof}We will first rewrite the homotopy operator of $C^\bullet_\mathrm{def}(\G)$ in terms of the de Rham model $\Omega^{1}_\mathrm{def}(\G^{(\bullet)})$; we then check that it restricts to a homotopy operator for the subcomplex $\Omega^{1}_\mathrm{def}(\G^{(\bullet)})_\mathrm{cl}$. We recall that the homotopy operator is described in terms of a Haar system and a cut-off function, which exist for any proper Lie groupoid, and permit integration along the target fibres. We refer to \cite{CMS} for details and for the proof of Theorem \ref{thm-proper-vanishing} using this operator.

Let $k\geq 2$ and let $\beta\in\Omega^{1}_\mathrm{def}(\G^{(k+1)})_{\delta-\mathrm{cl}}$. The homotopy operator on the de Rham model $\Omega^{1}_\mathrm{def}(\G^{(\bullet)})$ assigns to $\beta$ the transgression $\alpha\in \Omega^{1}_\mathrm{def}(\G^{(k)})$ defined by
$$\alpha(X^{1}_{g_1},...,X^{k}_{g_k})=\int_{s(g_k)}\beta(X^{1}_{g_1},...,X^{k}_{g_k},\tau_{h}(ds(X^{k}_{g_k})))dh;$$
where $\tau:t^{*}TM\to T\G$ is a splitting of the target map of $\G$. Notice that, since $\beta\in\Omega^{1}_\mathrm{def}(\G^{(k+1)})$, the expression in the integral actually does not depend on the value of the tangent vectors at $h\in\G$ (cf. Remark \ref{remark:explicite isomorph}). We remark that this expression can also be obtained by using the homotopy operator for VB-cohomology as in \cite{CD}. Equivalently, if $X\in\mathfrak{X}(\G^{(k)})$ the homotopy operator has the form
$$\alpha(X):=\int_{t\text{-fibres}}\beta(X^{(k+1)})d\mu_\G,$$
where $X^{(k+1)}\in\mathfrak{X}(\G^{(k+1)})$ is the vector field
$$X^{(k+1)}(g_1,...,g_{k+1}):=(X(g_1,...,g_{k}),\tau_{g_{k+1}}(ds(dpr_{k}X(g_1,...,g_{k}))))$$
induced by $X$. 

We now prove that the transgression $\alpha$ of a closed 1-form $\beta\in \Omega^{1}_\mathrm{def}(\G^{(\bullet)})$ is a closed 1-form.
Recall that the de Rham differential of 1-forms can be written as
\begin{equation*}
\begin{split}
d\alpha(X,Y)&=-\alpha([X,Y])+X(\alpha(Y))-Y(\alpha(X))\\
&=L_X\alpha(Y)-L_Y\alpha(X)-\alpha(L_{X}Y).
\end{split}
\end{equation*}

 Let $\beta\in\Omega^{1}_\mathrm{def}(\G^{(k+1)})_{\delta-\mathrm{cl}}$ be such that $d\beta=0$, and let $X,Y\in\mathfrak{X}(\G^{(k)})$. Then,

\begin{align*}
(L_X\alpha)(Y)&=\left(\frac{d}{d\e}\Bigr|_{\e=0}(\varphi^{X}_\e)^{*}\alpha\right)(Y)\\
&=\frac{d}{d\e}\Bigr|_{\e=0}\alpha((\varphi^{X}_\e)_{*}Y)\circ\varphi^{X}_\e\\
&=\int_{t\text{-fibres}}\frac{d}{d\e}\Bigr|_{\e=0}\beta((\varphi^{X^{(k+1)}}_\e)_{*}Y^{(k+1)})\circ\varphi^{X^{(k+1)}}_\e d\mu_\G\ \ \ \\ %(\text{because } \beta\in\Omega^{1}_\mathrm{def}(\G^{(k+1)}))\\
&=\int_{t\text{-fibres}}L_{X^{(k+1)}}\beta(Y^{(k+1)})d\mu_\G,
\end{align*}
using that $\beta\in\Omega^{1}_\mathrm{def}(\G^{(k+1)})$ to pass from the second to the third lines. Similarly we can compute $(L_Y\alpha)(X)$. And finally,
\begin{align*}
\alpha([X,Y])&=\int_{t\text{-fibres}}\beta([X,Y]^{(k+1)})d\mu_\G\\
&=\int_{t\text{-fibres}}\beta(L_{X^{(k+1)}}Y^{(k+1)})d\mu_\G\ \ (\text{because }\beta\in\Omega^{1}_\mathrm{def}(\G^{(k+1)})).
\end{align*}
Therefore, with these computations we get that 
$$d\alpha(X,Y)=\int_{t\text{-fibres}}d\beta(X^{(k+1)},Y^{(k+1)})=0,$$
which proves the claim.
\end{proof}

This result, together with the formulas for $E^{k,0}_3$ and for $E^{k-1,1}_3$ from the beginning of this section, lead to a simpler description of $H_\mathrm{def}^k(\G,\omega)\cong E^{k,0}_3\oplus E^{k-1,1}_3$.

\begin{corollary}
Let $\G$ be a proper Lie groupoid. Then for $k\geq 2$ \[H_\mathrm{def}^k(\G,\omega)\cong H^{k}_{\delta}\left(\frac{\Omega^{2}_\mathrm{cl}(\G^{(k-1)})}{d_{dR}(\Omega^1_\mathrm{def}(\G^{(k-1)}))}\right).\]
\end{corollary}

\section{The deformation cocycle of $s$-constant deformations, and Moser tricks}\label{section:cocycle}

In this section we describe the deformation cocycle mentioned in remark \ref{remark:cocycle} by using the simpler presentation of the deformation complex given by corollary \ref{corollary:deform.complex}.
We then use such a cocycle to state a result analogous to the classical Moser Theorem in the context of symplectic groupoids. We first carry out this work for $s$-constant deformations, where the deformation cocycles involved are canonically determined. In the case of  general deformations it is only the deformation cohomology class that will be well-defined. We consider that case in Section \ref{section:general}.

\subsection{The deformation class}

Let $(\calG_\e,\omega_\e)$ be a deformation of $(\calG,\omega)$. In remark \ref{remark:cocycle} we associate the cocycle $\xi_\e-\frac{d}{d\e}\omega_\e\in C^{2}_\mathrm{def}(\G,\omega)$ to the deformation $(\calG_\e,\omega_\e)$. Alternatively, using corollary \ref{corollary:deform.complex}, we can express such a cocycle as an element $\eta_\e$ of the mapping cone complex of the de Rham differential $-d_{dR}:\Omega^{1}(\G^{(\bullet)})\to\Omega^{2}(\G^{(\bullet)})$. Explicitly,
$$\eta_\e=\zeta_\e-\dot{\omega}_\e\in\Omega^{1}(\calG^{(2)})\oplus\Omega^{2}(\calG).$$

As we will show in proposition \ref{symp.socycle} and remark \ref{Basta en cero} below, every $\eta_\e$ is indeed a 2-cocycle of the deformation complex $C^{\bullet}_\mathrm{def}(\G_\e,\omega_\e)$ when viewed in terms of differential forms as in corollary \ref{corollary:deform.complex}. The corresponding cohomology class $[\eta_\e]$ will be called the \textbf{deformation class}.

We now prove that $\eta_0\in C^{2}_\mathrm{def}(\G,\omega)$ is in fact a cocycle, which will be called the \textbf{infinitesimal deformation cocycle} associated to the deformation $(\calG_\e,\omega_\e)$.

\begin{proposition}\label{symp.socycle}

Let $(\G_\e,\omega_\e)$ be an $s$-constant deformation of $(\G,\omega)$. The 2-cochain $\eta_0=\zeta_0-\dot{\omega}_0\in\Omega^{1}(\G^{(2)})\bigoplus\Omega^{2}(\G)$, defined above, is a 2-cocycle. The corresponding cohomology class in $H^{2}_\mathrm{def}(\calG,\omega)$ depends only on the equivalence class of the deformation.
\end{proposition}

To prove this Proposition we first enunciate two technical lemmas. In what follows, denote by $\calG^{[2]}$ the domain of the division map $\overline{m}$ of $\calG$, i.e., the space of pairs of arrows with the same source; denote by $p_i:\calG^{[2]}\longrightarrow\calG$ the projections on $\calG$, and denote by $pr_i:\calG^{(2)}\longrightarrow\calG$ the projections from the space of composable arrows to $\calG$, for $i=1,2$.

\begin{lemma}
A $k$-form $\omega\in\Omega^{k}(\calG)$ is multiplicative if, and only if
\begin{equation}\label{multipl.1}
\overline{m}^{*}\omega=p_1^{*}\omega-p_2^{*}\omega.
\end{equation}
\end{lemma}

\begin{proof}
We just need to show that equation (\ref{multipl.1}) is equivalent to the equation
\begin{equation}\label{multipl.2}
m^{*}\omega=pr_1^{*}\omega+pr_2^{*}\omega.
\end{equation}
This follows from the diffeomorphism $\psi:\calG^{[2]}\longrightarrow\calG^{(2)},$ $(p,q)\longmapsto(p,i(q))$, where $i:\calG\longrightarrow\calG$ is the inversion map of $\calG$. Indeed, by applying $\psi^{*}$ to equation (\ref{multipl.2}) and since $i^{*}\omega=-\omega$ (\cite{BCWZ}, Lemma 3.1), we obtain the equation (\ref{multipl.1}). \end{proof}

\begin{lemma}\label{M epsilon}
Let $(\calG,\overline{m}_\e)$ be an $s$-constant deformation of $\calG$. Define for every $\e\in I$ the diffeomorphism $\overline{M}_\e:\calG^{[2]}\longrightarrow\calG^{[2]},\ (p,q)\longmapsto(\overline{m}_\e(p,q),i_\e(q))$ and the map $\psi_\e:\calG^{[2]}\longrightarrow\calG^{(2)}_\e,\ (p,q)\longmapsto (p,i_\e(q))$. These families of maps satisfy the following properties:

\begin{enumerate}                  
     \item \ \ $\overline{m}_\e\circ\overline{M}_\e^{-1}=p_1=(pr_1)_\e\circ\psi_\e$;
     \item \ \ $p_1\circ\overline{M}_\e^{-1}=m_\e\circ\psi_\e$;
     \item \ \ $p_2\circ\overline{M}_\e^{-1}=i_\e\circ p_2=(pr_2)_\e\circ\psi_\e$.
    \item If $V_\e$ denotes the time-dependent vector field on $\calG^{[2]}$ associated to the smooth family of diffeomorphisms $\overline{M}_\e$ (i.e., $\left.\frac{d}{d\t}\right|_{\t=\e}\overline{M}_\t(p,q)=V_\e(\overline{M}_\e(p,q))$), then $$V_\e(p,q)=\left(\xi_\e(p,i_\e(q)),\left.\frac{d}{d\t}\right|_{\t=\e}i_\t(i_\e(q))\right).$$
\end{enumerate}
\end{lemma}

\begin{proof}
	Items (1)-(3) are straightforward once that $\overline{M}_\e^{-1}(p,q)=(m_\e(p,i_\e(q)),i_\e(q))$.
	For item (4) note that $V_\e(p,q)=\left.\frac{d}{d\t}\right|_{\t=\e}(\overline{M}_\t\circ\overline{M}_\e^{-1})(p,q)$, so its proof follows from applying $\left.\frac{d}{d\t}\right|_{\t=\e}$ to $(\overline{M}_\t\circ\overline{M}_\e^{-1})(p,q)=(\overline{m}_\t(m_\e(p,i_\e(q)),i_\e(q)),i_\t(i_\e(q)))$.
\end{proof}

\begin{proof2}\textbf{of proposition} \ref{symp.socycle}\\
The proof of the first part follows from differentiating the multiplicativity condition \eqref{multipl.1} of the family $\omega_\e$. In fact, by taking derivative with respect to $\e$ and reordering the terms, we get

\begin{equation}\label{derivative}
\left.\frac{d}{d\e}\right|_{\e=\l}\bar{m}_{\e}^{*}\omega_\e=\left(-\bar{m}_\l^{*}(\left.\frac{d}{d\e}\right|_{\e=\l}\omega_\e)+p_1^{*}(\left.\frac{d}{d\e}\right|_{\e=\l}\omega_\e)-p_2^{*}(\left.\frac{d}{d\e}\right|_{\e=\l}\omega_\e)\right).
\end{equation}

Now, we use Lemma \ref{M epsilon} and the family of diffeomorphisms $\bar{M}_\e:\calG^{[2]}\longrightarrow\calG^{[2]}$ defined there. Since $\bar{M}_{\e}(p,q)=(\bar{m}_\e(p,q),i_\e(q))$, then $\left.\frac{d}{d\e}\right|_{\e=\l}\bar{m}_{\e}^{*}\omega_\l$ is equal to $\left.\frac{d}{d\e}\right|_{\e=\l}\bar{M}_{\e}^{*}pr_1^{*}\omega_\l$, which is equivalent to $\left.\frac{d}{d\e}\right|_{\e=\l}\bar{M}_{\l}^{*}(\bar{M}_\e\circ\bar{M}_{\l}^{-1})^{*}pr_1^{*}\omega_\l$. In this way, letting $\varphi^{\e,\l}:=\bar{M}_\e\circ\bar{M}_{\l}^{-1}$ (which, by construction of $V_\e$, is its time-dependent flow), then we have

\begin{equation*}
\begin{split}
\left.\frac{d}{d\e}\right|_{\e=\l}\bar{M}_{\l}^{*}(\bar{M}_\e\circ\bar{M}_{\l}^{-1})^{*}p_1^{*}\omega_\l&=\bar{M}_{\l}^{*}\left.\frac{d}{d\e}\right|_{\e=\l}(\varphi^{\e,\l})^{*}p_1^{*}\omega_\l\\
&=\bar{M}_{\l}^{*}L_{V_\l}(p_1^{*}\omega_\l).%\\
%&=\bar{M}_{\l}^{*}d(\i_{V_\l}p_1^{*}\omega_\l). ...........
\end{split}
\end{equation*}

That is, by item (1) of Lemma \ref{M epsilon}, equation (\ref{derivative}) becomes

\begin{equation*}
\begin{split}
L_{V_\l}(p_1^{*}\omega_\l)&=(\bar{M}_{\l}^{-1})^{*}\left(-\bar{m}_\l^{*}(\left.\frac{d}{d\e}\right|_{\e=\l}\omega_\e)+p_1^{*}(\left.\frac{d}{d\e}\right|_{\e=\l}\omega_\e)-p_2^{*}(\left.\frac{d}{d\e}\right|_{\e=\l}\omega_\e)\right)\\
&=-\psi_\l^{*}(pr_1)_\l^{*}(\left.\frac{d}{d\e}\right|_{\e=\l}\omega_\e)+\psi_\l^{*}m_\l^{*}(\left.\frac{d}{d\e}\right|_{\e=\l}\omega_\e)-\psi_\l^{*}(pr_2)_\l^{*}(\left.\frac{d}{d\e}\right|_{\e=\l}\omega_\e)\\
&=-\psi_\l^{*}\delta^{\l}(\left.\frac{d}{d\e}\right|_{\e=\l}\omega_\e).
\end{split}
\end{equation*}

Also, since $\omega_\l$ is a closed form, by Cartan's formula we have

\begin{equation*}
\begin{split}
(\psi_\l^{-1})^{*}L_{V_\l}(p_1^{*}\omega_\l)&=d\left((\psi_\l^{-1})^{*}(V_\l\intprod p_1^{*}\omega_\l)\right)\\%.............\\
&=d[(\psi_\l)_{*}V_\l\intprod(\psi_\l^{-1})^{*}p_1^{*}\omega_\l]\\
&=d[(\psi_\l)_{*}V_\l\intprod(pr_1)_\l^{*}\omega_\l]\\
&=-d\zeta_\l,
\end{split}
\end{equation*}

where the last equality follows from identification \eqref{identification} in Section \ref{Def.comp.symp.grpds}. Namely, 

\begin{equation*}
\begin{split}
((\psi_\l)_{*}V_\l\intprod &(pr_1)_\l^{*}\omega_\l)_{(g,h)}=\\ &=\left.(dpr_1)_\l^{*}\right|_{(g,h)}\left[\left.\omega_\l^{b}\right|_{g}\left(d(pr_1)_\l\circ(d\psi_\l)_{(g,i_\l(h))}(V_\l(g,i_\l(h)))\right)\right]\\
&=\left.(dpr_1)_\l^{*}\right|_{(g,h)}\left[\left.\omega_\l^{b}\right|_{g}(\xi_\l(g,h))\right]\\
&=-\zeta_\l(g,h)\ \ \ \ (\text{by identification } \eqref{identification}).
\end{split}
\end{equation*}

Therefore, we conclude that

$$d\zeta_\l=\delta^{\l}\left(\left.\frac{d}{d\e}\right|_{\e=\l}\omega_\e\right).$$

That is, 
$-d\zeta_\l-\delta^{\l}(-\left.\frac{d}{d\e}\right|_{\e=\l}\omega_\e)=0$.
This completes the first part of the proof by taking $\l=0$.

Take now $(\calG'_{\e},\omega'_\e)$ an equivalent deformation of $(\calG,\omega)$, and denote by $\eta'_0$ its infinitesimal cocycle. We will prove that $\eta_0-\eta'_0$ is exact. This follows from taking derivatives at $\e=0$ of the equivalence condition $F_\e^{*}\omega'_\e=\omega_\e+\delta_\e(\alpha_\e)$ (recall that $\alpha_0=0$). Doing so, we obtain

$$F_0^{*}\left(di_{Z_{0}}\omega'_{0}+\dot{\omega}'_{0}\right)=\dot{\omega}_{0}+\delta_{0}(\dot{\alpha}_0),$$
where $Z_\e$ is the time-dependent vector field associated to the family $F_\e$. Hence,

$$\dot{\omega}'_0-\dot{\omega}_0=-di_{Z_{0}}\omega'_{0}+\delta_{0}(\dot{\alpha}_0).$$

On the other hand, since $F_\e:\calG_\e\rightarrow\calG'_\e$ is an isomorphism for every $\e$, then the groupoid deformation cocycles $\xi_0$ and $\xi'_0$ satisfy that $\xi'_0-\xi_0=\delta^\textrm{def}_0(Z_0)$. In terms of differential forms, this amounts to $\zeta'_0-\zeta_0=\delta_0(-i_{Z_0}\omega)$ (see identification \eqref{identification}). Therefore, the equation above becomes
\begin{align*}
\eta_0-\eta'_0&=-\dot{\omega}_0+ \zeta_0-(-\dot{\omega}'_0+\zeta'_0)\\
&=\zeta_0-\zeta'_0 + (\dot{\omega}'_0-\dot{\omega}_0)\\
&=\delta_0(i_{Z_0}\omega)-d_{dR}(i_{Z_{0}}\omega)+\delta_{0}(\dot{\alpha}_0)\\
&=D_{0}(i_{Z_0}\omega-\dot{\alpha}_0).
\end{align*}
\end{proof2}

\begin{corollary}
The 2-cocycle associated to the trivial deformation vanishes in cohomology.
\end{corollary}

\begin{remark}\label{Basta en cero}
Since a deformation $(\calG_{\e},\omega_\e)$ of $(\calG_0,\omega_0)$ can be seen as the deformation $(\calG_{\e+\l},\omega_{\e+\l})$ of $(\calG_{\l},\omega_\l)$ for any $\l\in I$, then the corresponding cochain $\eta_{\l}\in C^{2}_\mathrm{def}(\calG_\l,\omega_\l)$ is also a cocycle. Furthermore, \textbf{in the case of a trivial deformation we can say even more: every $\eta_\l$ is exact}. This follows from noting that the deformation $(\calG_{\e},\omega_\e)$ of $(\calG_0,\omega_0)$ being trivial implies that $(\calG_{\e+\l},\omega_{\e+\l})$ is also a trivial deformation of $(\calG_{\l},\omega_\l)$, for any $\l\in I.$
\end{remark}

In the next section we show that, under appropriate regularity and compactness conditions, the converse statement of the previous remark is also true.

\subsection{Moser arguments and moduli spaces}

\begin{theorem}\label{thm.triviality.s-constant} (Moser Theorem for symplectic groupoids)\\
Let $(\calG_\e,\omega_\e)$ be an $s$-constant deformation of a compact symplectic groupoid $(\calG,\omega)$. Then, the deformation $(\calG_\e,\omega_\e)$ is trivial if, and only if, the deformation class \[\eta_\e:=\zeta_\e-\dot{\omega}_\e \in C^{2}_\mathrm{def}(\calG_\e,\omega_\e)\] is smoothly exact.
\end{theorem}

\begin{proof}
We will show that the smooth exactness of the deformation class $\eta_\e$ implies the triviality of the deformation $(\calG_\e,\omega_\e)$. The remark \ref{Basta en cero} above shows the converse statement. Thus, let us assume that $\eta_\e$ is a smoothly exact family of cocycles. Then, we get
\begin{equation*}
\delta_{\t}(-\chi_\t)=\zeta_\t \text{ and } -\dot{\omega}_\t=-d(-\chi_\t)-\delta_\t(\tilde{\alpha}_\t);
\end{equation*}
 where $\chi_\t\in\Omega^{1}(\calG_\t)$ and $\tilde{\alpha}_\t\in\Omega^{2}(M)$ are smooth families of forms. The first equation amounts to $\delta^{\t}_\mathrm{def}(X_\t)=\xi_\t$, where $X_\t$ is the smooth family of vector fields such that $\chi_\t=\iota_{X_\t}\omega_\t$ for every $\t$. Let $(\phi_\e,\varphi_\e)$ be the flow of the time-dependent vector field $X_\e$ (starting at time zero). Then, by letting  $\bar{\alpha}_\t:=\int_{0}^{\t}(\phi_\e^{*}\tilde{\alpha}_\e) d\e$ be the  primitive of the curve $\t\mapsto\phi_\t^{*}\tilde{\alpha}_\t$, we obtain
 
 $$\dot{\omega}_\t=-d(\iota_{X_\t}\omega_\t)+\delta_\t((\varphi_\t^{-1})^{*}\frac{d}{ds}\bar{\alpha}_\t),$$
 which is,
 
 $$\frac{d}{d\e}\rvert_{\e=\t}\left[\phi_\e^{*}\omega_\e-\delta_0(\bar{\alpha}_\e)\right]=0.$$
 
 That is,
 $$\phi_\e^{*}\omega_\e-\delta_0(\alpha_\e)=\omega, \text{ for } \alpha_\e=\bar{\alpha}_0-\bar{\alpha}_\e.$$
 Hence, one concludes that the deformation $(\G_\e,\omega_\e)$ is trivial.
 \end{proof}  

As a consequence of this version of the Moser Theorem, then for a compact symplectic groupoid $(\G,\omega)$ we can describe a neighbourhood of $\omega$ in a moduli space of multiplicative symplectic structures on $\G$. We say \textit{a} moduli space and not \textit{the} moduli space, because there are a few options of which equivalence relation to consider. We will use one that is suggested by the equivalence relation on deformations of symplectic groupoids. 
 
The space of closed multiplicative $2$-forms on a Lie groupoid $\G$ is the subspace $\Omega_\mathrm{cl,mult}^2(\G)\subset \Omega^2(\G)$. It is a linear subspace, being the intersection of the kernel of the two linear operators $d_{dR}$ and $\delta$. Non-degeneracy is an open condition, so the space of multiplicative symplectic structures on $\G$ is an open neighbourhood of $\omega$ in  $\Omega_\mathrm{cl,mult}^2(\G)$, which we denote by $\mathcal{S(\G)}$.

\begin{definition}Two multiplicative symplectic forms $\omega_1,\omega_2\in\mathcal{S(\mathcal{\G})}$ are said to be \textbf{convex gauge-isotopic} if the deformation $(\G,\omega_\e)$ given by $\omega_1 +\e(\omega_2-\omega_1)$ is a trivial deformation of symplectic groupoids (in particular $\omega_\e$ is required to be non-degenerate for all $\e\in[0,1]$).

Denote by $\sim_\mathrm{cgi}$ the equivalence relation on $\mathcal{S}(\G)$ generated by convex gauge-isotopies.
\end{definition}

\begin{theorem}[Moduli space]\label{thm-moduli} Let $(\G,\omega)$ be a compact symplectic groupoid. Then there is a neighbourhood $\mathcal{U}$ of $\omega$ in $\mathcal{S}({\G})$ and a neighbourhood $\mathcal{V}$ of $0$ in $H^2_\mathrm{def}(\G,\omega)$, such that the map $\kappa :\mathcal{U}\to H^2_\mathrm{def}(\G,\omega)$, defined by $\kappa(\omega_1) = [0\oplus (\omega_1-\omega)]$ induces a 1:1 correspondence \[\mathcal{U}\big/\sim_\mathrm{cgi}\ \  \longleftrightarrow\ \  \mathcal{V}\subset H^2_\mathrm{def}(\G,\omega),\] sending the equivalence class of $\omega$ to $0$. 
\end{theorem}
\begin{proof}

The map $\kappa$ descends to the quotient: if $\omega_1$ and $\omega_2$ are convex gauge-isotopic, then the deformation $(\G,\omega_1 +\e(\omega_2-\omega_1))$ of $(\G,\omega_1)$ is trivial, so its deformation class $[0\oplus(\omega_2-\omega_1)]\in H^2_\mathrm{def}(\G,\omega_1)$ vanishes. Therefore $[0\oplus(\omega_2-\omega_1)]\in H^2_\mathrm{def}(\G,\omega)$ also vanishes. Thus $\kappa(\omega_2)-\kappa(\omega_1)=[0\oplus(\omega_2-\omega)]-[0\oplus(\omega_1-\omega)]=[0\oplus(\omega_2-\omega_1)]=0$. 

Let $\mathcal{U}''$ be a convex neighbourhood of $\omega$ in $\mathcal{S}(\G)$, and let $\omega_1,\omega_2\in \mathcal{U}''$ be such that $\kappa(\omega_1)-\kappa(\omega_2)=0$. By definition, $[0\oplus(\omega_1-\omega_2)]=0\in H^2_\mathrm{def}(\G,\omega)$, meaning that $(\omega_1-\omega_2)$ is exact in deformation cohomology. We can consider the $s$-constant deformation $(\G,\omega_2+\e(\omega_1-\omega_2))$ of $(\G,\omega_2)$ (it is a deformation of symplectic groupoids because $\mathcal{U}''$ is convex). This deformation has a constant deformation class $[0\oplus(\omega_1-\omega_2)]=0\in H^2_\mathrm{def}(\G,\omega_2)$ for all $\e$, which we know to be exact. Therefore the deformation classes are smoothly exact and by Theorem \ref{thm.triviality.s-constant} we conclude the triviality of the deformation. This implies that $(\G,\omega_1)$ and $(\G,\omega_2)$ are convex gauge-isotopic. Hence $\kappa$ descends to an injective map.

Neighbourhoods $\mathcal{U}'$ of $\omega$ and $\mathcal{V}'$ of $0$ can be chosen such that $\kappa$ maps $\mathcal{U}'$ onto $\mathcal{V}'$. The reason for this is that by Proposition \ref{prop-symp-prop} any class in $H^2_\mathrm{def}(\G,\omega)$ has a representative of the form $0\oplus \omega_1$ with $\omega_1\in \Omega_\mathrm{mult}^2(\G)$. Let $\mathcal{V}'$ be small enough that for every class in $\mathcal{V}'$ there is such a representative $0\oplus \omega_1$ for which $\omega+\omega_1$ is non-degenerate, and $\omega_1$ belongs to $\mathcal{U}''$. By definition $\kappa(\omega+\omega_1)=[0\oplus \omega_1]$, so $\kappa:\mathcal{U}'=\kappa^{-1}(\mathcal{V}')\to \mathcal{V'}$ becomes surjective. Finally, let $\mathcal{U}\subset \mathcal{U}'$ be a convex neighbourhood of $\omega$, and let $\mathcal{V}=\kappa(\mathcal{U})$.
\end{proof}

\begin{remark}[On the relation with Kodaira-Spencer maps]
The map $\kappa$ plays a similar role to a Kodaira-Spencer map in the theory of deformations of complex structures \cite[Sec. 4.2]{Kodaira}. Allowing for some heuristics, and if we disregard issues with spaces of infinite dimension, there is a tautological (infinite-dimensional) family of symplectic groupoids \[pr_{\mathcal{S}(\G)}: \mathfrak{G}=(\G\times \mathcal{S}(\G),\widetilde{\omega})\longrightarrow \mathcal{S}(\G),\] 
where $\tilde{\omega}$ is the tautological $pr_{S(\G)}$-foliated $2$-form on $\mathfrak{G}=\G \times S(\G)$ with values $\tilde{\omega}(g,\omega)=\omega(g)$. The fibre $pr_{\mathcal{S}(\G)}^{-1}(\omega)$ of this family over $\omega\in \mathcal{S}(\G)$ is equal to $(\G,\omega)$.

By analogy with the case of complex manifolds, the Kodaira-Spencer map of this family is a map $KS:T_\omega \mathcal{S}(\G)\to H^2_\mathrm{def}(\G,\omega)$; given a smooth curve $\gamma:I\to \mathcal{S}(\G)$, define $KS(\dot{\gamma}(0))$ to be the deformation class of the deformation $\textbf{p}:\gamma^*\mathfrak{G}\to I$.
Since $\mathcal{S}(\G)$ is an open subset of the vector space $\Omega^2_\mathrm{mult,cl}(\G)$, there is a natural identification $T_\omega \mathcal{S}(\G)\cong \Omega^2_\mathrm{mult,cl}(\G)$. Under this identification, the Kodaira-Spencer map $KS$ corresponds to $\kappa$.

For a compact symplectic groupoid $(\G,\omega)$, taking the analogy further, a germ of a neighbourhood of $0$ in $H^2_\mathrm{def}(\G,\omega)$ plays the role of what is called a Kuranishi space (cf.\ \cite[Def. 4]{Catanese}, see also Corollary 9 in \textit{loc.\ cit.}) because infinitesimal deformations of $(\G,\omega)$ are unobstructed, as the proof of Theorem \ref{thm-moduli} shows. 
\end{remark}

\begin{remark}[Other versions of the Moser theorem]\label{Rmk-other-moduli}
It also makes sense to study deformations of multiplicative symplectic forms on a fixed Lie groupoid $\G$, up to a stricter equivalence relation where triviality is given by paths of symplectic groupoid isomorphisms (i.e., not allowing for cohomologically trivial gauge transformations as equivalences). In this case, the correct deformation complex to consider would be $(\Omega^\bullet_\mathrm{mult}(\G), d_{dR})$, and a Moser theorem in that setting is detailed in \cite{CS}. It is similar to Theorem \ref{thm.triviality.s-constant} in the sense that when $\G$ is compact, deformations will be trivial if and only if their deformation classes in $H^2_\mathrm{mult}(\G)$ are smoothly exact. 
This in turn will lead to a similar correspondence to that of Theorem \ref{thm-moduli}, of the form
\[\mathcal{U}/\sim\ \ \longleftrightarrow\ \  \mathcal{V}\subset H^2_\mathrm{mult}(\G),\] for neighbourhoods $\mathcal{U}$ of $\omega$ in $\mathcal{S}(\G)$ and $\mathcal{V}$ of $0$ in $H^2_\mathrm{mult}(\G)$. Here $\sim$ denotes the equivalence relation generated by convex isotopies, where $\omega_0$ and $\omega_1$ are convex-isotopic if the form $\omega_\e=\omega_0+\e(\omega_1-\omega_0)$ is symplectic for all $\e\in [0,1]$, and there is an isotopy $F_\e$ of $\G$ such that each $F_\e$ is a Lie groupoid isomorphism satisfying $F_\e^*\omega_\e=\omega_0$.
\end{remark}

\begin{remark}\label{remark-moduli2} Any deformation $\textbf{p}:(\widetilde{\G}\rightrightarrows \tilde{M})\to (I\rightrightarrows I)$ of a compact Lie groupoid $\G$ is a trivial deformation \cite[Thm. 7.4, Rem. 7.9]{CMS}. If moreover $\textbf{p}$ is proper, then the entire family $\widetilde{\G}$ is trivial, without the need to restrict to a smaller interval $J\subset I$.
This means that any deformation $(\G_\e,\omega_\e)$ of a compact symplectic groupoid is equivalent to one of the form $(\G,\phi_\e^*\omega_\e)$, and in particular $s$-constant, for $\e$ small enough. Thus, by Proposition \ref{symp.socycle}, any deformation class in $H^2_\mathrm{def}(\G,\omega)$ is represented by a cocycle of the form $0\oplus \dot{\omega}_0\in C^2_\mathrm{def}(\G)\oplus C^1_\mathrm{def}(\omega)$, and has also a corresponding cocycle in $H^2_\mathrm{mult}(\G)$. 

Therefore, if the space of Lie groupoid structures on a compact manifold $\G$ (or perhaps having an underlying submersion $s:\G\to M$ as source) were locally path-connected, then the correspondences of Theorem \ref{thm-moduli} and of Remark \ref{Rmk-other-moduli} can be upgraded. They would give local descriptions not only of moduli spaces of multiplicative symplectic forms on $\G$, but of symplectic groupoids structures on the manifold $\G$ (respectively, with underlying source $s:\G\to M$). Although we do not know if the space of Lie groupoid structures on $\G$ is locally path-connected, in this direction we note that the space of actions of a compact Lie group on a compact manifold is, indeed, locally path-connected \cite{Palais_path}.
\end{remark}

%%%%%%%%%%%%%%%%%% SECTION %%%%%%%%%%%%%%%%%%%%%

\section{The deformation class of general deformations}\label{section:general}

In this section we show how general deformations of symplectic groupoids can also be handled in an infinitesimal way.

\subsection{Defining the deformation class}

In order to describe the deformation class of a deformation $(\widetilde{\G},\omega)$ we first reinterpret the definition of the deformation cocycle in the case of $s$-constant deformations $(\widetilde{\G}=\calG\times I,\omega_\e)$. The deformation cocycle $\eta_\e=\zeta_\e-\dot{\omega}_\e$ can be alternatively expressed as
\begin{equation}\label{eq:cocycle2}
\eta_\e=\sigma(\omega(\delta_\mathrm{def}(\partial/\partial\e),\cdot))_{|_{\G^{(2)}_\e}}-(\mathcal{L}_{\partial/\partial\e}\omega)_{|_{\G_\e}},
\end{equation}
where $\partial/\partial\e$ is the vector field on $\widetilde{\G}$ whose flow $F_\e$ is $(g,\l)\mapsto (g,\l+\e)$, and \[\sigma:pr_1^{*}T^{*}\widetilde{\calG}\to T^{*}\widetilde{\calG}^{(2)}\] is the map of vector bundles over $\G^{(2)}$ defined on Remark \ref{rmk:sigma_e} by $\sigma=(dpr_1)^*$.

Note also that $\omega$ is a 2-form on $\widetilde{\calG}$ which restricts to $\omega_\e$ over the fibre $\calG\times \{\e\}$.
This alternative description of the deformation cocycle
 can be extended in such a way that the deformation \textit{class} still makes sense for a general deformation.

\begin{proposition}\label{prop.gen.cocycle}
Let $(\widetilde{\calG},\omega)$ be a deformation of $(\G,\omega_\G)$, and let $X\in\mathfrak{X}(\widetilde{\G})$ be a \textit{transverse} vector field; that is, a vector field on $\widetilde{\G}$ which projects by $\textbf{p}\circ\tilde{s}$ to the vector field $\partial/\partial\e$ on $I$. Then,
\begin{equation}\label{eq:generalcocycle}
\eta_\e:=\sigma(\omega(\delta_\mathrm{def}X,\cdot))|_{\G^{(2)}_\e}-(\mathcal{L}_{X}\omega)|_{\G_\e}\  \in\ \Omega^{1}_\mathrm{def}(\calG_\e^{(2)})\oplus\Omega^{2}(\calG_\e).
\end{equation}
defines a family of cocycles $\eta_\e\in C^{2}_\mathrm{def}(\calG_\e,\omega_\e)$. 
Moreover, the cohomology class $[\eta_0]$ at time zero does not depend on the choice of the transverse vector field $X$.
\end{proposition}

\begin{definition} The resulting cohomology class $[\eta_0]\in H^2_\mathrm{def}(\G,\omega_\G)$ is called \textbf{the deformation class} associated to the deformation $(\widetilde{\G},\omega)$.
\end{definition}

\begin{remark}
Equivalently, the family of cocycles $\eta_\e$ can be viewed as a single cocycle $\eta$ on the ``\emph{foliated version}'' $C^\bullet_\mathrm{def}(\widetilde{\calG},\omega)_{\mathcal{F}}$ of $C^\bullet_\mathrm{def}(\widetilde{\calG},\omega)$. That is, on the mapping cone complex of the leafwise de Rham differential $d_{\mathcal{F}}:\Omega^{1}_\mathrm{def}(\widetilde{\calG}^{(\bullet)})_{\mathcal{F}}\to\Omega^{2}_{\mathcal{F}}(\widetilde{\calG}^{(\bullet)})$. Here, to shorten the notation we let $\mathcal{F}$ refer to the foliation $\calF^{(n)}$ on each $\widetilde{\G}^{(n)}$ with leaves $\G_\e^{(n)}$, induced by the foliation on $\widetilde{\G}$ previously denoted by $\calF$, having the fibres $\G_\e$ as leaves. The space $\Omega^{1}_\mathrm{def}(\widetilde{\calG}^{(\bullet)})_{\mathcal{F}}$ is the space of the foliated forms corresponding to forms in $\Omega^{1}_\mathrm{def}(\widetilde{\calG}^{(\bullet)})$. Explicitly, we have
\begin{equation}
\eta=\sigma(\omega(\delta_\mathrm{def}X,\cdot))|_{\mathcal{F}^{(2)}}-(\mathcal{L}_{X}\omega)|_{\mathcal{F}}\ \ \in\ \Omega^{1}_\mathrm{def}(\widetilde{\calG}^{(2)})_\mathcal{F}\oplus\Omega^{2}_\mathcal{F}(\widetilde{\calG}).
\end{equation}
\end{remark}

\begin{proof2} \textbf{of Proposition} \ref{prop.gen.cocycle}

Let $\calM:\widetilde{\calG}^{(2)}\rightarrow\widetilde{\calG}^{(2)}$ be the diffeomorphism defined by $\calM(g,h)=(M(g,h),i(h))$, where $M$ is the multiplication of $\widetilde{\calG}$.

On the one hand, notice that
\begin{align*}
M^{*}\calL_X\omega&=\calM^{*}pr_1^{*}\calL_X\omega\\
&=\calM^{*}\calL_{X^{pr_1}}pr_1^{*}\omega,\ \ \ \ \text{where }\ \ X^{pr_1}_{(g,h)}:=(X_g,(i_{*}X)_h)\in T_{(g,h)}\widetilde{\calG}^{(2)}\\
&=\calL_{X^{M}}\calM^{*}pr_1^{*}\omega,\ \ \ \ \text{where }\ \ X^{M}_{(g,h)}\ :=(\calM_{*}^{-1}X^{pr_1})_{(g,h)}\\ & \hspace{5.5cm} =(X_{gh}\cdot(i_{*}X)_{h^{-1}},X_h)\in T_{(g,h)}\widetilde{\calG}^{(2)}\\
&=\calL_{X^{M}}M^{*}\omega\\
&=\calL_{X^{M}}pr_1^{*}\omega+\calL_{X^{M}}p_2^{*}\omega.
\end{align*}

That is,
\begin{align*}
(M^{*}\calL_X\omega)_{(g,h)}&=\big(\calL_{X^{M}}pr_1^{*}\omega+\calL_{X^{M}}p_2^{*}\omega\big)_{(g,h)}\\
&=(d\i_{X^{M}}pr_1^{*}\omega)_{(g,h)}+(d_{(g,h)}pr_1)^{*}\big(d\omega(X_{gh}\cdot di(X_h),-,-)\big)\\ &+ (\calL_{X^{M}}p_2^{*}\omega)_{(g,h)}.
\end{align*}

On the other hand,
\begin{align*}
(pr_1^{*}\calL_X\omega)_{(g,h)}&=(pr_1^{*}\i_{X}d\omega)_{(g,h)}+(d\;pr_1^{*}\i_X\omega)_{(g,h)}\\
&=(d_{(g,h)}pr_1)^{*}d\omega(X_g,-,-)+(d\;pr_1^{*}\i_X\omega)_{(g,h)}
\end{align*}

and
\[p_2^{*}\calL_X\omega=\calL_{X^{M}}p_2^{*}\omega.\]

Therefore,
\begin{align*}
(pr_1^{*}\calL_X\omega &-M^{*}\calL_X\omega +p_2^{*}\calL_X\omega)_{(g,h)} = \\ &=(d_{(g,h)}pr_1)^{*}d\omega(X_g,-,-)+(d\;pr_1^{*}\i_X\omega)_{(g,h)}-(d\i_{X^{M}}pr_1^{*}\omega)_{(g,h)}\\ &-(d_{(g,h)}pr_1)^{*}\big(d\omega(X_{gh}\cdot di(X_h),-,-)\big)\\
&=(d_{(g,h)}pr_1)^{*}d\omega(X_g-X_{gh}\cdot di(X_h),-,-)+ \big(d(pr_1^{*}\i_X\omega-\i_{X^{M}}pr_1^{*}\omega)\big)_{(g,h)}\\
&=\big(\i_{(\delta_\mathrm{def}X,0)}pr_1^{*}d\omega\big)_{(g,h)} + \big(d\i_{(\delta_\mathrm{def}X,0)}pr_1^{*}\omega\big)_{(g,h)}.
\end{align*}

Thus, since $(\delta_\mathrm{def}X)_{(g,h)}$ is tangent to the fibres of $\G_\e$ of $\widetilde{\calG}$ and each restriction $\omega_\e\equiv\omega|_{\calG_\e}$ is a closed 2-form, then
$$\left.\big(\i_{(\delta_\mathrm{def}X,0)}pr_1^{*}d\omega\big)\right|_{\calG_\e}=0,$$
for all $\e$, which proves the proposition.
\end{proof2}

\begin{remark}
Alternatively, since the tangent lift map $\oplus^{2}T_\mathrm{def}$ takes transversal vector fields to transversal vector fields, Proposition \ref{prop.gen.cocycle} can also be proved in terms of deformations of Lie groupoid morphisms by using the results of \cite{CS}.
\end{remark}

%%%%%%%%%%%%%%%%%%%%% Section %%%%%%%%%%%%%%%%%%%%%%%

\subsection{Triviality of deformations with exact class}

In this section we generalize Theorem \ref{thm.triviality.s-constant} to general deformations of compact symplectic groupoids.

\begin{theorem}\label{thm.triviality.gen}
Let $(\widetilde{\calG},\omega)$ be a deformation of a compact symplectic groupoid $(\calG,\omega_\calG)$ and let $\eta=\sigma(\omega(\delta_\mathrm{def}X,\cdot))|_{\mathcal{F}^{(2)}}-(\mathcal{L}_{X}\omega)|_{\mathcal{F}}$ be the deformation cocycle in the foliated cohomology $C^{\bullet}_\mathrm{def}(\G,\omega)_{\mathcal{F}}$, associated to the deformation via a transverse vector field $X$ on $\widetilde{\G}$. Then, $(\widetilde{\calG},\omega)$ is trivial if and only if $\eta$ is exact.
\end{theorem}

The proof of this result will be based on the following lemma.

\begin{lemma} Let $(\widetilde{\G},\omega)$ be a deformation of $(\G,\omega_\G)$ as in the theorem above. Then, the associated foliated form $\omega_\calF$ induces the isomorphism
$$C^{\bullet}_\mathrm{def}(\widetilde{\calG})_{\calF}\to \Omega^{1}_\mathrm{def}(\widetilde{\calG}^{(\bullet)})_\calF,\ \xi\mapsto -\sigma(\omega(\xi,\cdot))_{\calF}=-\sigma_\calF(\omega_\calF(\xi,\cdot)).$$
Under this isomorphism, the image $\mathcal{Y}=-i_Y\omega_\calF$ of an element $Y\in C^{1}_\mathrm{def}(\widetilde{\calG})_{\calF}$ satisfies
$$d_\calF\mathcal{Y}=-(\mathcal{L}_{Y}\omega)_{\calF},\ \text{and }\  \delta_\calF\mathcal{Y}=-\sigma(\omega(\delta_\mathrm{def}Y,\cdot))_\calF.$$
\end{lemma}

\begin{proof}
Checking that $\omega_\calF$ induces an isomorphism is done by essentially the same arguments as those for Proposition \ref{prop:isomorph-nondegeneracy}.

Now,
\begin{align*}
d_\calF\mathcal{Y}&=(-d\i_Y\omega)_\calF\\
&=(-d\i_Y\omega)_\calF+\underbrace{(-\i_Yd\omega)_\calF}_{=0}\\
&=-(\mathcal{L}_Y\omega)_\calF;
\end{align*}
where the second equality is due to the fact that $\omega$ is $d_\calF$-closed. Moreover,

\begin{align*}
\delta_\calF\mathcal{Y}&=-\delta_\calF(\i_Y\omega_\calF) =-\delta_\calF(\i_Y\omega)_\calF =-(\delta\i_Y\omega)_\calF\\
& =(\delta(\omega^{b})^{*}(Y))_\calF\\
& =((\omega^{b})^{*}(\delta_\mathrm{def}Y))_\calF \\
& =-\sigma(\omega(\delta_\mathrm{def}Y,\cdot))_\calF. \qedhere
\end{align*} \end{proof}

\begin{proof2} \textbf{of Theorem} \ref{thm.triviality.gen}

If $\eta$ is exact, then there exists $\alpha\oplus\mathcal{Y}\in\Omega^{2}_{\calF}(\widetilde{M})\oplus\Omega^{1}_\mathrm{def}(\widetilde{\calG})_{\calF}$ such that 

$$\sigma(\omega(\delta_\mathrm{def}X,\cdot))|_{\mathcal{F}^{(2)}}=\delta\mathcal{Y}$$
and
\begin{equation}\label{eq:2*}
-(\mathcal{L}_{X}\omega)|_{\mathcal{F}}=-\delta_\calF(\alpha)-d_{\mathcal{F}}\mathcal{Y}.
\end{equation}

Since the foliated 2-form $\omega_\calF$ induces the isomorphism $C^{*}_\mathrm{def}(\calG)_{\calF}\to \Omega^{1}_\mathrm{def}(\calG^{(*)})_\calF$ (see previous Lemma), it follows that $\mathcal{Y}=-i_{Y}\omega_\calF$ for some vector field $Y$ on $\widetilde{\G}$ which is tangent to the fibres $\G_\e$ of $\widetilde{\G}$. Therefore, the first equation is equivalent to

\begin{equation}\label{eq:1*}
\sigma(\omega(\delta_\mathrm{def}X,\cdot))|_{\mathcal{F}^{(2)}}=-\sigma(\omega(\delta_\mathrm{def}Y,\cdot))|_{\mathcal{F}^{(2)}},
\end{equation}
and equation \eqref{eq:2*} becomes

\begin{equation}\label{eq:2**}
(\mathcal{L}_{X+Y}\omega)_{\mathcal{F}}=\delta_\calF(\alpha).
\end{equation}

Thus, equation \eqref{eq:1*} says that $X+Y$ is a multiplicative and transverse vector field on $\widetilde{\G}$. Let us denote by $F_\e$ the flow of $X+Y$ on $\widetilde{\G}$. Then the equation \eqref{eq:2**} is equivalent to
$$\frac{d}{d\e}F_\e^{*}\omega_\calF=F_{\e}^{*}\delta_\calF\left(\alpha\right).$$

Denote by $(\Phi_\e,\phi_\e)$ the family of restricted morphisms $F_\e|_{{\G}_0}:\G_0\to{\G}_\e$. Then, by restricting the previous equation to the fibre ${\G}_0$ we get
\begin{align*}
\frac{d}{d\e}\Phi_\e^{*}\omega_\e&=\Phi_\e^{*}(\delta_{{\G}_\e}(\alpha_\e))\\
&=\delta_{{\G}_0}(\phi_\e^{*}\alpha_\e)\\
&=\delta_{{\G}_0}(\frac{d}{d\e}\bar{\alpha}_\e)\\
&=\frac{d}{d\e}\delta_{{\G}_0}(\bar{\alpha}_\e);
\end{align*}
where $\bar{\alpha}_\t:=\int_{0}^{\t}(\phi_\e^{*}\alpha_\e) d\e$ is the primitive of the curve $\e\mapsto\phi_\e^{*}\alpha_\e$. Hence, the sequence of equalities tells us that the expression $\Phi_\e^{*}\omega_\e-\delta_{\widetilde{\G}_0}(\bar{\alpha}_\e)$ is constant in $\e$. Thus,
$$\Phi_\e^{*}\omega_\e=\omega_0+\delta_{\widetilde{\G}_0}(\bar{\alpha}_\e).$$
Therefore, if $\Phi:\G\times I\to\widetilde{\G}$ is the isomorphism determined by the family $\Phi_\e$ and $\bar{\alpha}\in\Omega^{2}_\mathcal{F}(M\times I)_\mathrm{cl}$ is the  closed foliated 2-form which equals $\bar{\alpha}_\e$ when restricted to the fibre $\G_\e$, then
$$(\Phi^{*}\omega)_{\mathcal{F}}=(pr_\G^{*}\omega_0)_\mathcal{F}+\delta_{\mathcal{F}}(\bar{\alpha}),$$
where $pr_\G:\G\times I\to \G$ is the obvious projection to $\G$. This proves the triviality of the deformation $(\widetilde{\calG},\omega)$.
\end{proof2}

\begin{remark} Alternatively, Theorem \ref{thm.triviality.gen} could be  deduced from Theorem \ref{thm.triviality.s-constant} by first using the Reeb stability theorem for the foliation of $\widetilde{\G}$ by the fibres $\G_\e$ (at least for $\widetilde{\G}$ Hausdorff) and using the triviality of proper deformations of compact Lie groupoids \cite[Thm. 7.4]{CMS}. Nonetheless, it seems to us that it is useful  having a direct proof of the general version of the theorem at hand, without the need to pass through a specific type of deformation. This proof should also leave more room for generalizations to relative or semi-local versions for groupoids admitting well adapted tubular neighborhoods, or normal forms, around submanifolds as for example in \cite[Section 8]{PMCT1}, or \cite{Pedro-Ionut-PT2}.
\end{remark}

%%%%%%%%%%%%%%%%%%%%%%% SECTION %%%%%%%%%%%%%%%%%%%%%%%

\section{More examples and relation with Poisson cohomology}\label{Section-more-examples}

We once more use a spectral sequence in order to compute deformation cohomology, as done in sections \ref{Sec-spectral-sequence-rows} and \ref{Sec-spectral-sequence-column}. This time we turn to the ``columns first'' spectral sequence for the double complex associated to the mapping cone of $-d_{dR}:\Omega^{1}(\G^{(\bullet)})\to\Omega_\mathrm{cl}^{2}(\G^{(\bullet)})$. Its first page is

\begin{small}
\begin{tikzpicture}
  \matrix (m) [matrix of math nodes,
    nodes in empty cells,nodes={minimum width=5ex,
    minimum height=5ex,outer sep=2pt},
    column sep=3ex,row sep=1ex]{
                &      &     &     &  & \\
          1     &  H_{dR}^2(M) &  H_{dR}^2(\G)  & H_{dR}^2(\G^{(2)}) & H_{dR}^2(\G^{(3)}) &\\
          0     &  \Omega^{1}_\mathrm{cl}(M) &  \Omega^{1}_\mathrm{cl}(\G)  & \Omega^{1}_\mathrm{cl}(\G^{(2)})  & \Omega^{1}_\mathrm{cl}(\G^{(3)}) &\\
    \quad\strut &   0  &  1  &  2  & 3 & \strut \\};
  \draw[-stealth] (m-2-2.east) -- (m-2-3.west);
    \draw[-stealth] (m-2-3.east) -- (m-2-4.west);
    \draw[-stealth] (m-2-4.east) -- (m-2-5.west);
    \draw[-stealth] (m-3-2.east) -- (m-3-3.west);
    \draw[-stealth] (m-3-3.east) -- (m-3-4.west);
    \draw[-stealth] (m-3-4.east) -- (m-3-5.west);
\draw[thick] (m-1-1.east) -- (m-4-1.east) ;
\draw[thick] (m-4-1.north) -- (m-4-6.north) ;
\end{tikzpicture}
\end{small}

If we assume that $M$, $\G$ and $\G^{(2)}$ have vanishing first and second de Rham cohomology, then this page becomes

\begin{small}
\begin{tikzpicture}
  \matrix (m) [matrix of math nodes,
    nodes in empty cells,nodes={minimum width=5ex,
    minimum height=5ex,outer sep=2pt},
    column sep=3ex,row sep=1ex]{
                &      &     &     &  & \\
          1     &  0  & 0 & 0 & H_{dR}^2(\G^{(3)}) &\\
          0     &  d(C^\infty(M)) &  d(C^\infty(\G)) & d(C^\infty(\G^{(2)}))  & \Omega^{1}_\mathrm{cl}(\G^{(3)}) &\\
    \quad\strut &   0  &  1  &  2  & 3 & \strut \\};

    \draw[-stealth] (m-3-2.east) -- (m-3-3.west);
    \draw[-stealth] (m-3-3.east) -- (m-3-4.west);
    \draw[-stealth] (m-3-4.east) -- (m-3-5.west);
\draw[thick] (m-1-1.east) -- (m-4-1.east) ;
\draw[thick] (m-4-1.north) -- (m-4-6.north) ;
\end{tikzpicture}
\end{small}

In  this case the second page is 

\begin{small}
\begin{tikzpicture}
  \matrix (m) [matrix of math nodes,
    nodes in empty cells,nodes={minimum width=5ex,
    minimum height=5ex,outer sep=2pt},
    column sep=2ex,row sep=1ex]{
                &      &     &     &  & \\
          1     &  0  & 0 & 0 &  H_\delta(H_{dR}^2(\G^{(3)})) &\\
          0     &  H_\delta(d(C^\infty(M))) &  H_\delta(d(C^\infty(\G))) & H_\delta(d(C^\infty(\G^{(2)})))  & H_\delta(\Omega^{1}_\mathrm{cl}(\G^{(3)})) &\\
    \quad\strut &   0  &  1  &  2  & 3 & \strut \\};
\draw[thick] (m-1-1.east) -- (m-4-1.east) ;
\draw[thick] (m-4-1.north) -- (m-4-6.north) ;
\end{tikzpicture}
\end{small}
and therefore the deformation cohomology of the symplectic groupoid $(\G,\omega)$ is very closely related in degrees up to $3$ (the most important ones for deformation theory) to the differentiable cohomology of $\G$, which is $H_\delta(C^\infty(\G^{(\bullet)}))$.

Recall also that there is a van Est map relating differentiable cohomology of $\G$ with the algebroid cohomology of its Lie algebroid \cite{C}, which is naturally isomorphic to the Poisson cohomology of the base, in the case of a symplectic groupoid.

We now take a closer look at some (families of) symplectic groupoids satisfying these conditions. We start by illustrating a general strategy with the following case.

\subsection{Linear Poisson structures - Part 2}

Let $(\mathfrak{g}^*, \pi_\mathfrak{g})$ be a linear Poisson structure and let $G$ be the $1$-connected Lie group integrating the Lie algebra $\mathfrak{g}$. We recall that the cotangent groupoid $(T^*G, \omega_\mathrm{can})$ is a source 1-connected integration of $(\mathfrak{g}^*, \pi_\mathfrak{g})$, and that $T^*G\cong G\ltimes \mathfrak{g}^*$. 

\begin{proposition}\label{proposition-lin-poisson} Let $\mathfrak{g}$ be a Lie algebra and let $G$ be the $1$-connected Lie group integrating it. Then there are isomorphisms 

 \[H_\mathrm{def}^0(T^*G,\omega_\mathrm{can})\cong H^0_{\pi_\mathfrak{g}}(\mathfrak{g}^*)/\mathbb{R}, \ \ \ H_\mathrm{def}^i(T^*G,\omega_\mathrm{can})\cong H^i_{\pi_\mathfrak{g}}(\mathfrak{g}^*),\] for $i=1,2$, and an injective map \[H_\mathrm{def}^3(T^*G,\omega_\mathrm{can})\to H^3_{\pi_\mathfrak{g}}(\mathfrak{g}^*).\]
\end{proposition}
\begin{proof}

Using that $T^*G\cong G\ltimes \mathfrak{g}^*$, we compute the deformation cohomology of $(T^*G,\omega_\mathrm{can})$ using the ``columns first'' spectral sequence for the double complex associated to the mapping cone of $d:\Omega^1((G\ltimes \mathfrak{g}^*)^{(\bullet)})\to\Omega^2_\mathrm{cl}((G\times \mathfrak{g}^*)^{(\bullet)})$. Its first page is

\begin{small}
\begin{tikzpicture}
  \matrix (m) [matrix of math nodes,
    nodes in empty cells,nodes={minimum width=5ex,
    minimum height=5ex,outer sep=2pt},
    column sep=3ex,row sep=1ex]{
                &      &     &     &  & \\
          1     &  H_{dR}^2(\mathfrak{g}^*) &  H_{dR}^2(G\ltimes \mathfrak{g}^*)  & H_{dR}^2((G\ltimes \mathfrak{g}^*)^{(2)}) & H_{dR}^2((G\ltimes \mathfrak{g}^*)^{(3)}) &\\
          0     &  \Omega^{1}_\mathrm{cl}(\mathfrak{g}^*) &  \Omega^{1}_\mathrm{cl}(G\ltimes \mathfrak{g}^*)  & \Omega^{1}_\mathrm{cl}((G\ltimes \mathfrak{g}^*)^{(2)})  & \Omega^{1}_\mathrm{cl}((G\ltimes \mathfrak{g}^*)^{(3)}) &\\
    \quad\strut &   0  &  1  &  2  & 3 & \strut \\};
  \draw[-stealth] (m-2-2.east) -- (m-2-3.west);
    \draw[-stealth] (m-2-3.east) -- (m-2-4.west);
    \draw[-stealth] (m-2-4.east) -- (m-2-5.west);
    \draw[-stealth] (m-3-2.east) -- (m-3-3.west);
    \draw[-stealth] (m-3-3.east) -- (m-3-4.west);
    \draw[-stealth] (m-3-4.east) -- (m-3-5.west);
\draw[thick] (m-1-1.east) -- (m-4-1.east) ;
\draw[thick] (m-4-1.north) -- (m-4-6.north) ;
\end{tikzpicture}\end{small}

For the action groupoid $G\ltimes \mathfrak{g}^*$, the nerve is $(G\ltimes \mathfrak{g}^*)^{(n)}\cong G^n\times \mathfrak{g}^*$. Because $G$ is a $1$-connected Lie group, and since $\pi_2(G)=0$ for any Lie group, we conclude by Hurewicz's Theorem that $(G\ltimes \mathfrak{g}^*)^{(n)}$ has vanishing first and second de Rham cohomology for any $n$.  

Thus, the first line of the page above vanishes, and the de Rham differential induces a surjective cochain map onto the $0$-th line, $d: C^\infty(G^\bullet\times \mathfrak{g}^*)\to \Omega_\mathrm{cl}^1(G^\bullet\times \mathfrak{g}^*)$. The first page becomes 

\begin{small}
\begin{tikzpicture}
  \matrix (m) [matrix of math nodes,
    nodes in empty cells,nodes={minimum width=5ex,
    minimum height=5ex,outer sep=2pt},
    column sep=3ex,row sep=1ex]{
                &      &     &     &  & \\
          1     &  0 &  0  & 0 & 0 &\\
          0     &  dC^\infty(\mathfrak{g}^*) &  dC^\infty(G\ltimes \mathfrak{g}^*)  & dC^\infty((G\ltimes \mathfrak{g}^*)^{(2)})  & dC^\infty((G\ltimes \mathfrak{g}^*)^{(3)}) &\\
    \quad\strut &   0  &  1  &  2  & 3 & \strut \\};
  \draw[-stealth] (m-2-2.east) -- (m-2-3.west);
    \draw[-stealth] (m-2-3.east) -- (m-2-4.west);
    \draw[-stealth] (m-2-4.east) -- (m-2-5.west);
    \draw[-stealth] (m-3-2.east) -- (m-3-3.west);
    \draw[-stealth] (m-3-3.east) -- (m-3-4.west);
    \draw[-stealth] (m-3-4.east) -- (m-3-5.west);
\draw[thick] (m-1-1.east) -- (m-4-1.east) ;
\draw[thick] (m-4-1.north) -- (m-4-6.north) ;
\end{tikzpicture}
\end{small}

Since both $G$ and $\mathfrak{g}^*$ are connected, the kernel of $d$ consists of constant functions, forming the complex $\mathbb{R}\stackrel{0}{\to}\mathbb{R}\stackrel{id}{\to}\mathbb{R}\stackrel{0}{\to}\mathbb{R}\stackrel{id}{\to}\cdots$. Its cohomology vanishes in all degrees except $0$, where it is isomorphic to $\mathbb{R}$. Therefore, the terms of this spectral sequence stabilize at the second page, where we have on the $0$-th line that $H_\delta(\Omega_\mathrm{cl}^1(\mathfrak{g}^*)))\cong H_\mathrm{diff}^0(G\ltimes \mathfrak{g}^*)/\mathbb{R}$ and for $i>0$, $H_\delta(\Omega_\mathrm{cl}^1(G^i\times \mathfrak{g}^*))\cong H_\mathrm{diff}^i(G\ltimes \mathfrak{g}^*)$.

Finally, since the $s$-fibres of $G\ltimes \mathfrak{g}^*$ are $2$-connected, the van Est map relating differentiable and algebroid cohomology is an isomorphism up to degree 2, and injective in degree 3. The statement follows, as the algebroid cohomology of the algebroid of $G\ltimes \mathfrak{g}^*$ is naturally isomorphic to the Poisson cohomology of $\mathfrak{g}^*$.
\end{proof}
 
\subsection{Groupoids with homologically 2-connected nerve}

As is clear from the proof of Proposition \ref{proposition-lin-poisson}, a similar situation happens for other symplectic groupoids $(\G,\omega)\rightrightarrows  (M,\pi)$ with simple enough topology. We will see examples of them in the rest of this section. First of all, exactly the same arguments as for the proof of Proposition \ref{proposition-lin-poisson} let us prove the following results.

\begin{proposition} Let $(\G,\omega) \rightrightarrows (M,\pi)$ be a  symplectic groupoid such that for every $n\geq 0$ the space $\G^{(n)}$ is homologically $2$-connected. Then the de Rham differential $d:C_\mathrm{diff}^\bullet(\G)\to \Omega^1(\G^{(\bullet)})$ induces isomorphisms \[H_\mathrm{diff}^0(\G)/\mathbb{R}\cong H_\mathrm{def}^0(\G,\omega), \ \ \ H_\mathrm{diff}^i(\G)\cong H_\mathrm{def}^i(\G,\omega)\ \ \text{ for }i>0.\]
\end{proposition}

\begin{corollary}\label{cor-simple-top} Let $(\G,\omega) \rightrightarrows (M,\pi)$ be a  symplectic groupoid with homologically $k$-connected source fibres, and such that $\G^{(n)}$ is homologically $2$-connected for each $n\geq 0$.
Then \[H_\mathrm{def}^0(\G,\omega)\cong H^0_\pi(M)/\mathbb{R}, \ \ \ H_\mathrm{def}^i(\G,\omega)\cong H^i_\pi(M),\] for $1\leq i\leq k$, and there is an injective map \[H_\mathrm{def}^{k+1}(\G,\omega)\to H^{k+1}_\pi(M).\]
\end{corollary}

\begin{remark} If $(\G,\omega)$ satisfies the conditions of Corollary \ref{cor-simple-top} for $k=2$, then it also satisfies Proposition \ref{rows-spectral-seq}. Combining the two, we obtain \[H^2_\pi(M)\cong H^2_\mathrm{mult}(\G)\oplus (H^2_\mathrm{def}(\G))_\mathrm{cl}.\]
\end{remark}

Note that the maps inducing the isomorphisms in these results are the de Rham differential $d:C^\bullet(\G)\to \Omega^1(\G^{(\bullet)})$ and the van Est map $VE:H^\bullet(\G)\to H^\bullet_\pi(M)$. We will revisit them in Theorem \ref{global of i} (see also Remark \ref{rmk-i-is-d}) for arbitrary source connected symplectic groupoids, without further topological assumptions.

There are several examples of symplectic groupoids satisfying these topological conditions. For example, let $\G\rightrightarrows M$ be a Lie groupoid such that the source map $s$ is a fibration with $k$-connected fibres, over a $k$-connected manifold $M$. Then $\G^{(n)}$ is $k$-connected for all $n$. This can be seen by inductively considering the long exact sequences in homotopy groups for the source map $s:\G\to M$ and for the fibrations $s^*:\G^{(n)}=\G \tensor[_s]{\times}{_t} \G^{(n-1)}\to \G^{(n-1)}$ (the pullback of $s$ along $(t\circ {pr}_1):\G^{(n-1)}\to M$).

\subsection{Zero Poisson structures}

Let us focus on the source $1$-connected symplectic groupoid $(T^*M, \omega_\mathrm{can})\rightrightarrows (M,0)$ integrating the zero Poisson structure on $M$. In this case, the nerve is $T^*M^{(n)}=\oplus^n T^*M$, so $H^k_\mathrm{dR}(T^*M^{(n)})\cong H_\mathrm{dR}^k(M)$, for all $k, n \geq 0$. The source fibres are contractible, being vector spaces.
We also note that for the zero Poisson bivector $\pi=0$, the differential of the Poisson complex vanishes, so $H^\bullet_\pi(M,0)=\mathfrak{X}^\bullet(M)$.
Thus, using Corollary \ref{cor-simple-top}, we come to the following result.

\begin{corollary}Let $M$ be a homologically $2$-connected manifold. Then \[H^i_\mathrm{def}(T^*M,\omega_\mathrm{can})\cong H^i_\pi(M,0)\cong \mathfrak{X}^i(M)\] for $i\geq 1$, and $H^0(T^*M,\omega_\mathrm{can})\cong C^\infty(M)/\mathbb{R}$.
\end{corollary}

\subsection{Poisson-Lie groups}

A Poisson-Lie group is a Lie group $G$ equipped with a Poisson structure $\pi_G$ for which the multiplication $m:G\times G\to G$ is a Poisson map, where $G\times G$ is equipped with the product Poisson structure. This amounts to the bivector $\pi_G$ being \textit{multiplicative}. Multiplicativity of $\pi_G$ implies that  it vanishes at the identity $e\in G$, so its linearization at $e$ induces a linear Poisson structure on the Lie algebra $\mathfrak{g}$. Therefore it induces a Lie algebra structure on the dual $\mathfrak{g}^*$. The dual Lie group of $G$ is then the $1$-connected Lie group integrating the Lie algebra $\mathfrak{g}^*$, and is denoted by $G^*$. It is also a Poisson-Lie group, which we denote $(G^*,\pi_{G^*})$.

For each $\xi\in\mathfrak{g}^*$, denote by $\xi^l$ and $\xi^r$ the left-invariant and the right-invariant 1-forms on $G$ with value $\xi$ at the identity, respectively. The maps \[\lambda,\rho:\mathfrak{g}^*\to \mathfrak{X}(G),\ \ \ \lambda(\xi)=\pi_G^\sharp(\xi^l),\ \ \ \rho(\xi)=-\pi_G^\sharp(\xi^r)\] define the left and right infinitesimal dressing actions of $\mathfrak{g}^*$ on $G$, respectively.
We say that a multiplicative Poisson tensor $\pi_G$ on $G$ is complete if each left (or, equivalently, each right) dressing vector field is complete on $G$.

In \cite{Lu_Weinstein}, Lu and Weinstein have constructed, for any Poisson-Lie group $G$, a symplectic groupoid $\Gamma$ integrating $G$, with a second compatible groupoid structure on $\Gamma$ making it into a symplectic groupoid integrating $G^*$. In general $\Gamma$ is a submanifold of $G\times G^*\times G^*\times G$. In the simplest case, when $(G, \pi_G)$ is a simply connected complete Poisson-Lie group, $\Gamma$ is $G\times G^*$, which is in this case also diffeomorphic to the double group $D$ of $G$; the groupoid structures on $D$ are in fact action Lie groupoid structures, for actions of $G$ and $G^*$ on each other obtained by slightly modifying the dressing actions (see \cite[Section 4.1]{Lu_thesis}). 

\begin{corollary}\label{simple-top-PLgroups} Let $(G,\pi_G)$ be a simply connected complete Poisson-Lie group with dual Lie group $G^*$ and 1-connected integration denoted by $(D, \omega)\rightrightarrows (G,\pi_G)$. 

Then there are isomorphisms 

 \[H_\mathrm{def}^0(D,\omega)\cong H^0_{\pi_G}(G)/\mathbb{R}, \ \ \ H_\mathrm{def}^i(D,\omega)\cong H^i_{\pi_G}(G),\] for $i=1,2$, and an injective map \[H_\mathrm{def}^3(D,\omega)\to H^3_{\pi_G}(G).\]

\end{corollary}
\begin{proof} Both $G$ and $G^*$ are $2$-connected, and so $D\rightrightarrows G$ has $2$-connected base $G$ and source fibres (diffeomorphic to $G^*$).  The source map of $D$ is a fibration, since $D$ is an action groupoid. Therefore, $D\rightrightarrows G$ satisfies the conditions of Corollary \ref{cor-simple-top} with $k=2$.
\end{proof}

Exchanging the roles of $G$ and $G^*$ and bearing in mind that a Poisson-Lie group $G$ is complete if and only if $G^*$ is complete \cite[Prop. 2.42]{Lu_thesis}, one obtains an analogous result for the symplectic groupoid structure of $D$ over $G^*$.
When $G$ is equipped with the zero Poisson structure, then $G^*$ is $\mathfrak{g}^*$ seen as an abelian group, with the linear Poisson structure, and $D=T^*G$; so in this sense Corollary \ref{simple-top-PLgroups} recovers Proposition \ref{proposition-lin-poisson} for complete $G$. Other interesting, and non-trivial, Poisson-Lie group structures have been constructed, for example for connected compact semisimple $G$ \cite{Lu-Weinstein-Bruhat}.

\subsection{Cotangent VB-groupoids}
For any cotangent VB-groupoid $T^*\G \rightrightarrows A^*$ of a Lie groupoid $\G$, the nerve $(T^*\G)^{(\bullet)}$ is a simplicial vector bundle over the nerve of $\G$. Moreover, the core exact sequence of $T^*\G$ (dual to sequence $(2)$ in Remark \ref{rmk-core-exact-seq-TG}) implies that the source fibres of $T^*\G$ are affine bundles over the source fibres of $\G$. These facts together with Corollary \ref{cor-simple-top} lead to the following result.

\begin{corollary}\label{Prop-cotang-grpd} Let $\G$ be a Lie groupoid such that $\G^{(n)}$ is homologically $2$-connected for all $n\geq 0$ and such that its source map $s$ has homologically $k$-connected fibres. Then there are isomorphisms 

 \[H_\mathrm{def}^0(T^*\G,\omega_\mathrm{can})\cong H^0_\pi(A^*)/\mathbb{R}, \ \ \ H_\mathrm{def}^i(T^*\G,\omega_\mathrm{can})\cong H^i_\pi(A^*),\] for $1\leq i\leq k$, and an injective map \[H_\mathrm{def}^{k+1}(T^*\G,\omega_\mathrm{can})\to H^{k+1}_\pi(A^*).\]
\end{corollary}

\begin{remark}
The results of this section indicate that at the infinitesimal level, the deformation theory of a symplectic groupoid $(\G,\omega)\rightrightarrows  (M,\pi)$ with homologically $2-$connected nerve and source fibres is very closely related with that of $(M,\pi)$ itself. And moreover, that this situation occurs in many natural examples.

Both $H_\mathrm{def}^2(\G,\omega)$ and $H^2_\pi(M)$ encode infinitesimal deformations modulo trivial deformations, while $H_\mathrm{def}^1(\G,\omega)$ and $H^1_\pi(M)$ encode infinitesimal automorphisms modulo trivial ones.
The missing piece, from the point of view of deformation theory, is relating obstructions. On the Poisson side, $H^3_\pi(M)$ encodes obstructions to extending an infinitesimal deformation to a formal one. This interpretation uses the Schouten-Nijenhuis differential graded Lie algebra (DGLA) structure on $\mathfrak{X}^\bullet(M)$.

 There is a well known principle of deformation theory due to Deligne \cite{Deligne} and Drinfeld \cite{Drinfeld}, and in an equivalent form, to Schlessinger and Stasheff \cite{Schlessinger_Stasheff}, (now formalized as a theorem in derived deformation theory \cite{lurie,Pridham}) and explored further by many others; it states that every reasonable formal deformation problem in characteristic zero is controlled by a DGLA $\mathfrak{g}$. Equivalence classes of infinitesimal deformations are given by elements $[c]\in H^2(\mathfrak{g})$; obstructions to extending $c$ to a formal deformation are detected by a squaring map $[\cdot,\cdot]:H^2(\mathfrak{g)}\to H^3(\mathfrak{g})$ (degrees may vary according to the index convention for $\mathfrak{g}$).

It is not known to date how to describe such a DGLA (nor an $L_\infty$-algebra) structure on the deformation complex of a Lie groupoid $\G$ (not even when $\G$ is a Lie group!). The work of \cite{Joost} points to its existence, although in a non-constructive way, via an interpretation of deformations of Lie groupoids as deformations of certain kinds of diagrams of $C^\infty$-schemes. Similarly, we do not know how to describe such a structure for $C_\mathrm{def}^\bullet(\G,\omega)$.
It is our hope that understanding better the close relation between the deformation complexes of $\G$ and of $(\G, \omega)$ with known DGLA's (such as $\mathfrak{X}^\bullet(M)$, as in this section, or the deformation complex of a Lie algebroid, in the next) may help in eventually finding an explicit description of the DGLA (or $L_\infty$-algebra) structures on these complexes.

In any case, if $H^3_\mathrm{def}(\G,\omega)$ does encode obstructions to extending an infinitesimal deformation of a symplectic groupoid to a formal one, it is reasonable to expect that the map of Proposition \ref{proposition-lin-poisson} sends the obstruction for an infinitesimal deformation $[\eta]$ of $(\G,\omega)$ to the obstruction of a corresponding infinitesimal deformation $[\Lambda]$ of Poisson structures on the base. If that is the case, injectivity would say that if $[\Lambda]$ is unobstructed, then so is $[\eta]$. On the other hand, if $[\eta]$ is the deformation class associated to an actual deformation of $(\G,\omega)$ then the corresponding infinitesimal deformation $[\Lambda]$ of Poisson structures should also be unobstructed.
\end{remark}

%%%%%%%%%%%%%%%%%% SECTION %%%%%%%%%%%%%%%%%%%%%

\section{The map between differentiable and deformation cohomologies}\label{i_calG}

In the previous section we have explored the relation between deformation cohomology and Poisson cohomology for symplectic groupoids with somewhat simple topology. We now turn our attention to what can be said about the relation between the deformation cohomology of any symplectic groupoid $(\G, \omega)$ and other related cohomologies, such as $H^\bullet(\G)$ and Poisson cohomology of the base.

Let $(M,\pi)$ be a Poisson manifold and let $(T^{*}M)_{\pi}$ be its associated Lie algebroid. In the study of deformation of Lie algebroids of \cite{CM}, the authors define a map $i:H^\bullet_{\pi}(M)\longrightarrow H^\bullet_\mathrm{def}\left((T^{*}M)_{\pi}\right)$ between the Poisson cohomology of $(M,\pi)$, which controls deformations of Poisson structures, and the deformation cohomology of the algebroid $(T^{*}M)_{\pi}$. 
In this section we construct the global counterpart $i_{\calG}$ of the map $i$.

We collect in the appendix some material, including definitions, on double vector bundles and VB-algebroids (vector bundle objects in the category of Lie algebroids), that is of use in the proofs of this section.

\subsection{Deformation cohomology of Lie algebroids}\label{Deform.algbrds}

We recall the deformation complex of a Lie algebroid $A$, denoted by $(C^{\bullet}_\mathrm{def}(A),\delta)$ and defined in \cite{CM}; we recall also its interpretation in terms of VB-algebroids.

A derivation on a vector bundle $E\longrightarrow M$ is a linear operator $D:\Gamma(E)\longrightarrow\Gamma(E)$ such that there exists a vector field $\sigma_D\in\mathfrak{X}(M)$, called the symbol of $D$, which satisfies
$$D(fs)=fD(s)+\sigma_D(f)s,\ \text{for } s\in\Gamma(E) \text{ and } f\in C^{\infty}(M).$$
A \textbf{multiderivation of degree} $n$ on $E$ is a multilinear and antisymmetric map
$$D:\underbrace{\Gamma(E)\times\cdots\times\Gamma(E)}_{n+1\ \mathrm{times}}\longrightarrow\Gamma(E)$$
which is a derivation in each entry, i.e., there is a \textbf{symbol map}
$$\sigma_D:   \underbrace{\Gamma(E)\times\cdots\times\Gamma(E)}_{n\ \mathrm{times}}\longrightarrow\mathfrak{X}(M)$$
which is $C^{\infty}(M)$-linear in each entry and satisfies
\[D(s_0,s_1,...,fs_{n})=fD(s_0,...,s_{n})+\sigma_D(s_0,...,s_{n-1})(f)s_{n},\] for $s_i\in\Gamma(E)$ and $f\in C^{\infty}(M).$
The space of multiderivations of degree $n$ is often denoted by $Der^{n}(E)$. One sets $Der^{-1}(E)=\Gamma(E)$.

The \textbf{deformation complex of} a Lie algebroid $A$ is the complex for which the space of $k$-cochains $C^{k}_\mathrm{def}(A)$ consists of the multiderivations of degree $k-1$ on $A$, i.e., $C^{k}_\mathrm{def}(A)=Der^{k-1}(A)$, with differential given by the 'de Rham type' formula
\begin{equation*}
\begin{split}
  \delta(D)(\alpha_0,...,\alpha_{k})=\Sigma_{i}(-1)^{i}[\alpha_i,&D(\alpha_0,...,\hat{\alpha}_{i},...,\alpha_k)]+\\
  &+\Sigma_{i<j}(-1)^{i+j}D([\alpha_i,\alpha_j],\alpha_0,...,\hat{\alpha}_i,...,\hat{\alpha}_j,...,\alpha_k).
  \end{split}
\end{equation*}

\paragraph{\textbf{VB-algebroid complex}}

For a VB-algebroid  $D\longrightarrow E$ over $A\longrightarrow M$, the \textbf{VB-algebroid complex} of $D$  is the subcomplex $C^{\bullet}_\mathrm{VB}(D)=C^{\bullet}_\mathrm{lin}(D)$ of $C^{\bullet}(D)$ (the Chevalley-Eilenberg complex of $D$) consisting  of cochains which are linear in the following sense: regard the cochains in $C^{k}(D)=\Gamma(E,\Lambda^{k}D^{*}_{E})$ as the $k$-multilinear and alternating functions $\oplus^{k}_{E}D\longrightarrow\mathbb{R}$; a $k$-cochain is called \textbf{linear} if it is fibrewise linear with respect to the vector bundle structure of $\oplus^k_{E}D$ over $\oplus^k_{M}A$ (cf.\ \cite{CD}, see also \cite{CM}).

\begin{remark}\label{D_{A^{*}}}
As pointed out in (\cite{CM}, Prop. 7), there is an interpretation of the deformation complex of the Lie algebroid $A$ in terms of the VB-algebroid complex of $T^{*}A^{*}\longrightarrow A^{*}$ (see example \eqref{ex:cotangentalgbrd}).
In fact, given any vector bundle $E\to M$, define the isomorphism (\cite{CM}, section 4.9) \[D_{E}:\mathfrak{X}_\mathrm{lin}^{k}(E)\longrightarrow Der^{k-1}(E^{*}),\ \ \ D_E(X)(s_1,...,s_k):=X(l_{s_1},...,l_{s_k}),\] where $s_i\in\Gamma(E^{*})$ and $l:\Gamma(E^{*})\stackrel{\cong}{\longrightarrow}C^{\infty}_\mathrm{lin}(E)$ is the function assigning to a section $s\in\Gamma(E^{*})$ the corresponding (fibrewise) linear function on $E$. Thus, if $E=A^{*}$, then $\mathfrak{X}^{k}_\mathrm{lin}(A^{*})=C^{k}_\mathrm{VB}(T^{*}A^{*})$, and $D_{A^{*}}:\mathfrak{X}^{\bullet}_\mathrm{lin}(A^{*})\longrightarrow C^{\bullet}_\mathrm{def}(A)$ is moreover an isomorphism of differential graded Lie algebras, where on $\mathfrak{X}^{\bullet}_\mathrm{lin}(A^{*})$ one considers the Schouten bracket of multivector fields.
\end{remark}

\paragraph{\textbf{The map $i$}} Let $(M,\pi)$ be a Poisson manifold. We now revisit the cochain map $i:C^{\bullet}_{\pi}(M)\rightarrow C^{\bullet}_\mathrm{def}(T^{*}M)$ introduced in \cite{CM}. It is defined in terms of cochains by $i:C^{k}_{\pi}(M)\rightarrow C^{k}_\mathrm{def}(T^{*}M)$, $X\mapsto D_{X}$, where $D_{X}\in Der^{k-1}(T^{*}M)$ acts on exact forms by
$$D_{X}(df_{1},..., df_k):=d(X(f_1,...,f_k)).$$
An expression for the multiderivation $D_{X}$ acting on arbitrary forms can also be obtained (see \cite{CM}, Prop. 3), but the previous formula is enough for our purposes.

In order to can study the map $i$ in another way, we use the notion of the \textbf{tangent lift} $T_{\mathfrak{X}}:\mathfrak{X}^{k}(M)\longrightarrow\mathfrak{X}^{k}_\mathrm{lin}(TM)\subset\mathfrak{X}^{k}(TM)$ of $k$\textbf{-multivector fields}, defined as follows. Regard the $k$-vector fields as the $k$-multilinear functions

$$\mathfrak{X}^{k}(M)=\Gamma(\bigwedge^{k}_{p_M}TM)=\{\bigoplus^{k}_{c_M}T^{*}M\stackrel{k\text{-multilinear}}{\longrightarrow}\mathbb{R}\}.$$
Thus,
$$\mathfrak{X}^{k}(TM)=\Gamma(\bigwedge^{k}_{p_{TM}}T(TM))=\{\bigoplus^{k}_{c_{TM}}T^{*}(TM)\stackrel{k\text{-multilinear}}{\longrightarrow}\mathbb{R}\}.$$
Then, if $X\in\mathfrak{X}^{k}(M)$, its tangent lift $\tilde{X}$ is given by
\begin{equation}\label{eq:tangliftvector}
\tilde{X}:=TX\circ\oplus^k\Theta_{TM}^{-1}:\bigoplus^{k}_{c_{TM}}T^{*}(TM)\longrightarrow\bigoplus^{k}_{Tc_{M}}T(T^{*}M)\longrightarrow\mathbb{R};
\end{equation}
where $TX\in C^{\infty}_{k-lin}(\bigoplus^{k}_{Tc_{M}}T(T^{*}M))$ means the tangent lift, induced by the differential, of the $k$-multilinear function corresponding to the vector field $X$.

Notice further that, in this way, $\tilde{X}$ is (fibrewise) linear with respect to the vector bundle structure of $\bigoplus^{k}_{c_{TM}}T^{*}(TM)$ over $\bigoplus^{k}_{c_M}T^{*}M$ ($TX$ is (fibrewise) linear with respect to the bundle projection $\oplus^{k} p_{T^{*}M}$, and $\oplus^{k}\Theta_{TM}$ is an isomorphism of DVBs). That is, $\tilde{X}\in\mathfrak{X}^{k}_\mathrm{lin}(TM)$, i.e., the tangent lift of the multivector field $X\in\mathfrak{X}^{k}(M)$ is a \textit{linear} multivector on $TM$.

\bigskip

We identify the map $i$ with the tangent lift of multivector fields, through the isomorphism $D_{TM}$ of multiderivations with linear multivector fields.The map $i$ is determined by the composition
\begin{equation}\label{i}\mathfrak{X}^{k}(M)\stackrel{T_{\mathfrak{X}}}{\longrightarrow}\mathfrak{X}_\mathrm{lin}^{k}(TM)\stackrel{D_{TM}}{\longrightarrow}Der^{k-1}(T^{*}M).
\end{equation}

For simplicity we will verify the equality in the case $k=2$. Let \[l:\mathfrak{X}^{2}(M)\longrightarrow C^{\infty}_{2-\mathrm{lin}}(\bigoplus^{2}_{c_M}T^{*}M)\] be the correspondence assigning to bivector fields on a manifold $M$ the associated bilinear antisymmetric functions on $\bigoplus^{2}_{c_M}T^{*}M$. With this assignment, one can write the element $TX$ in expression \eqref{eq:tangliftvector} as $T(l_X)\in C^{\infty}_{k-\mathrm{lin}}(\bigoplus^{k}_{Tc_{M}}T(T^{*}M))$. Thus, one gets
\begin{align*}
D_{TM}(\tilde{X})(df_1,df_2)&=\tilde{X}(l_{df_1},l_{df_2})\\
  %&=\tilde{X}(d\tilde{f}_1\wedge d\tilde{f}_2)...noo
  %&(=\tilde{X}(\tilde{f_1},\tilde{f_2}))\\..noo
  &=l_{\tilde{X}}(d\tilde{f}_1,d\tilde{f}_2)\ \ (\text{where }\tilde{f}_i:=l_{df_i}\text{is the tangent lift of }f_i)\\
  &=T(l_{X})\circ \Theta^{-1}_{TM}(d\tilde{f}_1,d\tilde{f}_2)\\  
  &=T(l_{X})(T(df_1),T(df_2))\\
  &=T(l_{X}(df_1,df_2))\\
  &=d(l_{X}(df_1,df_2))\\
  &=d(X(f_1,f_2))\ \ (\text{by definition of } l)\\
  &=i(X)(df_1,df_2). 
\end{align*}
  where in the fourth equality we use that $\Theta_{TM}\circ T(df)=d\tilde{f}$ (see \cite{M} p. 394).\\

\subsection{The map $i_\calG$ and van Est commutativity}

We now define the map \[i_\calG:H^{\bullet}_{\mathrm{diff}}(\calG)\longrightarrow H^{\bullet}_\mathrm{def}(\calG)\] which will be regarded as the global counterpart of the map $i$ in the sense of Theorem \ref{global of i} below. As we will see, similarly to what happened for $C^{\bullet}_\mathrm{def}(\G,\omega)$ (Remark \ref{Rem_sub_BSS}) the map $i_\G$ can be identified with a piece of the Bott-Shulman-Stasheff double complex.

Recall from Proposition \ref{deformation and vb-complexes} that $C^{\bullet}_\mathrm{def}(\calG)\cong C^{\bullet}_\mathrm{VB}(T^{*}\calG)$, and moreover, by Lemma 3.1 in \cite{CD}, that $H^{\bullet}_\mathrm{VB}(T^{*}\calG)\cong H^{\bullet}_\mathrm{lin}(T^{*}\calG)$.

Thus, we will define the map (between cohomologies!) $i_\calG: H^{\bullet}(\calG)\longrightarrow H^{\bullet}_\mathrm{def}(\calG)$ in terms of a chain map $j:C^{\bullet}(\calG)\longrightarrow C^{\bullet}_\mathrm{lin}(T^{*}\calG)$; the latter is the composition
$$C^{k}(\calG)\stackrel{T}{\longrightarrow}C^{k}_\mathrm{lin}(T\calG)\stackrel{(\omega^{\#})^{*}}{\longrightarrow} C^{k}_\mathrm{lin}(T^{*}\calG),$$
where $\omega^{\#}:=(\omega^{b})^{-1}:T^{*}\calG\rightarrow T\calG$, and $T$ is the natural \textbf{tangent lift of groupoid cochains}. It is given by $T(c):(v_{g_1},...,v_{g_k})\mapsto dc(v_{g_1},...,v_{g_k})$, using the canonical identification $(T\calG)^{(k)}\cong T\calG^{(k)}$.      

The main property of the map $i_\calG$ is that for $s$-connected symplectic groupoids, it is indeed the global counterpart of the map $i$ of \cite{CM}.

\begin{theorem}\label{global of i}
Let $(\calG,\omega)$ be a symplectic groupoid. If $\calG$ is $s$-connected, the map $i_{\calG}:C^{\bullet}(\calG)\stackrel{j}{\rightarrow}C^{\bullet}_\mathrm{lin}(T^{*}\calG)\stackrel{q.i.}{\sim} C^{\bullet}_\mathrm{def}(\calG)$, defined above, is the global counterpart of the map $i$. That is, $i_{\calG}$ together with $i$ and the van Est maps for differentiable \cite{C} and deformation cohomology \cite{CD,CMS} form the commutative diagram.

\begin{equation*}\label{Global}
%\[
  \begin{tikzcd}[column sep=4em, row sep=10ex]
    C^{k}(\calG) \MySymb[\circlearrowright]{dr} \arrow{d}[swap]{VE} \arrow{r}{i_{\calG}} & \makebox[5em][l]{$C^{*}_\mathrm{lin}(T^{*}\calG)\stackrel{q.i.}{\sim} C^{*}_\mathrm{def}(\calG)$} \arrow{d}{VE_\mathrm{def}}\\
    C^{k}(T^{*}M) \arrow{r}{i} & C^{k}_\mathrm{def}(T^{*}M).
  \end{tikzcd}
%\]
\end{equation*}
\end{theorem}

In the statement of this result, and throughout the rest of this section, we let $C^\bullet_\mathrm{diff}(\G)$, or simply $C^\bullet(\G)$ (respectively $C^\bullet_\mathrm{VB}(\Gamma)$ and $C^\bullet_\mathrm{lin}(\Gamma)$) denote the subcomplex of normalized cochains of $\G$ (respectively normalized VB-cochains and linear cochains of $\Gamma$), which we recall are those that vanish on all degeneracies.

\begin{remark}\label{rmk-i-is-d}
The map $i_\G:H^{\bullet}_\mathrm{diff}(\G)\to H^{\bullet}_\mathrm{def}(\G)$, according to its definition, can  essentially be viewed as the de Rham differential $d_{dR}:H^{\bullet}_\mathrm{diff}(\G)\to H^{\bullet}(\Omega^{1}(\G^{(*)}))$. In fact, the two maps are related by the following diagram.

\[\xymatrix{H^{\bullet}_\mathrm{def}(\G) \ar[rr]^{(\omega^{b})^{*}} & &H^{\bullet}(\Omega^{1}(\G^{(*)}))\\
 & H^{\bullet}_\mathrm{diff}(\G) \ar[ur]_{d_{dR}} \ar[ul]^{i_\G}. & }\]

Therefore, just as $C^\bullet_\mathrm{def}(\G,\omega)$ could be identified with the total complex of a sub complex of the Bott-Shulman-Stasheff double complex (Remark \ref{Rem_sub_BSS}), the map $i_\G$ appears as the vertical differential between the first two lines of the double complex.

Note also that the maps \[ \xymatrix{ & C^\bullet(\G) \ar[dr]^{i_\G} \ar[dl]_{VE}. & \\ C^\bullet(T^*M) & & C_\mathrm{lin}^\bullet(T^*\G)
}\] appearing in the statement were precisely the ones used to relate deformation cohomology of $(\G,\omega)$ and Poisson cohomology of the base in Section \ref{Section-more-examples}, in cases with simple topology.

\end{remark}

The proof of Theorem \ref{global of i} makes use of the \textbf{tangent lift of algebroid cochains}, $T:C^\bullet(A)\to C^\bullet_{\mathrm{lin}}(TA)$ which is the infinitesimal analogue of the tangent lift of groupoid cochains, and is defined as follows. 
A $k$-cochain $c\in C^{k}(A)=\Gamma(\Lambda^{k}A^{*})$ can be regarded as a $k$-linear and skew-symmetric map $c:\bigoplus^{k}A\to\mathbb{R}$. Its tangent lift $Tc\in C^{k}_\mathrm{lin}(TA)$ is
$$Tc(v_1,...v_k):=dc(v_1,...,v_k),$$
for $(v_1,...,v_k)\in\bigoplus^{k}_{TM}TA$, using the identification $T(\bigoplus^{k}A)\cong\bigoplus^{k}_{TM}TA$.

It is shown in Section \ref{sec-tangent-lift} of the Appendix that the tangent lift is a map of cochain complexes, and we detail some of its properties in Lemma \ref{tangent lift of algebroid-cochains}.

\begin{remark}
Notice that if $A=T^{*}M$ is the Lie algebroid associated to a Poisson manifold $(M,\pi)$, then the tangent lift of algebroid cochains agrees with the tangent lift $TX$ of multilinear elements used in expression \eqref{eq:tangliftvector} to define the tangent lift of multivector fields on $M$.
\end{remark}

\begin{proof2}\textbf{of Theorem} \ref{global of i}

The proof follows from working on each of the properties of the symplectic form $\omega$ (non-degeneracy, multiplicativity and closedness). On the one hand, notice first that, from non-degeneracy and multiplicativity, the map $j$ fits into the following diagram

\begin{equation}\label{d1}
  \begin{tikzcd}[column sep=4em, row sep=10ex]
    C^{k}(\calG) \MySymb[(I)]{ddr} \arrow{dd}{VE} \arrow{r}{T} & C^{k}_\mathrm{lin}(T\calG) \MySymb[(II)]{dr} \arrow{d}{VE} \arrow{r}{(\omega^{\#})^{*}} & C^{k}_\mathrm{lin}(T^{*}\calG) \arrow{d}{VE}\\
    & C^{k}_\mathrm{lin}(A_{T\calG}) \MySymb[(III)]{dr} \arrow{d}{j_{\calG}^{*}}[swap]{\cong} \arrow{r}{(Lie \omega^{\#})^{*}} & C^{k}_\mathrm{lin}(A_{T^{*}\calG}) \arrow{d}{(\theta^{-1}_{\calG})^{*}}\\
		C^{k}(A) \arrow{r}{T} & C^{k}_{T\pi}(TA)_\mathrm{lin} \arrow{r}{(L(\omega)^{\#})^{*}} & C^{k}_{A^{*}}(T^{*}A)_\mathrm{lin} %\arrow{r}{R^{*}}[swap]{\cong} & \cdots  & C^{k}_{A^{*}}(T^{*}A^{*})_\mathrm{lin}=\mathcal{X}^{k}_\mathrm{lin}(A^{*}) \arrow{r}{D} & Der^{k-1}(A)
  \end{tikzcd}
\end{equation}

The commutativity of diagram $(I)$ (which we prove in subsection \ref{(I)} below) is a general fact relating the tangent lifts of cochains on Lie groupoids and Lie algebroids by the van Est map, and it does not involve the symplectic structure. The commutativity of diagram $(II)$ follows from the naturality of the van Est map with respect to morphisms of Lie groupoids (Lemma 2.10 \cite{CD}). And finally, the commutativity of diagram $(III)$ amounts to regarding the morphism of Lie algebroids $\mathrm{Lie}(\omega^{b}):A_{T\calG}\rightarrow A_{T^{*}\calG}$ in terms of the (canonically) isomorphic Lie algebroids $TA_{\calG}$ and $T^{*}A_{\calG}$. We denote such a morphism by $L(\omega)^{b}:TA_{\calG}\longrightarrow T^{*}A_{\calG}$. Moreover, since this latter morphism is the infinitesimal counterpart of a morphism of Lie groupoids induced by a multiplicative 2-form on $\calG$, then it also turns out to be induced by a 2-form $L(\omega)$ on $A$ (that fact explains our notation $L(\omega)^{b}$ for the morphism). In fact, one can explicitly describe such a 2-form $L(\omega)$ by using the tangent lift $\omega^{T}\in\Omega^{2}(T\calG)$ of the form $\omega\in\Omega^{2}(\calG)$. One obtains $L(\omega)=\i_A^{*}\omega^{T}$, where $\i_A$ is the inclusion map $A_\calG\hookrightarrow T\calG$. This last point is developed in (\cite{BCO}, Prop. 3.7).\\
Composing the reversal isomorphism $T^{*}A^{*}\stackrel{R_A}{\longrightarrow}T^{*}A$ (Section \ref{Reversal isomorphism}) and the isomorphism $D_{A^{*}}$ (Remark \ref{D_{A^{*}}}) with the lower part of diagram (\ref{d1}), we obtain
\begin{equation}\label{d2}
\begin{tikzcd}[column sep=1.5em, row sep=10ex]
H^{k}(\calG) \arrow{rr}{j} \arrow{d}{VE} & & H^{k}_\mathrm{lin}(T^{*}\calG) \arrow{rr}{\cong} \arrow{d}{VE_\mathrm{lin}} & & H^{k}_\mathrm{def}(\calG) \arrow{d}{VE_\mathrm{def}}\\
H^{k}(A) \arrow{rr}{(L(\omega)^{\#})^{*}\circ T} & & H^{k}_{A^{*}}(T^{*}A)_\mathrm{lin} \arrow{r}{R^{*}}[swap]{\cong} & H^{k}_{A^{*}}(T^{*}A^{*})_\mathrm{lin} \arrow{r}{D_{A^{*}}}[swap]{\cong} & H^{k}_\mathrm{def}(A).
\end{tikzcd}
\end{equation}

On the other hand, adding the closedness property of the 2-form $\omega$,  we get both a Poisson structure $\pi$ on $M$  and, at the infinitesimal level, an isomorphism of Lie algebroids $\sigma:A_{\calG}\rightarrow (T^{*}M)_{\pi}$. Thus, having in mind the expression (\ref{i}) for the map $i$, we have the following commutative diagram.
\begin{flushleft}

\begin{small}
\begin{equation}\label{d3}
\begin{tikzcd}[column sep=2.8em, row sep=10ex]
C^{k}(A)\arrow{r}{T} \arrow{d}{(\sigma^{-1})^{*}} & C^{k}_{T\pi}(TA)_\mathrm{lin} \arrow{r}{(\Theta_{A^{*}}^{-1})^{*}} \arrow{d}{(T\sigma^{-1})^{*}} & C^{k}_{A^{*}}(T^{*}A^{*})_\mathrm{lin} \arrow{r}{D_{A^{*}}} & Der^{k-1}(A) \arrow{d}{(\sigma^{-1})_{\#}}\\
C^{k}_{c_{M}}(T^{*}M) \arrow{r}{T} & C^{k}_{Tc_{M}}(T(T^{*}M))_\mathrm{lin} \arrow{r}{(\Theta_{TM}^{-1})^{*}} & C^{k}_{c_{TM}}(T^{*}(TM))_\mathrm{lin} \arrow{r}{D_{TM}} \arrow{u}{((T\omega)^{*})^{\#}} & Der^{k-1}(T^{*}M),
\end{tikzcd}
\end{equation}
\end{small}

\end{flushleft}
where $\Theta_{A^{*}}$ is the isomorphism of VB-algebroids induced from $\Theta_{TM}$ by using the isomorphism $\sigma:A\longrightarrow T^{*}M$. That is, \[\Theta_{A^{*}}=\left((T\sigma^{*})^{\star}_{(\sigma^{*})^{-1}}\right)^{-1}\circ\Theta_{TM}\circ T\sigma:TA\longrightarrow T^{*}A^{*}\] (see Remark \ref{dualstar}) and for simplicity we write simply $\Theta_{A^{*}}=\left((T\sigma^{*})^{*}\right)^{-1}\circ\Theta_{TM}\circ T\sigma$.

Hence, comparing the upper part of diagram (\ref{d3}) and the lower part of diagram \eqref{d2}, we obtain two VB-algebroid isomorphisms from $T^{*}A^{*}$ to $TA$: $L(\omega)^{\#}\circ R_{A}$ and $\Theta_{A^{*}}^{-1}$. We will prove now that these two maps are the same, which will complete the proof of the theorem (just put together diagrams (\ref{d2}) and the diagram induced in cohomology by (\ref{d3})).
In fact,
\begin{align*}
\Theta_{A^{*}}&=((T\sigma^{*})^{*})^{-1}\circ\Theta_{TM}\circ(T\sigma)\\
&=((T\sigma^{*})^{*})^{-1}\circ(R_{TM}\circ\omega^{b}_\mathrm{can})\circ T\sigma\ \ \ (\text{Prop. }\ref{Tulczyjew})\\
&=R_{A^{*}}\circ(T\sigma)^{*}\circ\omega^{b}_\mathrm{can}\circ T\sigma\ \ \ (\text{Prop. }\ref{R_A, R_B})\\
&=R_{A}^{-1}\circ(\sigma^{*}\omega_\mathrm{can})^{b}\\
&=R_{A}^{-1}\circ L(\omega)^{b};
\end{align*}
where the fact that $R_{A^{*}}=R_A^{-1}$ is easily verified from a local point of view %(see e.g.\ \cite{BCdH}, Rmk 2.2.2),
and the last step follows from the characterization of the \textit{linear} 2-forms on $A_\calG$ coming from multiplicative 2-forms on $\calG$ (\cite{BCO}, Prop. 4.6).
\end{proof2}

\begin{remark}(IM-2-forms) It is worth to mention that the map $\sigma$ of the previous proof is the IM-2-form associated to the multiplicative 2-form $\omega$ (cf.\ \cite{BC,BCO,BCWZ,CSS}), and it makes sense even when $\omega$ is only multiplicative but not closed.
\end{remark}

\subsection{Proof of the commutativity of diagram $(I)$}\label{(I)}

Here we complete the proof of Theorem \ref{global of i} by proving the commutativity of diagram $(I)$ in (\ref{d1}) above.  
We will proceed by using the properties of the tangent lift detailed in Lemma \ref{Tang. lift on algbrds}, and by working on core sections and tangent lift of sections, which together span all sections of the tangent algebroid (example \ref{Tangent}).

Recall the definition of the van Est map $VE:C^{k}(\calG)\longrightarrow C^{k}(A_\calG)$, between the differentiable cohomology of a Lie groupoid and the algebroid cohomology of its associated Lie algebroid. Given $c\in C^{k}(\calG)$, $VE(c)$ is defined (using the conventions from \cite{CD, C}) by
$$VE(c)(X_1,...,X_k)=\sum_{\sigma\in S_{n}}sgn(\sigma)R_{\sigma(X_1)}\circ\cdots\circ R_{\sigma(X_k)}c;\ \text{for } X_i\in\Gamma(A_\calG),$$
where if $X\in\Gamma(A_\calG)$, the map $R_{X}:C^{k}(\calG)\longrightarrow C^{k-1}(\calG)$ is given by
$$R_{X}c(g_1,...,g_{k-1})=\left.\frac{d}{d\e}\right|_{\e=0}c(\psi^{\vec{X}}_\e(t(g_1)),g_1,...,g_{k-1}),$$
$\psi^{\vec{X}}_\e$ being the flow of the right invariant vector field associated to $X\in\Gamma(A_\calG)$.\\

We prove explicitly the commutativity of $(I)$ for $k=2$, commenting along the way how to extend the proof to the general case. Note that, for $c\in C^{2}(\calG)$, by definition of $R_{X}$ we have
\begin{equation}\label{vE}
(R_{X_1}R_{X_2}c)_{(x)}=\left.\frac{d}{d\e_1}\right|_{\e_1=0}\left.\frac{d}{d\e_2}\right|_{\e_2=0}c\left(\psi_{\e_2}^{\vec{X}_2}(\phi_{\e_1}^{V_1}(x)),\psi_{\e_1}^{\vec{X}_1}(x)\right),
\end{equation}
where $V_1=\rho(X_1)\in\mathfrak{X}(M)$ is the projection of $X_1$ by the anchor.

In order to use formula (\ref{vE}) for sections of the tangent Lie algebroid, we study the flow of the appropriate right-invariant vector fields on $T\calG$. We then split the proof in three cases:\\

\paragraph{\textit{Linear sections:}} ($X_i=j_{\calG}\circ(T\alpha_i)$, $\alpha_i\in\Gamma(A_\calG)$, $i=1,2.$)\\

In this case, the flow of $\vec{X}_i\in\mathfrak{X}(T\calG)$ is the tangent lift of the flow of $\vec{\alpha}_i$ (\cite{MX2}, Thm 7.1), thus for instance,
\begin{equation*}
\begin{split}
\psi_\e^{\vec{X}_\i}(Tu(w_y))&=(T\psi_\e^{\vec{\alpha}_i})(Tu(w_y))\\
&=T(\psi_\e^{\vec{\alpha}_i}\circ u)(w_y)\\
&=\left.\frac{d}{d\l}\right|_{\l=0}\psi_\e^{\vec{\alpha}_i}\circ u(\gamma(\l)),
\end{split}
\end{equation*}
for $\gamma(\l)$ a curve determining $w_y$.
Then,
\begin{equation*}
\begin{split}
(R_{X_1}&R_{X_2}(Tc))_{(w_x)}=\\
&=\left.\frac{d}{d\e_1}\right|_{\e_1=0}\left.\frac{d}{d\e_2}\right|_{\e_2=0}dc\left(\left.\frac{d}{d\l}\right|_{\l=0}(\psi_{\e_2}^{\vec{\alpha}_2}(\psi_{\e_1}^{V_1}(\gamma(\l)))), \psi_{\e_1}^{\vec{\alpha}_1}((\gamma(\l)))\right)\\
&=\left.\frac{d}{d\l}\right|_{\l=0}\left(R_{\alpha_1}R_{\alpha_2}c\right)_{(\gamma(\l))}\\
&=d(R_{\alpha_1}R_{\alpha_2}c)_{(w_x)}.
\end{split}
\end{equation*}

Therefore, it follows that $\left.VE(Tc)(X_1,X_2)\right|_{w_x}=\left.d\left(VE(c)(\alpha_1,\alpha_2)\right)\right|_{w_x},$ which, by item (1) in Lemma \ref{Tang. lift on algbrds}, is the same as $$\left.j_{\calG}^{*}(VE(Tc))(T\alpha_1,T\alpha_2)\right|_{w_x}=\left.T\left(VE(c)\right)(T\alpha_1,T\alpha_2)\right|_{w_x}.$$
The same argument works for any $k$-cochain ($k>0$) if we take only linear sections.\\

\paragraph{\textit{One core section case:}} Let $\hat{\alpha}\in\Gamma(TA_\calG)$ be the core section associated to $\alpha\in\Gamma(A_\calG)$, and $X:=j_{\calG}\circ\hat{\alpha}$.\\

In this case, the right-invariant vector field on $T\calG$ associated to $X$ is $\vec{X}=(\vec{\alpha})^{\uparrow}$ (\cite{MX2}, Thm 7.1), therefore its flow is given by $\psi_\e^{(\vec{\alpha})^{\uparrow}}(v_g)=v_g+\e\vec{\alpha}_g$.

Thus, if $X_1:=j_{\calG}\circ T\alpha_1$ and $X_2:=j_{\calG}\circ\hat{\alpha}_2$, denote by $u_1=\rho(\alpha_1)\in\mathfrak{X}(M)$:
\begin{equation*}
\begin{split}
\left(R_{X_1}R_{X_2}(Tc)\right)_{(w_x)}&=\left.\frac{d}{d\e_1}\right|_{\e_1=0}\left.\frac{d}{d\e_2}\right|_{\e_2=0}Tc\left(\psi_{\e_2}^{(\vec{X}_2)}(t_{T\calG}(\psi_{\e_1}^{\vec{X}_1}(w_x))), \psi_{\e_1}^{\vec{X}_1}(w_x)\right)\\
&=\left.\frac{d}{d\e_1}\right|_{\e_1=0}\left.\frac{d}{d\e_2}\right|_{\e_2=0}Tc(\underbrace{T\psi_{\e_1}^{u_1}(w_x)}_{w'_{\psi_{\e_1}^{u_1}(x)}}+\e_2\left.\vec{\alpha}_2\right|_{\psi_{\e_1}^{u_1}(x)},T\psi_{\e_1}^{\vec{\alpha}_1}(w_x))\\
&=\left.\frac{d}{d\e_1}\right|_{\e_1=0}\underbrace{\left.\frac{d}{d\e_2}\right|_{\e_2=0}Tc(w'_{\varphi_{\e_1}^{u_1}(x)},T\psi_{\e_1}^{\vec{\alpha_1}}(w_x))}_{0}\\
&+\left.\frac{d}{d\e_1}\right|_{\e_1=0}\left.\frac{d}{d\e_2}\right|_{\e_2=0}dc\left(\e_2\left.\vec{\alpha}_2\right|_{\varphi_{\e_1}^{u_1}(x)},0_{\psi_{\e_1}^{\vec{\alpha}_1}(x)}\right)\\
&=\left.\frac{d}{d\e_1}\right|_{\e_1=0}\left.\frac{d}{d\l}\right|_{\l=0}c\left(\psi_{\l}^{\vec{\alpha}_2}(\varphi_{\e_1}^{u_1}(x)),\psi_{\e_1}^{\vec{\alpha}_1}(x)\right)\\
&=\left(R_{\alpha_1}R_{\alpha_2}c\right)_{(x)},
\end{split}
\end{equation*}
where in the third equality we use the linearity of $Tc$ over $\calG^{(2)}$ ($Tc\in C^{2}_\mathrm{lin}(T\calG)$).

Similarly, by changing the order,
\begin{equation*}
\begin{split}
(R_{X_2}&R_{X_1}(Tc))_{(w_x)}=\\ &=\left.\frac{d}{d\e_2}\right|_{\e_2=0}\left.\frac{d}{d\e_1}\right|_{\e_1=0}Tc\left((T\psi_{\e_1}^{\vec{\alpha_1}})\left(\underbrace{\varphi_{\e_2}^{u^{\uparrow}_2}(w_x)}_{w_x+\e_2(u_2)_{x}}\right), \underbrace{\psi_{\e_2}^{(\vec{\alpha}_2)^{\uparrow}}(w_x)}_{w_x+\e_2(\vec{\alpha}_2)_{x}}\right)\\
&=\left.\frac{d}{d\e_2}\right|_{\e_2=0}\left.\frac{d}{d\e_1}\right|_{\e_1=0}dc\left((T\psi_{\e_1}^{\vec{\alpha_1}})(\e_2(u_2)_{x}),\e_2(\vec{\alpha}_2)_{x}\right)\\
&+\left.\frac{d}{d\e_2}\right|_{\e_2=0}\left.\frac{d}{d\e_1}\right|_{\e_1=0}dc(T\psi_{\e_1}^{\vec{\alpha_1}}(w_x),w_x)\\
&=\left.\frac{d}{d\e_2}\right|_{\e_2=0}\left.\frac{d}{d\e_1}\right|_{\e_1=0}\e_2\;dc\left((T\psi_{\e_1}^{\vec{\alpha_1}})(\left.\frac{d}{d\l}\right|_{\l=0}\varphi^{u_2}_\l(x)), \left.\frac{d}{d\l}\right|_{\l=0}\psi^{\vec{\alpha}_2}_\l(x)\right)\\
&=\left.\frac{d}{d\e_2}\right|_{\e_2=0}\left.\frac{d}{d\e_1}\right|_{\e_1=0}\e_2\!\left.\frac{d}{d\l}\right|_{\l=0}c\left(\psi_{\e_1}^{\vec{\alpha_1}}(\varphi^{u_2}_\l(x)), \psi^{\vec{\alpha}_2}_\l(x)\right)\\
&=\left.\frac{d}{d\e_2}\right|_{\e_2=0}\left.\frac{d}{d\e_1}\right|_{\e_1=0}c\left(\psi_{\e_1}^{\vec{\alpha_1}}(\varphi^{u_2}_{\e_2}(x)), \psi^{\vec{\alpha}_2}_{\e_2}(x)\right)\\
&=\left(R_{\alpha_2}R_{\alpha_1}c\right)(x),
\end{split}
\end{equation*}
where in the second equality we view $dc$ as an element of $C^{2}_\mathrm{lin}(T\calG)$. Hence, it follows that $VE(Tc)(X_1,X_2)_{(w_x)}=\left(VE(c)_{(\alpha_1,\alpha_2)}\right)_{(x)},$ which, by Lemma \ref{Tang. lift on algbrds}, is \[\left.j_{\calG}^{*}(VE(Tc))(T\alpha_1,\hat{\alpha}_2)\right|_{(w_x)}=\left.T(VE(c))(T\alpha_1,\hat{\alpha}_2)\right|_{(w_x)},\]
as we want.

The main ideas used in the proof of this case are the following. $(i)$ in the expression $\psi_{\e}^{(\vec{\alpha_2})^{\uparrow}}(v_x)=v_x+\e\left.\vec{\alpha}_2\right|_x$ for the flow of the right-invariant vector field on $T\calG$ associated to the core section $\hat{\alpha}_2$, the parameter $\e$ only appears multiplied by $\left.\vec{\alpha}_2\right|_x$; $(ii)$ by the linearity of $Tc\in C^{2}_\mathrm{lin}(T\calG)$ and $(i)$, we can form two vectors of $(T\calG)^{(2)}$: one multiplied by $\e_2$, the other one independent of $\e_2$. Using these facts, a completely analogous argument works for $k$-cochains. In fact, by skew-symmetry of the elements of $C^{k}(A_{T\calG})$, we can assume that $X_k\in\Gamma(A_{T\calG})$ is the core section in the expression $VE(Tc)(X_1,...,X_k)$, so that the proof of the equality is just like that above.\\

\paragraph{\textit{More than one core section case:}} ($X_i:=j_{\calG}\circ\hat{\alpha}_i$, $\alpha_i\in\Gamma(A)$, $i=1,2.$)

\begin{equation*}
\begin{split}
(R_{X_1}&R_{X_2}(Tc))_{(w_x)}=\\ &=\left.\frac{d}{d\e_1}\right|_{\e_1=0}\left.\frac{d}{d\e_2}\right|_{\e_2=0}Tc\left(\psi_{\e_2}^{(\vec{X}_2)}(t_{T\calG}\psi_{\e_1}^{\vec{X}_1}(w_x)), \psi_{\e_1}^{(\vec{X}_1)}(w_x)\right)\\
&=\left.\frac{d}{d\e_1}\right|_{\e_1=0}\left.\frac{d}{d\e_2}\right|_{\e_2=0}Tc\left(t_{T\calG}\psi_{\e_1}^{(\vec{\alpha}_1)^{\uparrow}}(w_x)+\e_2(\vec{\alpha}_2)_{x},w_x+\e_1(\vec{\alpha}_1)_{x}\right)\\
&=\left.\frac{d}{d\e_1}\right|_{\e_1=0}\left.\frac{d}{d\e_2}\right|_{\e_2=0}Tc\left(w_x+\e_1(\rho(\alpha_1))_{x}+\e_2(\vec{\alpha}_2)_{x}, w_x+\e_1(\vec{\alpha}_1)_{x}\right)\\
&=\left.\frac{d}{d\e_1}\right|_{\e_1=0}\left.\frac{d}{d\e_2}\right|_{\e_2=0}Tc\left(w_x+\e_1(\rho(\alpha_1))_{x}, w_x+\e_1(\vec{\alpha}_1)_{x}\right)\\
&+\left.\frac{d}{d\e_1}\right|_{\e_1=0}\left.\frac{d}{d\e_2}\right|_{\e_2=0}Tc\left(\e_2(\vec{\alpha}_2)_{x},0_x\right)\\
&=0.
\end{split}
\end{equation*}

That is, $j_{\calG}^{*}\left[VE(Tc)\right](\hat{\alpha}_1,\hat{\alpha}_2)=0$. Then item (3) in Lemma \ref{Tang. lift on algbrds} completes the commutativity in this case.
Analogously it is shown that the equality holds for $k$-cochains if we have more than one core section: it suffices to decompose the vector in $(T\calG)^{(k)}$ as a sum of two vectors, one depending only on $\e_1$ and another depending only on $\e_2$. 

This finishes the proof that the van Est maps intertwine tangent lifts of cochains (of Lie groupoids and Lie algebroids), i.e., of the commutativity of the diagram $(I)$.

\begin{remark}
Recall that any strict deformation of Lie groupoids induces deformations of Lie algebroids, fibrewise linear Poisson manifolds, and symplectic (VB-)groupoids (Example \ref{Example cotangent-gpd-deformation}).
\[\G_\e \rightsquigarrow A_\e \rightsquigarrow (A_\e^*, \pi_\e)\rightsquigarrow (T^*\calG_\e, \omega_\mathrm{can})\]

There are now plenty of relations between these objects at the level of the corresponding deformation cohomologies. There is a van Est map $H^\bullet_\mathrm{def}(\G)\to H^\bullet_\mathrm{def}(A)$ \cite[Thm. 10.1]{CMS}; there is an isomorphism $H^\bullet_\mathrm{def}(A)\cong H^\bullet_{\pi,\mathrm{lin}}(A^*)$ \cite[Prop. 8]{CM}; the map $i:H^\bullet_{\pi}(A^*)\longrightarrow H^\bullet_\mathrm{def}\left(T^{*}A^*\right)$ relates infinitesimal deformations of Poisson structures on $A^*$ and of Lie algebroids on $T^*A^*$. Proposition \ref{Prop-cotang-grpd} gives a relation between $H^\bullet_{\pi}(A^*)$ and $H^\bullet_\mathrm{def}(T^*\G,\omega_\mathrm{can})$. There is also an inclusion induced by the tangent lift of \cite{PierVitag}, $H^\bullet_\mathrm{def}(\G)\xhookrightarrow{} H^\bullet_\mathrm{def, lin}(T^*\G)\subset H^\bullet_\mathrm{def}(T^*\G)$, a linear version of $i_{T^*\G}$. 
\end{remark}
\newpage
\appendix

\section{Double structures}

\subsection{VB-algebroids}

The infinitesimal counterpart of VB-groupoids are objects called VB-algebroids \cite{GrMe, Mdoublealgbrds}, which are vector bundles objects in the category of Lie algebroids. 
Similarly to the situation with VB-groupoids, VB-algebroids provide alternative viewpoints on the deformation cohomology (Rmk \ref{D_{A^{*}}}) and on the representation theory of Lie algebroids \cite{GrMe}.

\begin{definition}
A \textbf{double vector bundle} $(D,E,A,M)$, or just $D$, is a diagram

\begin{equation}\label{DVB}
%\[
  \begin{tikzcd}[column sep=4em, row sep=10ex]
    D \MySymb[\circlearrowright]{dr} \arrow{d}[swap]{q^{D}_{A}} \arrow{r}{q^D_E} & E \arrow{d}{q_E}\\
    A \arrow{r}{q_A} & M,
  \end{tikzcd}
%\]
\end{equation}
such that the rows and columns are vector bundles, and $q^D_E$ and the addition map of $D\stackrel{q^D_E}{\longrightarrow}E$, $+_E:D\oplus_{E}D\longrightarrow D$, are vector bundle morphisms over $q_A$ and the addition of $A$, $+_:A\oplus_{M}A\longrightarrow A$, respectively.
\end{definition}

In the \textbf{vertical bundle structure} on $D$ with base $A$, $\tilde{D}_A$, we use the notation $\tilde{0}^A:A\longrightarrow D,\ a\longmapsto\tilde{0}^A_a$ for the zero-section; similarly, in the \textbf{horizontal vector bundle structure} on $D$ over $E$, $\tilde{D}_E$, we write $\tilde{0}^E:E\longrightarrow D,\ b\longmapsto\tilde{0}^E_b$. The vector bundles $A$ and $E$ over $M$ are called the \textbf{side bundles} of (\ref{DVB}); we denote their zero-sections by $0^A$ and $0^E$, respectively.

\begin{definition}%(Morphisms of DVB)
A \textbf{morphism of double vector bundles} \[(\phi,\phi_E,\phi_A,\phi_M):(D, E, A, M)\longrightarrow (D', E', A', M')\] consists of maps $\phi:D\longrightarrow D'$, $\phi_E:E\longrightarrow E'$, $\phi_A:A\longrightarrow A'$, $\phi_M:M\longrightarrow M'$ such that each of $(\phi,\phi_E)$, $(\phi,\phi_A)$, $(\phi_E,\phi_M)$ and $(\phi_A,\phi_M)$ are morphisms of the corresponding vector bundles.\\
If $M=M'$, $E=E'$ and $\phi_E=id_E$, one says that $\phi$ \textit{preserves} $E$. If, further, $A=A'$ and $\phi_A=id_A$ one says that $\phi$ \textit{preserves the side bundles}.
\end{definition}

\paragraph{\textbf{The core bundle and core and linear sections}}A third vector bundle over $M$ associated to $(D,E,A,M)$ is the \textbf{core bundle} $C$, defined as the intersection of the kernels of $q^D_E$ and $q^D_A$. To avoid confusion, when
regarding $c\in C$ as belonging to $D$, we will denote it by $\bar{c}$. The core fits into an exact sequence of vector bundles over $A$.
\begin{equation}\label{Coreseq.}
0\rightarrow q^{*}_{A}C\stackrel{\tau_A}{\longrightarrow}\tilde{D}_A\stackrel{(q^{D}_E)^{!}}{\longrightarrow}q_A^{*}E\rightarrow 0,
\end{equation}

where $(q^D_E)^{!}$ is the induced projection on $q_A^{*}E$ and $\tau_A(a,c)=\tilde{0}^A_a+_E\bar{c}$ (which makes sense because (i) $\tilde{0}^A_a$ and $\bar{c}\in D$ are in the same $q^D_E$-fibre over $0^E_{q_A(a)}$ and (ii) $q^D_A$ is a morphism, thus $\tau_A(a,c)\in\tilde{D}_A$ is over $a\in A$). This sequence is called the \textbf{core sequence of} $D$ \textbf{over} $A$. Analogously, there is a core sequence of $D$ over $E$.

An important aspect of the core sequences is that, for instance in (\ref{Coreseq.}), a section of $C$ induces a section in $\Gamma_A(D)$ of $\tilde{D}_A$. In fact, $c\in\Gamma(C)$ defines $c^{A}\in\Gamma_A(D)$ by
$$c^{A}(a):=\tau(a,c_{q_A(a)})=\tilde{0}^A_a+_E\bar{c}_{q_A(a)}.$$
The section $c^{A}$ is called the \textbf{core section over} $A$ \textbf{corresponding to} $c$. The space of core sections over $A$ is denoted by $\Gamma_\textrm{core}(D,A)$. Analogously, with the core sequence over $E$, one gets core sections over $E$.

Another special type of sections of $\tilde{D}_A$, called the \textbf{linear sections} of $D$ over $A$, are those which are vector bundle morphisms from $A\longrightarrow M$ to $D\longrightarrow E$. The space of linear sections is denoted by $\Gamma_\mathrm{lin}(D,A)$. 

An important fact about linear and core sections is that together they span all sections of $D$ over $A$ \cite{Mdoubles}. 

\begin{definition}
 A VB-algebroid is a DVB as in (\ref{DVB}), where $D\longrightarrow E$ is a Lie algebroid with anchor map $\rho_D:D\longrightarrow TE$ being a vector bundle morphism over $A\longrightarrow TM$ and such that the Lie bracket $[\cdot,\cdot]_D$ satisfies the following conditions:
 \begin{enumerate}
  \item $[\Gamma_\mathrm{lin}(D,E),\Gamma_\mathrm{lin}(D,E)]_D\subset\Gamma_\mathrm{lin}(D,E)$,
  \item $[\Gamma_\mathrm{lin}(D,E),\Gamma_\mathrm{core}(D,E)]_D\subset\Gamma_\mathrm{core}(D,E)$,
  \item $[\Gamma_\mathrm{core}(D,E),\Gamma_\mathrm{core}(D,E)]_D=0$.
 \end{enumerate}
\end{definition}

As pointed out in \cite{GrMe}, a VB-algebroid $D\longrightarrow E$ induces a Lie algebroid structure on $A\longrightarrow M$ and the structure maps (projection, zero section, sum) of the vertical bundle structures form Lie algebroid morphisms. In this sense, a VB-algebroid can be though as a vector bundle in the category of Lie algebroids, (\cite{GrMe}, Thm. 3.7).

\begin{example}\label{Tangent}(Tangent prolongation DVB; Tangent Lie algebroid)
\begin{enumerate}
	\item Applying the tangent functor to the structure maps (projection, addition, zero-section) of the vector bundle $E\stackrel{q}{\longrightarrow} M$  yields the DVB
\begin{equation*}\label{Tangent DVB}
%\[
  \begin{tikzcd}[column sep=4em, row sep=10ex]
    TE \MySymb[\circlearrowright]{dr} \arrow{d}[swap]{p_E} \arrow{r}{Tq} & TM \arrow{d}{p_M}\\
    E \arrow{r}{q} & M,
  \end{tikzcd}
%\]
\end{equation*}
 with core $E$. The core sections over $E$ and $TM$ corresponding to $\alpha\in\Gamma(E)$ are respectively:
$$\alpha^{\uparrow}\in\Gamma_E(TE),\ \ \alpha^{\uparrow}(e)=\tilde{0}^E_e+_{TM}\overline{\alpha(p_E(e))}=\left.\frac{d}{d\e}\right|_{\e=0}(e+\e\alpha(p_E(e))),$$
called the vertical lift of $\alpha$, and
$$\hat{\alpha}\in\Gamma_{TM}(TE),\ \hat{\alpha}(v_x)=\tilde{0}^{TM}_{v_x}+_E\overline{\alpha(x)}\in T_{0^{E}_x}E.$$
Note that $\tilde{0}^{TM}=T(0^{E})$, and $\overline{\alpha(x)}=\alpha^{\uparrow}(0^{E}_x)$.

For any $X\in\Gamma(E)$, its tangent prolongation $T(X):=d(X)\in\Gamma(TE,TM)$ is a linear section of $TE$ over $TM$. Linear sections of $TE$ over $E$ are called \textbf{linear vector fields} on $E$.
\item Let $F:E\longrightarrow E'$ be a vector bundle morphism over $f:M\longrightarrow M'$. The tangent prolongation of $F$ induces the morphism of vector bundles $(TF,Tf,F,f):TE\longrightarrow TE'$.
\end{enumerate}

Given a Lie algebroid $A\stackrel{\pi}{\longrightarrow} M$, there is a Lie algebroid structure on $TA\stackrel{T\pi}{\longrightarrow}TM$ making it into a VB-algebroid. Denote by $\hat{X}_i\in\Gamma_c(TA,TM),\ i=1,2,$ the core section corresponding to $X_i\in\Gamma(A)$. It suffices to define the Lie  bracket on core sections and linear sections of the form $T(X)$, $X\in\Gamma(A)$:
\begin{equation}\label{Bracket}
[TX_1,TX_2]=T[X_1,X_2],\ [TX_1,\hat{X}_2]=\hat{[X_1,X_2]},\ [\hat{X}_1,\hat{X}_2]=0.
\end{equation}

The anchor $\rho_T$ is defined by $\rho_T=J^{-1}\circ T(\rho)$ where
\[\begin{tikzcd}[column sep=2em, row sep=10ex]
   T(TM) \arrow{rr}{J} \arrow{dr}[swap]{p_{TM}} & & T(TM)\arrow{dl}{Tp_{M}}\\
    & TM & \\
\end{tikzcd}\]
is the canonical \textit{involution} of the double tangent bundle $T(TM)$, called the canonical flip, determined locally by $J(x_i,\dot{x}_i,\delta x_i,\delta\dot{x}_i)=(x_i,\delta x_i,\dot{x}_i,\delta\dot{x}_i),$ where for local coordinates $(x^i)$ of $M$, $(\dot{x}^{i})$ are the coordinates on the fibres of $TM$ and $(\delta x^{i}, \delta\dot{x}^{i})$ are the coordinates on the fibres of $T(TM)\stackrel{p_{TM}}{\longrightarrow}TM$.
\end{example}

\begin{example}\label{Dual}(Dual DVB)
By dualizing the core exact sequence over $A$ %(\ref{Coreseq.})
, we can induce a double vector bundle
\begin{equation}\label{Dual DVB}
%\[
  \begin{tikzcd}[column sep=4em, row sep=10ex]
    D^{*}_A \MySymb[\circlearrowright]{dr} \arrow{d}[swap]{q^{*A}_{A}} \arrow{r}{q^{*A}_{C^{*}}} & C^{*} \arrow{d}{q_{C^{*}}}\\
    A \arrow{r}{q_A} & M,
  \end{tikzcd}
%\]
\end{equation}
with core $E^{*}$, where $D^{*}_A$ denotes the dual over $A$, and $q^{*A}_A:D^{*}_A\longrightarrow C^{*}$ comes from the dual of $\tau_A$:
$$\left\langle q^{*A}_{C^{*}}(\eta_a),c\right\rangle=\left\langle \eta_a,\tau_A(a,c)\right\rangle,$$
for $\eta_a:(q^D_A)^{-1}(a)\stackrel{\text{linear}}{\longrightarrow}\mathbb{R}$ and $c\in C_{q_A(a)}$. The addition $+_{C^{*}}:D^{*}_A\oplus_{C^{*}}D^{*}_A\longrightarrow D^{*}_A$ is defined in such a way that the natural pairing $\left\langle ,\right\rangle:D^{*}_A\oplus\tilde{D}_A\longrightarrow\mathbb{R}$ is linear with respect to the vector bundle structure over $C^{*}\oplus_M E$, i.e.,
$$\left\langle \eta_a+_{C^{*}}\eta'_{a'}, d_a+_{E}d'_{a'}\right\rangle=\left\langle \eta_a,d_a\right\rangle+\left\langle \eta'_{a'},d'_{a'}\right\rangle.$$

Note that $\eta_a+_{C^{*}}\eta'_a$ is determined by the expression above due to the fact that any element in $(\tilde{D}_A)_{a+a'}$ can be written as the sum of elements $d\in(\tilde{D}_A)_{a}$ and $d'\in(\tilde{D}_A)_{a'}$. It is not hard to see that $+_{C^{*}}$ given in this form is well-defined. The zero above $\kappa\in C^{*}_x$, denoted by $\tilde{0}^{*A}_\kappa:(\tilde{D}_A)_{0^{A}_x}=\left[\mathrm{Ker}(q^D_A)^{!}\right]_{0^{A}_x}\stackrel{\textit{linear}}{\longrightarrow}\mathbb{R}$, is defined by
$$\left\langle \tilde{0}^{*A}_\kappa,\tilde{0}^{E}_e+_A \bar{c}\right\rangle=\left\langle \kappa,c\right\rangle,\text{ for } e\in E_x,\ c\in C_x.$$
Analogously, one can take the dual of the core exact sequence over $E$, inducing a DVB $(D^{*}_E,E,C^{*},M)$ with core $A^{*}$. See \cite{M} or \cite{Mduality-triple} for further details.
\end{example}

\begin{remark}
Similarly to the dual of a morphism of vector bundles over the same base covering the identity, one defines the dual of a morphism of DVBs which have one same side bundle. If $(\phi,\phi_E,id_A,\phi_M):(D,E,A,M)\longrightarrow(D',E',A,M)$ is a DVB morphism preserving $A$, dualizing $\phi$ as a morphism of vector bundles over $A$ yields $(\phi_A^{*},\phi_C^{*},id_A,\phi_M):(D'^{*}_{A},(C')^{*},A,M)\longrightarrow(D^{*}_{A},C^{*},A,M)$, a DVB morphism preserving $A$  with core morphism $\phi_E^{*}:(E')^{*}\longrightarrow E^{*}$.
\end{remark}

\paragraph{\textbf{Isomorphisms of duals of DVBs}}\label{Z_A}

As example \ref{Dual} shows, there are two different ways to dualize a DVB: the \textit{vertical} and the \textit{horizontal} dualizations, which are related by mixed iteration. Indeed, the horizontal dual $((D^{*}_A)^{*}_{C^{*}},C^{*},E,M)$ of the dual DVB (\ref{Dual DVB}) is a DVB with the same side and core bundles as $(D^{*}_E,E,C^{*},M)$. One can check that they are isomorphic DVBs. Namely, $Z_E:D^{*}_E\longrightarrow(D^{*}_A)^{*}_{C^{*}}$ is the isomorphism induced by a natural pairing $\mid\cdot,\cdot\mid$ between the vertical and horizontal duals $D^{*}_A$ and $D^{*}_E$ as vector bundles over $C^{*}$. The pairing is defined by
\begin{equation}\label{eq:DVBpairing}
\mid\eta_a,\theta_e\mid=\langle\eta_a,d\rangle_A-\langle\theta_e,d\rangle_E;
\end{equation}
where $\eta_a\in D^{*}_A,\ \theta_e\in D^{*}_E$ with $q^{*A}_{C^{*}}(\eta_a)=q^{*E}_{C^{*}}(\theta_e)$ and $d\in D$ is any element such that the canonical pairings in the RHS make sense.

Of course, this pairing also yields the isomorphism $D^{*}_A\stackrel{Z_A}{\cong}(D^{*}_E)^{*}_{C^{*}}$ which induces the identity on the cores $E^{*}$ and on the side bundles $C^{*}$, and is $-id_A$ on the remaining side bundles $A$ (\cite{M}, Corollary 9.2.4). Summing up, taking duals over $C^{*}$ interchanges the vertical and horizontal duals of $(D,E,A,M)$. Or equivalently, mixed iteration of vertical and horizontal duals interchanges the duals of $D$ (e.g.\ horizontal dual followed by vertical one is the flip of the vertical one, where the flip of a DVB $(D,E,A,M)$ is $(D,A,E,M)$).

\subsection{Reversal isomorphism}\label{Reversal isomorphism}
The reversal isomorphism of DVBs \[R_A:T^{*}A^{*}\longrightarrow T^{*}A\] is a Legendre type map which allows us to relate two cotangent spaces. We present here some details of its definition.

Examples \ref{Tangent} and \ref{Dual} show us two important ways to get double vector bundles from usual vector bundles (by tangent prolongation and dualizing). By considering also dualization of vector bundles, there exists a certain compatibility between these processes: dualization commutes with tangent prolongation (up to a canonical isomorphism). Indeed, the tangent lift of the canonical pairing, $\left\langle, \right\rangle_A$, between $A$ and $A^{*}$ over $M$ induces the \textbf{tangent pairing} $\left\langle \left\langle ,\right\rangle\right\rangle_A$ between $TA$ and $T(A^{*})$ over $TM$. For $(v_{a^{*}},w_a)\in TA^{*}\oplus_{TM}TA$, $$\left\langle \left\langle v_{a^{*}}, w_a \right\rangle\right\rangle_A:=\left.\frac{d}{d\e}\right|_{\e=0}\left\langle \gamma(\e),\alpha(\e)\right\rangle_A;$$
where $\gamma(\e)$ and $\alpha(\e)$ are curves on $A^{*}$ and $A$ representing $v_{a^{*}}$ and $w_a$, respectively, with $p_{A^{*}}(\gamma(\e))=p_{A}(\alpha(\e))$. This pairing yields the isomorphism \[I_A:TA^{*}\longrightarrow(TA)^{*}_{TM}\] which is moreover an isomorphism of DVBs: between the tangent prolongation of the dual of $A$ and the (horizontal) dual of the tangent prolongation of $A$. The isomorphism $I_A$ preserves the sides and core bundles, and is often called the \textbf{internalization map} (\cite{M}, Prop. 9.3.2).

\begin{equation}\label{Internalization}
%\[
  \begin{tikzcd}[column sep=4em, row sep=10ex]
    TA^{*} \MySymb[\circlearrowright]{dr} \arrow{d}[swap]{p_{A^{*}}} \arrow{r}{Tq_{A^{*}}} & TM \arrow{d}{p_M}\\
    A^{*} \arrow{r}{q_{A^{*}}} & M
  \end{tikzcd}\stackrel{I_A}{\longrightarrow} \begin{tikzcd}[column sep=4em, row sep=10ex]
    (TA)^{*}_{TM} \MySymb[\circlearrowright]{dr} \arrow{d}[swap]{q^{*TM}_{A^{*}}} \arrow{r}{q^{*TM}_{TM}} & TM \arrow{d}{p_M}\\
    A^{*} \arrow{r}{q_{A^{*}}} & M.
  \end{tikzcd}
%\]
\end{equation}
Consider now the tangent prolongation DVB of $A\longrightarrow M$. The natural pairing \eqref{eq:DVBpairing} existing between its (horizontal and vertical) duals induces the isomorphism
$$Z_A:T^{*}A=(TA)^{*}_{A}\longrightarrow((TA)^{*}_{TM})^{*}_{A^{*}}.$$
Composing such an isomorphism with the dual $(I_A)^{*}_{A^{*}}$ of $I_A$ over $A^{*}$, one obtains the following isomorphism of DVBs
$$(I_A)^{*}_{A^{*}}\circ Z_A:T^{*}A\longrightarrow T^{*}A^{*}$$
which (like $Z_A$) induces $-id_A$ on the side bundles $A$ and preserves the cores $T^{*}M$ and sides $A^{*}$. 
The \textbf{reversal isomorphism} $R_A:T^{*}A^{*}\longrightarrow T^{*}A$ is then defined by $R_A:=((I_A)^{*}_{A^{*}}\circ Z_A)^{-1}\circ(-_{A^{*}}id_{T^{*}A^{*}})$, which will be then an isomorphism of DVBs preserving the sides and inducing $-id_{T^{*}M}$ on the cores.

Alternatively, it is possible to describe this map in a simple way by using local coordinates. 
Let $(x^{i},u_d)$ and $(x^{i},u^{d})$ be (fibred) local coordinates of the vector bundles $A\longrightarrow M$ and $A^{*}\longrightarrow M$. Let then $(x^{i},u_d,\dot{x}^i,\dot{u}_d)$ denote the tangent coordinates on $TA$, where $(\dot{x}^{i})$ are the coordinates of elements in $T_xM$ and $(\dot{u}_d)$ the ones of tangent elements to the fibre $A_x$ of $A\to M$, which identify with elements of $A_x$ and therefore are also called the \textbf{core} coordinates of $TA$. Similarly, the vertical and horizontal duals of $TA$ has a local description. Denote by $(x^{i},u_d,\delta x^i,\delta u_d)$ the cotangent coordinates on $T^{*}A$, and let $(x^{i},\zeta_d,\dot{x}^i,\eta_d)$ be the coordinates on the dual $T^{\bullet}A:=(TA)^{*}_{TM}$, where $\zeta_d$ and $\eta_d$ are the dual components to $u^{d}$ and $\dot{u}^{d}$, respectively. In these coordinates,  the dual map $(I_A)^{*}_{A^{*}}:(T^{\bullet}A)^{*}_{A^{*}}\longrightarrow(TA^{*})^{*}_{A^{*}}$ is given by the flip
$$(x^{i},\alpha^{d},\beta_i,\eta_d)\mapsto(x^{i},\eta_d,\beta_i,\alpha^{d});$$
analogously, the isomorphism $Z_A$ above is determined locally by
$$(x^{i},u^{d},\delta x^i,\delta u^{d})\mapsto(x^{i},-u^{d},\delta x^i,\delta u^{d}).$$ Therefore, since locally the automorphism $-_{A^{*}}id_{T^{*}A^{*}}$ only changes the sign on the third and fourth components, the reversal isomorphism has the form
\begin{equation}\label{Reversal}
R_A(x^i,u^{d},\delta x^{i},\delta u^{d})=(x^i,\delta u^{d},-\delta x^{i},u^{d}).
\end{equation}

The following propositions tell us about some naturality properties that the reversal isomorphism satisfies, and which will be useful for us later on. In order to state them, we first adopt some convenient notation for the dual of morphisms of vector bundles and DVBs.

\begin{remark}\label{dualstar}
Consider  two vector bundles $A\longrightarrow M_1$ and $B\longrightarrow M_2$, and a vector bundle morphism $\Phi:A\longrightarrow B$  covering the diffeomorphism $f:M_1\longrightarrow M_2$. Notice that the pointwise dual of $\Phi$ induces a natural morphism $\Phi^{\star}_{f^{-1}}:B^{*}\longrightarrow A^{*}$ between the duals of $B$ and $A$ covering the diffeomorphism $f^{-1}$. 

We remark that such a dual construction can also be made for morphisms of DVBs which share one of their side bundles. For instance, let $f:A\longrightarrow B$ be an isomorphism, then one has the dual of the DVB morphism $Tf:TA\longrightarrow TB$ (which covers the isomorphism $f$), denoted by $(Tf)^{\star}_{f^{-1}}:T^{*}B\longrightarrow T^{*}A$.
\end{remark}

\begin{proposition}\label{R_A, R_B}

Consider two vector bundles $A$ and $B$  over $M$. If $f:A\longrightarrow B$ is an isomorphism covering the identity, the following diagram commutes.
 \begin{equation}\label{Naturality of R_A}
%\[
  \begin{tikzcd}[column sep=4em, row sep=10ex]
    T^ {*}A^{*} \MySymb[\circlearrowright]{dr} \arrow{d}[swap]{((T(f^{*}))^{\star}_{(f^{*})^{-1}}} \arrow{r}{R_A} & T^{*}A\\
    T^{*}B^{*} \arrow{r}{R_B} & T^{*}B \arrow{u}[swap]{(Tf)^{\star}_{f^{-1}}}
  \end{tikzcd}
%\]
 \end{equation}
\end{proposition}

\begin{proof}
The commutativity of this diagram follows directly from the local description of the reversal isomorphism given in \eqref{Reversal}.
\end{proof}

%*****************************************************
%Coordinate-free description

%***************************************************

\begin{remark}
The reversal isomorphism $R_A$ satisfies some additional properties related to symplectic geometry: to begin with, it is an anti-symplectomorphism with respect to the canonical symplectic structures on the cotangent bundles $T^*A^*$ and $T^*A$ (\cite{M}, Thm 9.5.2)). Another property involving the canonical symplectic structure on $T^*M$ is given in Proposition \ref{Tulczyjew} below. 
\end{remark}

\begin{example}\label{ex:cotangentalgbrd}(Cotangent Lie algebroid)
The reversal isomorphism allows to make sense of the so called \emph{cotangent Lie algebroid}. It is well known that a Lie algebroid $A\longrightarrow M$ induces a Poisson structure on its dual $A^{*}\longrightarrow M$, which is linear with respect to the vector bundle structure over $M$ (\cite{Courant}, Thm. 2.1.4). In this way, the cotangent bundle of $A$, \[T^{*}A\stackrel{R_A^{-1}}{\cong}T^{*}A^{*}\longrightarrow A^{*}\] inherits a Lie algebroid structure, which is also a VB-algebroid over $A\longrightarrow M$. Such a VB-algebroid is also often defined as the \emph{dual VB-algebroid} of $TA$ (see Section 3 in \cite{GrMe}).
\end{example}

For a Lie groupoid $\G$ with Lie algebroid $A_\G$, the tangent and cotangent algebroids of $A_\G$ are (isomorphic to) the Lie algebroids of the tangent and cotangent groupoids (section \ref{VB-grpds}) of $\G$, respectively. In fact, the canonical flip $J_\calG:T(T\calG)\longrightarrow T(T\calG)$ restricts to the isomorphism of Lie algebroids 
$$TA_\calG\stackrel{j_\calG}{\longrightarrow}A_{T\calG}$$
over $TM$ (\cite{M}, Thm 9.7.5). And the dual of the map $J_\calG$ gives rise to the isomorphism between the Lie algebroid of $T^*\calG$ and the cotangent algebroid. Explicitly, the isomorphism
$$\Theta_{T\calG}:=J_\calG^{*}\circ I_{T\calG}:T(T^{*}\calG)\longrightarrow (T(T\calG))^{*}_{TM}\longrightarrow T^{*}(T\calG)$$
given by the composition of the internalization map and the dual of the involution induces the isomorphism (\cite{M}, p. 463)
$$A_{T^{*}\calG}\stackrel{\theta_\calG}{\longrightarrow}T^{*}A_{\calG}.$$
The map $\Theta_{T\calG}$ is often called the \textbf{Tulczyjew map}, and it is closely related to the canonical symplectic structure $\omega_\mathrm{can}$ on the cotangent bundle $T^{*}\calG$.

\begin{proposition}\label{Tulczyjew}(\cite{M}, Thm 9.6.7)
Given a manifold $M$, the Tulczyjew and reversal isomorphisms $\Theta_{TM}$ and $R_{TM}$ are related by the following commutative diagram
\begin{equation*}
%\[
  \begin{tikzcd}[column sep=4em, row sep=10ex]
    T(T^ {*}M) \arrow{dr}[swap]{\omega_\mathrm{can}^{b}} \arrow{r}{\Theta_{TM}} & T^{*}(TM)\\
    & T^{*}(T^{*}M) \arrow{u}[swap]{R_{TM}},
  \end{tikzcd}
%\]
 \end{equation*}
where $\omega_\mathrm{can}\in\Omega^{2}(T^{*}M)$ is the canonical symplectic structure on the cotangent bundle $T^{*}M$.
\end{proposition}

\begin{remark}
In the case of a Poisson manifold $M$ , the Tulczyjew map of $M$ $\Theta_{TM}:T(T^{*}M)\longrightarrow T^{*}(TM)$ has an additional property: it is an isomorphism of VB-algebroids (i.e., it is an isomorphism of DVBs which preserves the Lie algebroid structures involved), between the tangent algebroid and the Lie algebroid associated to the linear Poisson structure on $TM$ (\cite{M}, Prop. 10.3.13).
\end{remark}

\subsection{Tangent lift of algebroid cochains}\label{sec-tangent-lift}

Just as there is a tangent lift of Lie groupoid cochains, there is an infinitesimal version, allowing to lift algebroid cochains from a Lie algebroid $A$ to the algebroid $TA$. Let us recall its definition, and detail some of its properties.

\begin{definition} Let $A\to M$ be a Lie algebroid. The \textbf{tangent lift of algebroid cochains} is the map \[T:C^\bullet(A)\to C^\bullet_{\mathrm{lin}}(TA)\] defined as follows. 
An algebroid $k$-cochain $c\in C^{k}(A)=\Gamma(\Lambda^{k}A^{*})$ can be regarded as a $k$-linear and skew-symmetric map $c:\bigoplus^{k}A\to\mathbb{R}$. We define its tangent lift $Tc\in C^{k}_\mathrm{lin}(TA)$ by
$$Tc(v_1,...v_k):=dc(v_1,...,v_k),$$
where $(v_1,...,v_k)\in\bigoplus^{k}_{TM}TA$, and using the identification $T(\bigoplus^{k}A)\cong\bigoplus^{k}_{TM}TA$.
\end{definition}

\ \

\ \ 

\begin{lemma}\label{Tang. lift on algbrds}
Let $k$ be a positive integer. The tangent lift of Lie algebroid cochains $T:C^{k}(A)\longrightarrow C^{k}(TA)$ satisfies the following conditions:

\begin{enumerate}
	\item (Linear sections)
	\begin{align*}\left.Tc(T\alpha_1,...,T\alpha_k)\right|_{w_x}&=(Tc)_{w_x}\left(T\alpha_1(w_x),...,T\alpha_k(w_x)\right)=\\ &=\left[T\left(c(\alpha_1,...,\alpha_k)\right)\right](w_x),\end{align*}
	\item (One core section) $\left.Tc(T\alpha_1,...,\hat{\alpha}_k)\right|_{w_x}=\left.c(\alpha_1,...,\alpha_k)\right|_{x}$,
  \item (More than one core sections) $\left.Tc(T\alpha_1,...,\hat{\alpha}_{k-1},\hat{\alpha}_k)\right|_{w_x}=0$.
\end{enumerate}
\end{lemma}

\begin{proof}
Part 1 is direct. For simplicity, we prove the parts 2 and 3 for $k=2$, the general case $(k\neq2)$ being completely analogous.

\begin{equation*}
\begin{split}
\left.Tc(T\alpha_1,\hat{\alpha}_2)\right|_{w_x}&=(Tc)_{(w_x)}(T\alpha_1(w_x),\hat{\alpha}_2(w_x))\\
&=(Tc)_{(w_x)}\left(\overbrace{T\alpha_1(w_x)}^{\in\ T_{\alpha_1(x)}A}, T0^{A}(w_x)+_{A}(\alpha_2^{\uparrow})_{0^{A}_x}\right)\\
&=\underbrace{(Tc)_{(w_x)}\left(T\alpha_1(w_x), T0^{A}(w_x)\right)}_{=0}+\\
&+(Tc)_{(0^{A}_x)}\left(\tilde{0}^{A}_{\alpha_1(x)}, (\alpha_2^{\uparrow})_{0^{A}_x}\right)\ \ \ \ (\text{Linearity of } Tc \text{ over} \oplus^{2}_{M}A)\\
&=\left.\frac{d}{d\l}\right|_{\l=0}c(\alpha_1(x), \l\alpha_2(x))\\
&=\left.c(\alpha_1,\alpha_2)\right|_{(x)}\ \ \ \ (\text{Bilinearity of } c\ (\text{over } M)),
\end{split}
\end{equation*}
where in the third equality the first term vanishes by multilinearity of $Tc$ with respect to the vector bundle $TA\longrightarrow TM$.

Finally, considering more than one core section:

\begin{equation*}
\begin{split}
\left.Tc(\hat{\alpha}_1,\hat{\alpha}_2)\right|_{w_x}&=Tc(T0^{A}(w_x)+_{A}(\alpha_1^{\uparrow})_{0^{A}_x}, T0^{A}(w_x)+_{A}(\alpha_2^{\uparrow})_{0^{A}_x})\\
&=Tc(T0^{A}(w_x), T0^{A}(w_x))+Tc((\alpha_1^{\uparrow})_{0^{A}_x},(\alpha_2^{\uparrow})_{0^{A}_x})\\
&=0+\left.\frac{d}{d\l}\right|_{\l=0}c(\l\alpha_1(x),\l\alpha_2(x))\\
&=\left.\frac{d}{d\l}\right|_{\l=0}\l^{2}\cdot c(\alpha_1(x),\alpha_2(x))\\
&=0,
\end{split}
\end{equation*}
where in the third equality $Tc(T0^{A}(w_x), T0^{A}(w_x))=0$ by multilinearity of $Tc$ with respect to the vector bundle $TA\longrightarrow TM$.

To extend the proof to the case $k\neq2$, one again uses the linearity of $Tc$ to get a sum of a vanishing term with a simpler expression in $c$ and sections of $A_\calG$. \end{proof}

\ \

\begin{lemma}\label{tangent lift of algebroid-cochains}
Let $A$ be a Lie algebroid over $M$. The tangent lift is a cochain complex map
$$T:C^{\bullet}(A)\longrightarrow C^{\bullet}_\mathrm{lin}(TA)\subset C^{\bullet}(TA).$$
\end{lemma}

\begin{proof}
We divide this proof in three cases, by evaluating the cochain $\delta(Tc)$ on tangent and core sections of $TA\to TM$, which together span all sections of the tangent algebroid (example \ref{Tangent}).

First we remark some useful facts about the anchor $\rho_{TA}$ of $TA$ and the image of core and tangent sections by the anchor. Recall that $\rho_{TA}=J_M\circ d\rho_A$. Then, on core sections
$$\rho_{TA}(\hat{\alpha})=(\rho(\alpha))^{\uparrow},$$
where for $X\in\mathfrak{X}(M)$, $X^{\uparrow}\in\mathfrak{X}(TM)$ denotes the vertical vector field on $TM$ induced by $X$. This follows from the facts that the involution map $J_M$ identifies $T_{M}(TM)$ with the vertical bundle $V(TM)$, and it is the identity on $V_{M}(TM)$. On tangent sections,
$$\rho_{TA}(T\alpha)=\rho_A(\alpha)^{T},$$
where for $X\in\mathfrak{X}(M)$, $X^{T}\in\mathfrak{X}(TM)$ is the tangent lift of $X$. Recall also the compatibility between the tangent lifts of vector fields and functions: $X^{T}$ applied to the tangent lift $Tf$ of $f\in C^{\infty}(M)$ is the tangent lift of the function $X(f)$.

Let now $c\in C^{k}(A)$ be an algebroid $k$-cochain.

\begin{itemize}[leftmargin=*]

\item \emph{Tangent sections}\\
The compatibility between the tangent lifts of vector fields and functions allows to check directly that $\delta(Tc)=T(\delta c)$ when applied to $k+1$ tangent sections $T\alpha_i$ of $TA$.\\

\item \emph{One core section}\\
Recall that, by the second statement of Lemma \ref{Tang. lift on algbrds}, \[Tc(T\alpha_1,...,\hat{\alpha}_k)={p_M}^{*}(c(\alpha_1,...,\alpha_k)),\] thus on the one hand $T(\delta c)(T\alpha_1,...,\hat{\alpha}_{k+1})={p_M}^{*}(c(\alpha_1,...,\alpha_{k+1}))$. On the other hand,
\begin{align*}
\delta(Tc)(T\alpha_1,...,\hat{\alpha}_{k+1})&=\Sigma_{i<j}(-1)^{i+j-1}{p_M}^{*}(c([\alpha_i,\alpha_j],\alpha_1,...,\alpha_{k+1}))\\
&+\Sigma_{i}^{k}(-1)^{i}\rho_{TA}(T\alpha_i)({p_M}^{*}c(\alpha_1,...,\alpha_{i-1},\alpha_{i+1},...,\alpha_{k+1}))\\
&+(-1)^{k+1}\rho_{TA}(\hat{\alpha_{k+1}})(T(c(\alpha_1,...\alpha_k))).
\end{align*}
Now, since in the second and third row of this equation we have the tangent and vertical lifts of vector fields of $M$, we use the expression for their flows and that will allow us to prove the equality with $T(\delta c)(T\alpha_1,...,\hat{\alpha}_{k+1})$.\\

\item \emph{More than one core section}\\
On the one hand, statement (3) of Lemma \ref{Tang. lift on algbrds} says that \[T(\delta c)(\alpha_1,...,\hat{\alpha}_{k},\hat{\alpha}_{k+1})=0.\]
On the other hand, the same statement (3) implies that
\begin{align*} 
\delta(Tc)(\alpha_1,...,\hat{\alpha}_{k},\hat{\alpha}_{k+1})&=(-1)^{k}\rho_{TA}(\hat{\alpha}_k)Tc(\alpha_1,...,\alpha_{k-1},\hat{\alpha}_{k+1})\\
&+(-1)^{k+1}\rho_{TA}(\hat{\alpha}_{k+1})Tc(\alpha_1,...,\hat{\alpha}_{k})\\
&=(-1)^{k}(\rho_{A}(\alpha_k))^{\uparrow}{p_M}^{*}c(\alpha_1,...,\alpha_{k-1},\alpha_{k+1})\\
&+(-1)^{k+1}(\rho_{A}(\alpha_{k+1}))^{\uparrow}{p_M}^{*}c(\alpha_1,...,\alpha_{k})\\
&=0.\qedhere
\end{align*}
\end{itemize}
\end{proof}

%\singlespace
\bibliographystyle{plain} 
\bibliography{bibliografia2}

\end{document}